\documentclass[12pt,a4paper,fleqn]{article}
\pdfoutput=1
\input{packages}
\usepackage{xpatch}
\makeatletter
\xpatchcmd{\proof}{\topsep6\p@\@plus6\p@\relax}{}{}{}
\makeatother
\newcommand{\lra}{\longrightarrow} 
\newcommand{\eps}{\varepsilon}
\newcommand{\teststat}[1][\Di]{S_{{#1}, \alpha_{#1}}}
\newcommand{\test}{\phi_{\Di, \alpha_{\Di}}}
\newcommand{\maxtest}{\phi_{\cK, \alpha}}
\newcommand{\maxteststat}{S_{\cK, \alpha}}
\renewcommand{\subset}{\subseteq}
\newcommand{\IN}{\mathbb{N}}

\newcommand{\IR}{\mathbb{R}}
\newcommand{\lcb}{\left\lbrace} 
\newcommand{\rcb}{\right\rbrace} 
\newcommand{\lv}{\left\vert} 
\newcommand{\rv}{\right\vert} 
\newcommand{\lV}{\left\Vert} 
\newcommand{\rV}{\right\Vert} 
\newcommand{\lb}{\left(} 
\newcommand{\rb}{\right)} 
\newcommand*{\mc}[1]{\mathcal{#1}}
\newcommand{\dif}{\text{d}}
\DeclareMathOperator*{\argmin}{arg\,min}

\declaretheoremstyle[
    spaceabove=10pt, 
    spacebelow=6pt, 
    headfont=\color{\colre}\normalfont\bfseries,
    notefont=\mdseries\bfseries, 
    notebraces={(}{)}, 
    bodyfont=\normalfont\itshape,
    postheadspace=.3em,
    headpunct={.}]{restyle}

\declaretheoremstyle[
    spaceabove=8pt, 
    spacebelow=8pt, 
    headfont=\color{\colrem}\normalfont\bfseries,
    notefont=\mdseries\bfseries, 
    notebraces={(}{)}, 
    bodyfont=\normalfont\itshape,
    postheadspace=.3em,
    qed=\smaller$\color{\colrem}\square$, 
    headpunct={.}]{remstyle}

\declaretheoremstyle[
    spaceabove=8pt, 
    spacebelow=8pt, 
    headfont=\color{\colil}\normalfont\bfseries,
    notefont=\mdseries\bfseries, 
    notebraces={(}{)}, 
    bodyfont=\normalfont\itshape,
    postheadspace=.3em,
    qed=\smaller$\color{\colil}\square$, 
    headpunct={.}]{ilstyle}

\declaretheoremstyle[
    spaceabove=8pt, 
    spacebelow=8pt, 
    headfont=\color{\colas}\normalfont\bfseries,
    notefont=\mdseries\bfseries, 
    notebraces={(}{)}, 
    bodyfont=\normalfont\itshape,
    postheadspace=.3em,
    qed=\smaller$\color{\colil}\square$, 
    headpunct={.}]{destyle}

\declaretheorem[name=Theorem, style=restyle, numberwithin=section]{theorem}

\declaretheorem[name=Lemma, style=restyle, numberlike=theorem]{lemma}
\declaretheorem[name=Proposition, style=restyle, numberlike=theorem]{proposition}
\declaretheorem[name=Corollary, style=restyle, numberlike=theorem]{corollary}
\declaretheorem[name=Remark, style=remstyle, numberlike=theorem]{remark}
\declaretheorem[name=Assumption, style=destyle, numbered=no]{assumption}
\declaretheorem[name=Illustration, style=ilstyle, numberlike=theorem]{illustration}
\usepackage{WNIPMAT-abrev-package}

\makeatletter
\newcommand{\mylabel}[2]{#2\def\@currentlabel{#2}\label{#1}}
\makeatother

\makeatletter%
\def\@fnsymbol#1{\ensuremath{\ifcase#1\or * \or 1 \or 2 \or 3 \or  *\or  \star \or 4\or  , \or 
g\or h\or i\else\@ctrerr\fi}}%
\makeatother%

\author{{\sc Sandra Schluttenhofer}\;\thanks{Institut f\"ur Angewandte
    Mathematik, M$\Lambda$THEM$\Lambda$TIKON, Im Neuenheimer Feld 205,
  D-69120 Heidelberg, Germany, e-mail:
  \url{{schluttenhofer|johannes}@math.uni-heidelberg.de}} \and {\sc Jan Johannes}$\;^*$}
\date{Ruprecht-Karls-Universität Heidelberg} 
\title{Adaptive minimax testing in inverse Gaussian sequence space models} 
\begin{document}
\maketitle
\begin{abstract}
  In the inverse Gaussian sequence space model with additional noisy
  observations of the operator, we derive nonasymptotic minimax radii
  of testing for ellipsoid-type alternatives simultaneously for both
  the signal detection problem (testing against zero) and the
  goodness-of-fit testing problem (testing against a prescribed
  sequence) without any regularity assumption on the null
  hypothesis. The radii are the maximum of two terms, each of which
  only depends on one of the noise levels. Interestingly, the term
  involving the noise level of the operator explicitly depends on the
  null hypothesis and vanishes in the signal detection case.

  The minimax radii are established by first showing a lower bound for
  arbitrary null hypotheses and noise levels. For the upper bound we
  consider two testing procedures, a direct test based on estimating
  the energy in the image space and an indirect test. Under mild
  assumptions, we prove that the testing radius of the indirect test
  achieves the lower bound, which shows the minimax optimality of the
  radius and the test. We highlight the assumptions under which the
  direct test also performs optimally.  Furthermore, we apply a
  classical Bonferroni method for making both the indirect and the
  direct test adaptive with respect to the regularity of the
  alternative. The radii of the adaptive tests are deteriorated by an
  additional log-factor, which we show to be unavoidable. The results
  are illustrated considering Sobolev spaces and mildly or severely
  ill-posed inverse problems.
\end{abstract}
{\footnotesize
\begin{tabbing} 
\> unknown operator,  aggregation, adaptation,  gooodness-of-fit,
signal detection\\[.2ex] 
\end{tabbing}}%
%
%
%
%
%
\section{Introduction}
\paragraph{The statistical model.} We consider an inverse Gaussian
sequence space model with heteroscedastic errors and unknown operator
\begin{equation}\label{model}
  \yOb[j]  \sim \nVg[{\Ev[j]\So[j], \nlIm[j]}] \quad\text{and}
  \quad \xOb[j]  \sim \nVg[{\Ev[j],  \nlOp[j]}], \quad j \in \IN,
\end{equation}
where $\EvS:=\Nsuite[j]{\Ev[j]} \in \lp^\infty$ is an unknown bounded
sequence, $\SoS:=\Nsuite[j]{\So[j]} \in \lp^2$ is an unknown square
summable sequence,
$\nlImS:=\Nsuite[j]{\nlIm[j]} \in \pRz^\Nz$ and
$\nlOpS:=\Nsuite[j]{\nlOp[j]} \in \pRz^\Nz$ are known sequences of
positive real numbers, called noise levels. The sequences
$\yObS:=\Nsuite[j]{\yOb[j]}$ and $\xObS:=\Nsuite[j]{\xOb[j]}$ are
assumed to be independent with independent Gaussian components, we
denote their respective distributions by
$\yObS\sim \FuVg[\nlImS]{\EvS\SoS}$ and
$\xObS\sim \FuVg[\nlOpS]{\EvS}$ and their joint distribution by
$(\yObS,\xObS) \sim \FuVg[\nlImS,\nlOpS]{\SoS,\EvS}$.  For a given
$\oSoS \in \lp^2$ we want to test the
null hypothesis $\lcb \SoS = \oSoS\rcb$ against the alternative
$\lcb \SoS \ne \oSoS \rcb$ based on the observation $(\yObS,\xObS)$,
where $\EvS \in \lp^\infty$ is a nuisance parameter and optimality is measured in a minimax sense.

Model \eqref{model} is an idealised formulation of a statistical
inverse problem with unknown operator, where a signal $\SoS$
transformed by a multiplication with the unknown sequence $\EvS$ is
observed. In the particular case $\EvS=\Nsuite{1}$, the model is called
direct, otherwise inverse, and ill-posed if additionally $\EvS$ tends
to zero. For inverse problems with fully known operator (corresponding
to known $\EvS$), we refer to \cite{JohnstoneSilverman1990},
\cite{MairRuymgaart1996}, \cite{MathePereverzev2001},
\cite{CavalierTsybakov2002},
\cite{CavalierGolubevPicardTsybakov2002}, and the references therein.
\cite{IngsterSapatinasSuslina2012a} describe typical examples, where
the inverse Gaussian sequence space model with known $\EvS$ arises
naturally, one of which is deconvolution (\cite{Ermakov1990a};
\cite{Fan1991}; \cite{StefanskiCarroll1990}). In \eqref{model} the
sequence $\EvS$ is unknown, but an additional noisy observation of it
is available.  \cite{CavalierHengartner2005},
\cite{IngsterSapatinasSuslina2012}, \cite{JohannesSchwarz2013} or
\cite{MarteauSapatinas2017}, for instance, provide a detailed
discussion and motivation of this particular statistical inverse
problem with unknown operator. An example is density deconvolution
with unknown error distribution (cf. \cite{ComteLacour2011},
\cite{Efromovich1997} or \cite{Neumann1997}).  Oracle or minimax
optimal nonparametric estimation and adaptation in the framework of
inverse problems has been extensively studied in the literature (see
\cite{EfromovichKoltchinskii2001},
\cite{CavalierGolubevLepskiTsybakov2003}, \cite{Cavalier2008} and
\cite{HoffmannReiss2008}, to name but a few).
\paragraph{The testing task.}Coming back to the nonparametric testing
task, one usually introduces an \textit{energy condition}
$\SoS - \oSoS \in \lp[\rho]^2:= \set{a_{\mbullet} \in \lp^2 :
  \Vnormlp{a_{\mbullet}} \geq \rho}$ for a separation radius
$\rho \in \IR_+$ in order to make the null hypothesis and the
alternative distinguishable.  Additionally, \textit{regularity
  conditions} are imposed on the unknown sequences $\SoS$ and $\EvS$
by introducing nonparametric classes of parameters
$\cSo \subset \lp^2$ and $\cEv \subset \lp^\infty$. We define these
classes below such that they are flexible enough to capture typical
smoothness and ill-posedness assumptions. Summarising we consider the
testing task
\begin{align}
\label{testing:e1}
  H_0: \SoS = \oSoS, \EvS \in \cEv \qquad
  \text{ against } \qquad H_1^\rho: \SoS - \oSoS \in \lp[\rho]^2 \cap \cSo, \EvS\in\cEv.
\end{align}
Roughly speaking, in minimax testing one searches for the smallest $\rho$ such that
\eqref{testing:e1} is still testable with small error probabilities. In the literature there exist several
definitions of rates and radii of testing in an asymptotic and
nonasymptotic sense.  The classical definition of an asymptotic rate
of testing for nonparametric alternatives was essentially introduced
in the series of papers  \cite{Ingster1993},
\cite{Ingster1993a} and \cite{Ingster1993b}. For fixed noise levels, there exist two alternative definitions of a nonasymptotic radius of testing. For prescribed error
probabilities $\alpha,\beta\in(0,1)$, \cite{Baraud2002}, \cite{LaurentLoubesMarteau2012} and
\cite{MarteauSapatinas2017}, amongst others, define a nonasymptotic radius of testing as the smallest separation radius
$\rho$ such that there is an $\alpha$-test with maximal type II error
probability over the $\rho$-separated alternative smaller than $\beta$. \cite{MarteauSapatinas2015}, for
example, provide a unified treatment of asymptotic minimax rates and
nonasymptotic minimax radii of testing.  Following
e.g. \cite{CollierCommingesTsybakov2017}, 
in this paper we measure the
accuracy of a test by its maximal
risk defined as the sum of the maximal type I and II error probability
over the null hypothesis and the $\rho$-separated alternative,
respectively,
\begin{multline*}
	\RiT{\teF}{\cSo,\cEv}{\oSoS,\rho}
	:= \sup\set{ \FuVg[\nlImS,\nlOpS]{\oSoS, \EvS}(\teF = 1):\EvS \in
		\cEv}\\
	+ \sup\set{\FuVg[\nlImS,\nlOpS]{\SoS, \EvS}(\teF = 0):
		\SoS - \oSoS \in \lp[\rho]^2 \cap \cSo, \EvS\in\cEv}
\end{multline*}
and compare it to the minimax risk for the testing task
\eqref{testing:e1}
\begin{align*}
	\mRiT{\cSo,\cEv}{\oSoS,\rho}  
	:= \inf_{\teF} \RiT{\teF}{\cSo,\cEv}{\oSoS,\rho},
\end{align*}
where the infimum is taken over all possible tests, i.e. over all measurable functions $\varphi: \IR^\IN \times \IR^\IN \lra \lcb 0, 1 \rcb$. A separation
radius
$\dRai[\nlImS,\nlOpS] := \dRai[\nlImS,\nlOpS](\cSo, \cEv, \oSoS)$ is
called minimax radius of testing, if for all $\alpha \in (0,1)$ there
exist constants $\uacst{\alpha}$, $\oacst{\alpha}$ $\in \pRz$
with 
\begin{enumerate}
\item[(i)] for all $A \in [\oacst{\alpha}, \infty):$
  $\mRiT{\cSo,\cEv}{\oSoS,A\dRai[\nlImS,\nlOpS]} \leq \alpha$; \hfill
  (upper bound)
\item[(ii)] for all $A \in [0, \uacst{\alpha}]:$
  $\mRiT{\cSo,\cEv}{\oSoS,A\dRai[\nlImS,\nlOpS]} \geq 1 - \alpha.$
  \hfill (lower bound)
\end{enumerate}
Note that this definition of the minimax radius of testing is entirely
nonasymptotic. However, in our illustrations we compare our findings
to existing asymptotic results by considering the homoscedastic
case, i.e., constant noise levels $\nlImS=\Nsuite{\nlIm}$ and
$\nlOpS=\Nsuite{\nlOp}$ with $\nlIm,\nlOp\in\pRz$, and the behaviour of the radii of
testing for $\nlIm$ and $\nlOp$ tending to zero.
\paragraph{Related literature.}
Minimax testing for the direct homoscedastic version of the model
\eqref{model}, i.e. $\EvS = \Nsuite{1}$, $\nlOpS = \Nsuite{0}$ and
$\nlImS = \Nsuite{\nlIm}$, has been studied extensively in the
literature for various classes of alternatives. Asymptotic results and
a list of references can be found in the book by
\cite{IngsterSuslina2012}. Let us briefly mention some further
references. \cite{LepskiSpokoiny1999} derive asymptotic minimax rates for
Besov-type alternatives. Following this result, \cite{Spokoiny1996}
considers adaptive testing strategies, showing that asymptotic
adaptation comes with the unavoidable cost of a
$\log$-factor.  Introducing the nonasymptotic framework for minimax
testing, \cite{Baraud2002} derives matching upper and lower bounds in
the direct model for ellipsoid-type
alternatives. \cite{CollierCommingesTsybakov2017} provide similar
results for sparse alternatives, using tests based on minimax-optimal
estimators of the squared norm of  the parameter of
interest. \cite{CarpentierVerzelen2019} derive minimax radii of
testing for composite (null) hypotheses, which explicitly depend on
the complexity of the null hypothesis. Both phenomena -- an estimator
of the squared norm yields a minimax optimal test and minimax radii
depend on the null hypothesis -- reappear in our results.

In the inverse problem setting with fully known operator and
homoscedastic errors, i.e. $\nlOpS = \Nsuite{0}$ and
$\nlImS = \Nsuite{\nlIm}$, asymptotic rates over ellipsoids $\cSo$ are
derived in \cite{IngsterSapatinasSuslina2012}. Simultaneously,
\cite{LaurentLoubesMarteau2012} establish the corresponding
nonasymptotic radii.  Moreover, \cite{LaurentLoubesMarteau2011}
compare direct and indirect testing approaches, i.e. based on the
estimation of $\lV \Ev \lb \SoS - \oSoS \rb \rV_{\ell^2}^2$
respectively of $\lV \SoS - \oSoS \rV_{\ell^2}^2$, concluding that the
direct approach is preferable (under certain assumptions), since it
achieves the minimax radius without requiring an inversion.

Let us now return to the testing task \eqref{testing:e1} in the model
with unknown operator. In this situation there is a natural
distinction between the cases $\oSoS = \nSoS:=\Nsuite[j]{0}$
(\textit{signal detection}) and $\oSoS \ne \nSoS$
(\textit{goodness-of-fit}) on which we comment further below.
\cite{MarteauSapatinas2017} additionally impose on the null hypothesis
an abstract smoothness condition $\oSoS \in \cSo$ and thereby obtain
radii depending on $\cSo$ rather than on the given null hypothesis
$\oSoS$. Let us emphasise that we instead seek radii for a given
$\oSoS$ for the testing problem \eqref{testing:e1}, which typically
are much smaller than the uniform ones. Restricting themselves to the
goodness-of-fit ($\oSoS \ne \nSoS$) testing task in the homoscedastic
setting, \cite{MarteauSapatinas2017} derive upper and lower bounds for
the uniform radii, featuring a logarithmic gap. Treating the signal
detection task and the goodness-of-fit testing task separately,
\cite{Kroll2019} establishes matching upper and lower bounds for the
minimax radii of testing uniformly over null hypotheses in $\cSo$.
\paragraph{Minimax results.} In this paper we derive nonasymptotic
minimax radii of testing in the inverse Gaussian sequence space model
\eqref{model} for ellipsoid-type alternatives $\cSo$ simultaneously
for both the signal detection ($\oSoS = \nSoS$) and the
goodness-of-fit testing problem ($\oSoS \neq \nSoS$) without any
regularity assumption on the null hypothesis $\oSoS$. For known
operators ($\nlOpS = \nSoS$) there is typically no distinction between
the goodness-of-fit and signal detection task. Minimax results for the
goodness-of-fit testing task can be obtained from the signal detection
task by simply shifting the observations, i.e. considering the
sequence $\yObS - \EvS \oSoS$ instead of $\yObS$. Obviously, this is
no longer possible if $\EvS$ is unknown, which motivates their
separate treatment in \cite{MarteauSapatinas2017} and
\cite{Kroll2019}. In contrast, the reparametrisation $(\tyObS,\xObS)$
with $\tyObS := \yObS- \oSoS\xObS$ of the model \eqref{model} allows
us to deal with the signal detection problem and the goodness-of-fit
problem simultaneously.  The components of
$\tyObS = \Nsuite[j]{\tyOb[j]}$ are still independent and follow a
normal distribution
$\tyOb[j]= \yOb[j] - \oSo[j]\xOb[j] \sim \nVg[{\Ev[j](\So[j]-\oSo[j]),
  \sonlIm[j]}]$ with noise level
$\sonlIm[j]:= \snlIm[j]+ \oSo[j]^2\snlOp[j]$. The reparametrisation
already indicates that $\oSoS \nlOpS$ is the \textit{effective} noise
level instead of the \textit{original} noise level $\nlOpS$. In the
following, the minimax radii will first be derived in terms of the
\textit{reparametrised} noise level $\onlImS$ and then expressed as the maximum
of two terms, each of which only depends on one of the noise levels
$\nlImS$ and $\oSoS \nlOpS$. We shall stress that thereby the
dependence of the minimax radius on the null hypothesis $\oSoS$ is
explicit. In particular, this shows that the $\nlOpS$-term in the radius
vanishes in the signal detection task ($\oSoS = \nSoS$). Furthermore,
for $\nlOpS = \nSoS$ we recover the minimax radii for known operator,
which consequently do not depend on the null hypothesis $\oSoS$. Using
the reparametrised observation $(\tyObS, \xObS)$, we propose an
\textit{indirect} test based on the estimation of a squared weighted
$\lp^2$-norm of $\oSoS - \SoS$. More precisely, we use an estimator
that mimics an inversion of $\EvS$ by using the class $\cEv$ and aims
to estimate the quadratic functional
$\qFPr{ \oSoS - \SoS} := \sum_{j=1}^{\Di} \lb \oSo[j] - \So[j]
\rb^2$. If $\Di$ is chosen appropriately, the test attains the minimax
radius given by a classical trade-off between the variance of the
quadratic functional and a bias term.  To avoid the inversion, we
investigate a \textit{direct} testing procedure inspired by
\cite{LaurentLoubesMarteau2011}, that is based on the estimation of
the squared $\lp^2$-norm of $\EvS \lb \oSoS - \SoS \rb$. We show its
minimax optimality for the corresponding direct testing task, i.e. for
testing the null hypothesis $\lcb \EvS \oSoS = \EvS \SoS \rcb$ against
the alternative $\lcb \EvS \oSoS \neq \EvS \SoS \rcb$. In contrast to
inverse problems with known operator, we show that the direct approach is not always preferable if the operator is unknown, but characterise
situations in which it is. In particular in signal detection the direct test
achieves the minimax radius under very mild assumptions. However, its
advantage over the indirect test is that it only implicitly depends on the
knowledge of the model's ill-posedness characterised by the class
$\cEv$ via an optimal choice of the dimension parameter $\Di$.
\paragraph{Adaptation.} For both testing procedures the optimal choice
of the dimension parameter $\Di$ relies on the knowledge of
characteristics of the classes $\cSo$ and $\cEv$. A classical
procedure to circumvent this problem is to aggregate several tests for
various dimension parameters $\Di$ into a \textit{maximum}-test, which
rejects the null hypothesis as soon as one of the tests does. We apply
this aggregation to both testing procedures and derive the radii of
testing of their corresponding $\max$-tests. Thereby, the indirect
$\max$-test is adaptive (i.e. assumption-free) with respect to the
smoothness of $\SoS$ characterised by a family of
$\cSo$-alternatives. Comparing its radius to the minimax radius, there
is a deterioration. Heuristically, the adaptive radius is obtained by
magnifying the error level in the minimax radius by an
\textit{adaptive factor} (cp. \cite{Spokoiny1996}). Depending on the
complexity of the families of $\cSo$-alternatives, we show that
adaptive factors of $\log \log$- or even of $\log \log \log$-order are
possible. Moreover, we derive a lower bound, which shows that these
\textit{adaptive factors} are unavoidable. The indirect $\max$-test is
still only adaptive with respect to the smoothness of $\SoS$, but
explicitly depends on the model's ill-posedness characterised by
$\cEv$. In contrast, the direct $\max$-test is adaptive with respect
to both smoothness and ill-posedness. Again its radius features an
adaptive factor. Interestingly, also adaptation with respect to the
ill-posedness of $\EvS$ only results in a $\log \log$-loss, which we
show to be unavoidable.

\paragraph{Outline of the paper.} The paper is organised as
follows. In \cref{minimax:radii:of:testing} the minimax radii of
testing are derived by first establishing a lower bound
(\cref{subsec:lower}) and then a matching upper bound
(\cref{subsec:upper}) via an indirect testing procedure. In
\cref{subsec:upper:direct} we investigate a direct testing procedure.
\cref{sec:adaptation} is devoted to adaptive testing.  Technical
results and their proofs are deferred to the \cref{appendix}.


%
%
%
%
\section{Minimax Radii of Testing}
\label{minimax:radii:of:testing}
\paragraph{Notation and definitions.} For sequences
$a_{\mbullet}= \Nsuite[j]{a_j}$ and $b_{\mbullet}=\Nsuite[j]{b_j}$ in
$\Rz^\Nz$ operations and inequalities are defined component-wise,
i.e. $a_{\mbullet}^2 = \Nsuite[j]{a_j^2}$,
$a_{\mbullet}b_{\mbullet} = \Nsuite[j]{a_jb_j}$,
$a_{\mbullet}\vee b_{\mbullet}=\Nsuite[j]{a_j\vee
  b_j:=\max(a_j,b_j)}$,
$a_{\mbullet}\wedge b_{\mbullet}=\Nsuite[j]{a_j\wedge
  b_j:=\min(a_j,b_j)}$ or $a_{\mbullet}\leq xb_{\mbullet}$ with
$x\in\pRz$, if $a_j\leq x b_j$ for all $j\in\Nz$. If $a_{\mbullet}$
attains a minimum on a subset $\cK\subset\Nz$, we write
$\min_{\cK}(a_{\mbullet}):=\min\Nset[j\in\cK]{a_j}$ and
$\argmin\nolimits_{\cK}(a_{\mbullet}):=\min\{n\in\cK:a_n\leq
a_j,\;\forall j\in\cK\}$, where we suppress the index in the case
$\cK=\Nz$. For $\Di\in\Nz$ we denote
$\nset{\Di}:=[1,\Di]\cap \Nz$. Further, we define monotonically
nondecreasing sequences
$\qFPrS{a_{\mbullet}}=\Nsuite[\Di]{\qFPr{a_{\mbullet}}}$ and
$\mFPrS{a_{\mbullet}}=\Nsuite[\Di]{\mFPr{a_{\mbullet}}}$ in
$\Rz^{\Nz}$ with $\qFPr{a_{\mbullet}} := \sum_{j\in\nset{\Di}}a_j^2$
and $\mFPr{a_{\mbullet}} := \max_{\nset{\Di}}(a_{\mbullet})$ for
$\Di\in\Nz$ and set
$\rqFPrS{a_{\mbullet}}:=(\qFPrS{a_{\mbullet}})^{1/2}\in\pRz^{\Nz}$.\linebreak
For $\SoS\in\lp^2$ define the nonincreasing sequence of bias terms
$\sbFSoS=\Nsuite[\Di]{\sbFSo}:=\Vnormlp{\SoS}^2-\qFPrSoS\in\pRz^{\Nz}$,
i.e., $\sbFSo=\Vnormlp{\SoS}^2-\qFPrSo\geq0$ for $\Di\in\Nz$, where
$\lim_{\Di\to\infty}\sbFSo=0$ for all $\SoS\in\lp^2$.  With this
notation, for $\wSoS,\wEvS \in\pRz^\Nz$ and $\rSo,\rEv \in \pRz$ with
$\rEv\geq1$, we introduce nonparametric classes
\begin{multline*}
  \rwcSo:= \set{\SoS \in \lp^2: \sbFSoS \leq \rSo\swSoS}\subset\lp^2\quad\text{and}\\
  \rwcEv:=\set{\EvS\in\lp^\infty: \EvS^2\leq\rEv\wEvS^2\wedge
    \wEvS^2\leq \rEv\EvS^2}\subset\lp^\infty
\end{multline*}
for the parameters $\SoS$ and $\EvS$, respectively. Here and
subsequently, we impose the following minimal regularity conditions.
\begin{assumption}\label{min:ass}The sequences $\wSoS,\wEvS \in\pRz^\Nz$ are  
strictly  positive and monotonically nonincreasing
 with $\VnormInf{\wSoS},\VnormInf{\wEvS}\leq1$.
\end{assumption}
Let us emphasise that under the minimal regularity assumption $\EvS>\nSoS$ holds for all
$\EvS\in\rwcEv$ and hence the parameter $\SoS$ is
identifiable in the model \eqref{model}.  For a sequence
$x_{\mbullet} \in \IR^\IN$ let us define the following minimum and minimiser,
respectively,
\begin{align}\label{indirect:radius}
  \sdRaCi[\wSoS,\wEvS]{x_{\mbullet}}:=
  \min(\rqFPrS{x_{\mbullet}^2/\wEvS^2}\vee\swSoS) \quad
  \text{and} \quad  \dDiCi[\wSoS,\wEvS]{x_{\mbullet}}:=
  \argmin(\rqFPrS{x_{\mbullet}^2/\wEvS^2}\vee\swSoS).
\end{align}
Throughout this section the sequences $\wSoS$ and $\wEvS$ are
arbitrary but fixed. In particular, the optimal testing procedures
explicitly exploit the prior knowledge of $\wSoS$ and $\wEvS$, i.e. the fact that the unknown parameters satisfy $\SoS - \oSoS \in\rwcSo$ and
$\EvS\in\rwcEv$ for some $\rSo,\rEv \in \pRz$. Given subsets
$\CwSo,\CwEv\subset\pRz^\Nz$ of strictly positive, monotonically
nonincreasing respectively bounded sequence, we
discuss adaptive testing strategies when $\wSoS\in\CwSo$ and
$\wEvS\in\CwEv$  in \cref{sec:adaptation} .
\subsection{Lower Bound}
\label{subsec:lower}
In this section we first prove a lower bound for the minimax radius of testing in terms of the
reparametrised noise level $ \sonlImS=\nlImS^2+\oSoS^2\nlOpS^2$.  We then
infer a lower bound in terms of the original and effective
noise level, $\nlImS$ and $\oSoS\nlOpS$, respectively. Consider
$\smRai$ as in \eqref{indirect:radius}, replacing $x_{\mbullet}$ by
$\onlImS$, which represents the lower bound proved in the next theorem,
and set $\mDii:=\dDiCi[\wSoS,\wEvS]{\onlImS}$.
\begin{theorem}\label{lowerbound} Let $\eta\in(0,1]$ satisfy
  \begin{equation}\label{assump1}
    \eta \leq
    \frac{\rqFPr[\mDii]{\sonlImS/\wEvS^2}\wedge\swSo[\mDii]}{\smRai}
    =\frac{\rqFPr[\mDii]{\sonlImS/\wEvS^2}\wedge\swSo[\mDii]}
    {\rqFPr[\mDii]{\sonlImS/\wEvS^2}\vee\swSo[\mDii]} .
  \end{equation}
  For $\alpha \in (0,1)$ define  $\uacst[2]{\alpha} := \eta
  \lb \rSo \wedge \sqrt{2\log(1+2\alpha^2)} \rb$. Then 
  \begin{align}
    \label{ig.ul.e4}
    \forall A \in [0, \uacst{\alpha} ]:
    \mRiT[\nlImS,\nlOpS]{\rwcSo, \rwcEv}{\oSoS, A
    \dRaCi[\wSoS,\wEvS]{\onlImS}}
    \geq 1- \alpha,
  \end{align}
  i.e.  $\dRaCi[\wSoS,\wEvS]{\onlImS}$ 
  is a lower bound for the {minimax radius of testing}.
\end{theorem}
\begin{proof}[Proof of \cref{lowerbound}]
  The proof is based on a classical reduction scheme. For a fixed $\Di \in \IN$, let us first
  introduce deviations from the null
  $\tSoS \in \rwcSo \cap \lp[\uacst{\alpha}\mRai]^2$ with $\tSo[j] =0$
  for $j > \Di$ (to be specified below). For each
  $\tau \in \set{\pm 1}^{\Di}$ we define $\tSoS^\tau $ by
  $\tSo[j]^\tau:= \tau_j \tSo[j]$, $j\in\nset{\Di}$, where by
  construction $\oSoS + \tSoS^\tau$ belongs to the alternative. We
  consider the uniform mixture measure over the vertices of a
  hypercube
  $\FuVg{1} := \frac{1}{2^{k}} \sum_{\tau \in \lcb \pm 1\rcb^{k}}
  \FuVg[\nlImS,\nlOpS]{\oSoS + \tSoS^\tau, \wEvS}$ and
  $\FuVg{0} := \FuVg[\nlImS,\nlOpS]{\oSoS, \wEvS}$, supported on the
  alternative and the null hypothesis, respectively. Considering the
  reparametrised observation $(\tyObS:=\yObS-\oSoS\xObS,\xObS)$ let
  $\tFuVg{0}$ and $\tFuVg{1}$ denote its joint distribution given
  $\FuVg{0}$ and $\FuVg{1}$, respectively. Obviously, their total
  variation distance satisfies
  $\text{TV}(\FuVg{1}, \FuVg{0})=\text{TV}(\tFuVg{1}, \tFuVg{0})$.
  Applying a classical reduction argument we therefore obtain
  \begin{multline}\label{reduction}
    \mRiT[\nlImS,\nlOpS]{\rwcSo, \rwcEv}{\oSoS, A \mRai} 
    \geq \inf_{\teF}\{ \FuVg{0}(\varphi=1) + \FuVg{1}(\varphi=0) \}
    = 1 - \text{TV}(\FuVg{1}, \FuVg{0})\\
    = 1-\text{TV}(\tFuVg{1}, \tFuVg{0}) 
    \geq 1 - \sqrt{\frac{\chi^2(\tFuVg{1},\tFuVg{0})}{2}},
  \end{multline}
  where the last inequality for the $\chi^2$-divergence follows
  e.g. from Lemma 2.5. and inequality (2.7) in \cite{Tsybakov2009}. Keep in mind
  that the coordinates of $(\tyObS,\xObS)$ are independent and
  normally distributed. More precisely, if
  $(\yObS,\xObS)\sim \FuVg[\nlImS,\nlOpS]{\oSoS + \tSoS^\tau,\wEvS}$
  then the $j$-th of the coordinates of $\tyObS$ is
  normally distributed with mean $\wEv[j]\tSo[j]^\tau$ and variance
  $\sonlIm[j]$, i.e.  $\tyObS\sim\FuVg[\onlImS]{\wEvS\tSoS^\tau}$.
  Since $\tyObS$ is a sufficient statistic for $\tSoS$, the
  conditional distribution of $\xObS$ given $\tyObS$ does not depend
  on $\tSoS$. Hence, the $\chi^2$-divergence between $\tFuVg{1}$ and
  $\tFuVg{0}$ equals the $\chi^2$-divergence of the mixture over the
  marginal distribution $\FuVg[\onlImS]{\wEvS\tSoS^\tau}$ of $\tyObS$.
  From \cref{lemma:adaptive_lower_bound} in the appendix it follows
  that
  \begin{align*}
    \chi^2(\tFuVg{1}, \tFuVg{0}  ) =
    \chi^2\big(\tfrac{1}{2^{k}} \sum_{\tau \in \lcb \pm 1 \rcb^{k}}
    \FuVg[\onlImS]{\wEvS\tSoS^\tau}, \FuVg[\onlImS]{\nSoS} \big)  \leq
    \exp \big( \tfrac{1}{2} \sum_{j\in\nset{\Di}}
    \tfrac{\wEv[j]^4\tSo[j]^4}{\onlIm[j]^4}
    \big) - 1 = \exp \big( \tfrac{1}{2} \qFPr{\tfrac{\wEvS^2\tSoS^2}{\sonlImS}}\big)- 1.
  \end{align*}
  For each $\alpha\in(0,1)$ the last bound together with
  \eqref{reduction}, implies the assertion \eqref{ig.ul.e4}, if for
  some $\tSoS\in\lp^2$, $\Di\in\Nz$ and $\uacst{\alpha}\in\pRz$
  both
  \begin{inparaenum}[i]\renewcommand{\theenumi}{\dgrau\rm(\alph{enumi})}
  \item\label{lowerbound.a}
    $\tSoS \in \rwcSo \cap
    \lp[\uacst{\alpha}\mRai]^2$   and \item\label{lowerbound.b}
    $ \qFPr{\wEvS^2\tSoS^2/\sonlImS} \leq 2
    \log(1+2\alpha^2)$
  \end{inparaenum}
  hold.  It remains to define these quantities: Let
  $\Di:= \mDii:=\dDiCi[\wSoS\wEvS]{\onlImS}$ and consider
  $\tSoS=\Nsuite{\tSo[j]}$ with $\tSo[j]=0$ for $j>\mDii$, and 
  \begin{equation*}
    \tSo[j] :=  \frac{\sqrt{\zeta \eta \smRai}}{\rqFPr[\mDii]{\sonlImS/\wEvS^2}}\;
    \frac{\sonlIm[j]}{\wEv[j]^2}\quad \text{for }j\in\nset{\mDii}
    \quad \text{and} \quad \zeta := \rSo \wedge \sqrt{2\log(1+2\alpha^2)}.
  \end{equation*}
  Since
  $\Vnormlp{\tSoS}^2 = \qFPr[\mDii]{\tSoS} = \zeta \eta \smRai =
  \uacst[2]{\alpha}\smRai$ with $\uacst[2]{\alpha} := \zeta\eta$ the
  parameter $\tSoS$ is separated by $\uacst{\alpha}\mRai$ from the
  null. Moreover, $\tSoS$ lies in $\rwcSo$. Indeed, keeping
  \eqref{assump1} and the definition of $\zeta$ in mind for all
  $l\in\nset{\mDii}$ we have
  $\sbF[l]{\tSoS} \leq \qFPr[\mDii]{\tSoS}=
  \uacst[2]{\alpha}\smRai\leq \zeta \swSo[\mDii] \leq \rSo \swSo[l]$,
  while $\sbF[l]{\tSoS} = 0 \leq \swSo[l]$ for each $l >
  \mDii$. Therefore, $\tSoS$ satisfies \ref{lowerbound.a}. On the other
  hand, exploiting again \eqref{assump1} and the definition of $\zeta$
  we have
  $\qFPr[\mDii]{\wEvS^2\tSoS^2/\sonlImS} =\zeta^2{( \eta
    \smRai)^2}/{(\rqFPr[\mDii]{\sonlImS/\wEvS^2})^2} \leq \zeta^2 \leq
  2\log(1+2\alpha^2)$, and thus also \ref{lowerbound.b} holds, which
  completes the proof.
\end{proof}
Note that the lower bound in \eqref{ig.ul.e4} involves the value
$\eta$ satisfying \eqref{assump1}, which depends on the joint
behaviour of the sequences $\wEvS$ and $\wSoS$ and essentially
guarantees an optimal balance of the bias and the variance term in the
dimension $\mDii$.  Next, consider $\symRaCi$ and $\sxmRaCi$ as in
\eqref{indirect:radius}, replacing $x_{\mbullet}$ by the original and
the effective noise level, $\nlImS$ and $\oSoS\nlOpS$, respectively.
The elementary inequality
$\rqFPrS{\snlImS/\wEvS^2}\vee\rqFPrS{\oSoS^2\snlOpS/\wEvS^2}\leq
\rqFPrS{\sonlImS/\wEvS^2}$ directly implies
 $\ymRaCi \vee \xmRaCi \leq \mRai$. 
 Therefore, the next
corollary is an immediate consequence of \cref{lowerbound}, and we
omit its proof.
\begin{corollary}\label{lowerbound2}
  Under the assumptions of \cref{lowerbound},
\begin{equation*}
    \forall A \in [0, \underline{A}_\alpha]: 
    \mRiT[\nlImS,\nlOpS]{\rwcSo, \rwcEv}{\oSoS, A [\ymRaCi \vee
      \xmRaCi]}
    \geq 1- \alpha,
  \end{equation*}
  i.e.  $\ymRaCi \vee \xmRaCi$ is a lower bound for the minimax radius
  of testing.
\end{corollary}

%
%
%
%
\subsection{Indirect testing procedure}\label{subsec:upper}
In this section we derive an upper bound for the minimax radius of
testing based on the estimation of the energy of the parameter of
interest $\SoS - \oSoS$. Precisely, for
$\sonlImS=\snlImS+\oSoS^2\snlOpS\in\pRz^\Nz$ consider a sequence
$\hqFS=\Nsuite[\Di]{\hqF}$, where
$\hqF:=\sum_{j\in\nset{\Di}}\wEvDi[j]^{-2}((\yOb[j]-\oSo[j]\xOb[j])^2-\sonlIm[j])$
is an unbiased estimator of the quadratic functional
$\qFPr{\tfrac{\EvS}{\wEvS}(\SoS-\oSoS)}=\sum_{j\in\nset{\Di}}
\tfrac{\Ev[j]^2}{\wEv[j]^2}(\So[j]-\oSo[j])^2$, which differs from
$\qFPr{\SoS-\oSoS}$ only by a factor $\rEv$ for all
$\EvS \in \rwcEv$ and all $\Di \in \IN$.  Our evaluation of the
performance of the test under both the null hypothesis and the alternative
relies on bounds for quantiles of (non-)central
$\chi^2$-distributions, which we present in \cref{re:qchi} in the
\cref{appendix}. Its proof is based on a result given in
\cite{Birge2001} (Lemma 8.1), which is a generalisation of Lemma 1 of
\cite{LaurentMassart2000} and can also be found with slightly
different notation in \cite{LaurentLoubesMarteau2012} (Lemma 2).    
\begin{proposition}\label{m:re:ub}For $u\in(0,1)$ set
  $\lcst{u}:=\sqrt{|\log u|}$.  Let $\alpha,\beta\in(0,1)$.  For each
  $\Di\in\Nz$ it holds
  \begin{equation}\label{m:re:ub:e1}
    \sup\set{\FuVg[\nlImS,\nlOpS]{\oSoS,\EvS}\big(\hqF
      >2\lcst{\alpha}
      \rqFPr{\sonlImS/\wEvS^2}+2\lcst[2]{\alpha}\mFPr{\sonlImS/\wEvS^2}\big),
      \EvS\in\rwcEv}
    \leq\alpha.
  \end{equation}
  Let $\mDii:=\ymDiCi\wedge\xmDiCi$ as in \eqref{indirect:radius} and
  $\cst{\alpha,\beta}:=
  5(\lcst{\alpha}+\lcst[2]{\alpha}+\lcst{\beta}+5\lcst[2]{\beta})$,
  then for each $\SoS-\oSoS\in\rwcSo\cap\lp[\rho]^2$ with
  $ \rho^2\geq (\rSo+\rEv\,\cst{\alpha,\beta})[\symRaCi\vee\sxmRaCi]$
  it holds
  \begin{equation}\label{m:re:ub:e2}
    \sup\set{\FuVg[\nlImS,\nlOpS]{\SoS,\EvS}\big(\hqF[{\mDii}]
      \leq2\lcst{\alpha}\rqFPr[{\mDii}]{\sonlImS/\wEvS^2}+
      2\lcst[2]{\alpha}\mFPr[{\mDii}]{\sonlImS/\wEvS^2}\big),\EvS\in\rwcEv}
    \leq\beta.
  \end{equation}
\end{proposition}
\begin{proof}[Proof of \cref{m:re:ub}]
  We intend to apply \cref{re:qchi} and use the notation introduced there. If
  $(\yObS,\xObS)\sim\FuVg[\nlImS,\nlOpS]{\SoS,\EvS}$, then for each
  $\Di \in \IN$,
  $Q_{\Di}:=\hqF+\qFPr{\onlImS/\wEvS}\sim\FuVgQ[\tnlImS]{\mu_{\mbullet},\Di}$
  with $\stnlImS:=\sonlImS/\wEvS^2$ and
  $\mu_{\mbullet}:=\EvS(\SoS-\oSoS)/\wEvS$.  Under the null
  hypothesis, i.e.,
  $(\yObS,\xObS)\sim\FuVg[\nlImS,\nlOpS]{\oSoS,\EvS}$, we have
  $Q_{\Di}\sim\FuVgQ[\tnlImS]{\nSoS,\Di}$ and with \eqref{qchi:e1} in
  \cref{re:qchi} it follows
  $\quVgQ[\tnlImS]{\nSoS,\Di}{\alpha}\leq \qFPr{\tnlImS}+
  2\lcst{\alpha}\rqFPr{\stnlImS}+2\lcst[2]{\alpha}\mFPr{\stnlImS}$,
  which implies \eqref{m:re:ub:e1}.  Under the alternative, i.e.,
  $(\yObS,\xObS)\sim\FuVg[\nlImS,\nlOpS]{\SoS,\EvS}$ with
  $\EvS\in\rwcEv$, $\SoS-\oSoS\in \rwcSo\cap\lp[\rho]^2$ and
  $ \rho^2\geq (\rSo+\rEv\cst{\alpha,\beta})[\symRaCi\vee\sxmRaCi]$, we obtain
  \begin{multline}\label{m:re:ub:p2}
    \Vnormlp{\SoS-\oSoS}^2 \geq \rSo\wSoDi[\mDii]^2 + \rEv
    [\rqFPr[\mDii]{\snlImS/\wEvS^2}\vee\rqFPr[\mDii]{\snlOpS\oSoS^2/\wEvS^2}]
    \cst{\alpha,\beta}\\
    \geq \rSo\wSoDi[\mDii]^2 + \rEv
    \frac{5}{2}\big(\lcst{\alpha}\rqFPr[\mDii]{\stnlImS}+\lcst[2]{\alpha}
    \mFPr[\mDii]{\stnlImS}+\rqFPr[\mDii]{\stnlImS}(\lcst{\beta}+5\lcst[2]{\beta})\big)
  \end{multline}
  using
  $\symRaCi\vee\sxmRaCi =
  \rqFPr[\mDii]{\snlImS/\wEvS^2}\vee\rqFPr[\mDii]{\snlOpS\oSoS^2/\wEvS^2}
  \vee\wSo[\mDii]^2$ due to \cref{re:argmin} and
  $2[\rqFPr[\mDii]{\snlImS/\wEvS^2}\vee\rqFPr[\mDii]{\snlOpS\oSoS^2/\wEvS^2}]
  \geq \rqFPr[\mDii]{\stnlImS}\geq\mFPr[\mDii]{\stnlImS}$. Moreover,
  for each $\Di\in\Nz$ and $\EvS\in\rwcEv$ we have
  $\rEv\qFPr{\mu_{\mbullet}}\geq\qFPr{\SoS-\oSoS}=
  \Vnormlp{\SoS-\oSoS}^2-\sbF{\SoS-\oSoS}$, which in turn for each
  $\SoS-\oSoS\in\rwcSo$ implies
  $\rEv\qFPr{\mu_{\mbullet}}\geq\Vnormlp{\SoS-\oSoS}^2-\rSo\wSoDi[\Di]^2$. This
  bound together with \eqref{m:re:ub:p2} implies
  $\tfrac{4}{5}\qFPr[\mDii]{\mu_{\mbullet}}\geq
  2\lcst{\alpha}\rqFPr[\mDii]{\stnlImS}+2\lcst[2]{\alpha}
  \mFPr[\mDii]{\stnlImS}+\rqFPr[\mDii]{\stnlImS}
  2(\lcst{\beta}+5\lcst[2]{\beta})$. Rearranging the last inequality,
  \eqref{qchi:e2} in \cref{re:qchi} implies
  \begin{equation*}
    2\lcst{\alpha}\rqFPr[\mDii]{\stnlImS}+2\lcst[2]{\alpha}\mFPr[\mDii]{\stnlImS} +
    \qFPr[\mDii]{\tnlImS} \leq \qFPr[\mDii]{\tnlImS} +
    \tfrac{4}{5}\qFPr[\mDii]{\mu_{\mbullet}} - \rqFPr[\mDii]
    {\stnlImS}2(\lcst{\beta}+5\lcst[2]{\beta})
    \leq\quVgQ[\tnlImS]{\mu_{\mbullet},\mDii}{1-\beta}
  \end{equation*} 
  and thus \eqref{m:re:ub:e2}, which completes the proof. 
\end{proof}
\paragraph{Definition.} For $\alpha\in(0,1)$ and $\Di\in\Nz$ we define
the test statistic and the test
\begin{equation}\label{de:it}
  \teSi[\Di,]{\alpha}:=\hqF-2\lcst{\alpha}
  \rqFPr{\sonlImS/\wEvS^2}-2\lcst[2]{\alpha}\mFPr{\sonlImS/\wEvS^2}
  \quad\text{ and }\quad
  \teFi[\Di,]{\alpha}:=\Ind{\set{\teSi[\Di,]{\alpha}>0}}.
\end{equation}  
Exploiting \eqref{m:re:ub:e1}, the test $\teFi[\Di,]{\alpha/2}$ defined in
\eqref{de:it} is a level-$\alpha/2$-test for any
$\Di\in\Nz$. Moreover,
\begin{equation}\label{de:it:op}
\teFi[]{\alpha/2}:=\teFi{\alpha/2}\quad\text{with
}\mDii:=\ymDiCi\wedge\xmDiCi\text{ as in \eqref{indirect:radius}}
\end{equation}  
is a $(1-\alpha/2)$-powerful test over
$\oacst{\alpha}[\ymRaCi\vee\xmRaCi]$-separated alternatives due to
\eqref{m:re:ub:e2} with $\beta=\alpha/2$ and
$\oacst[2]{\alpha}:=\rSo+\rEv(10\lcst{\alpha/2}+30\lcst[2]{\alpha/2})$. Hence,
$\RiT[\nlIm,\nlOp]{\teFi[]{\alpha/2}}{\wrcSo,\wrcEv}{\oSoS,\acst{}[\ymRaCi\vee\xmRaCi]}$
$\leq\alpha/2+\alpha/2=\alpha$ for all
$\acst{}\in[\oacst{\alpha},\infty)$. In other words,
$\ymRaCi\vee\xmRaCi$ is an upper bound for the radius of testing of
$\teFi[]{\alpha/2}$, which is summarised in the next theorem.
\begin{theorem}\label{m:re:mrt}
  For $\alpha\in(0,1)$ define
  $\oacst[2]{\alpha}:=\rSo+\rEv(10\lcst{\alpha/2}+30\lcst[2]{\alpha/2})$. Then
  \begin{equation*}
    \forall \acst{}\in[\oacst{\alpha},\infty):
    \mRiT{\wrcSo,\wrcEv}{\oSoS,\acst{}[\ymRaCi\vee\xmRaCi]}\leq\alpha,
  \end{equation*}
  i.e.  $\ymRaCi\vee\xmRaCi$ is an upper bound for the minimax radius
  of testing.
\end{theorem}
\begin{proof}[Proof of \cref{m:re:mrt}]
  The claim follows from \cref{m:re:ub} considering
  $\teFi[]{\alpha/2}$ as in \eqref{de:it:op} and the elementary bound
  \begin{multline*}
    \forall \acst{}\in[\oacst{\alpha},\infty):
    \mRiT{\wrcSo,\wrcEv}{\oSoS,\acst{}[\ymRaCi\vee\xmRaCi]}\\\leq
    \RiT{\teFi[]{\alpha/2}}{\wrcSo,\wrcEv}{\oSoS,\acst{}[\ymRaCi\vee\xmRaCi]}\leq\alpha.
  \end{multline*}
\end{proof}
The last result establishes the upper bound condition, and thus
together with the lower bound condition derived in \cref{lowerbound}
the minimax optimality of the testing radius $\ymRaCi\vee\xmRaCi$ and
hence the test $\teFi[]{\alpha/2}$.
\begin{remark}\label{m:rem:mrt1}
  Considering the signal detection task, i.e., $\oSoS=\nSoS$, we have
  $\xmRaCi=0$ for all $\nlOpS\in\pRz^\Nz$, and thus the minimax
  testing radius does not depend on the noise levels $\nlOpS$.
  Considering the goodness of fit task, i.e., $\oSoS\ne\nSoS$, for all
  $\nlImS\geq\nlOpS$ we have
  $\rqFPrS{\snlOpS\oSoS^2/\wEvS^2}\leq \rqFPrS{\snlImS\oSoS^2/\wEvS^2}
  \leq \VnormInf{\oSoS}^2\rqFPrS{\snlImS/\wEvS^2}$ and thus
  $\xmRaCi\leq\VnormInf{\oSoS} \ymRaCi$. In other words, in this
  situation $\xmRaCi$ is negligible compared to
  $\VnormInf{\oSoS}\ymRaCi$.
\end{remark}
\begin{remark}\label{m:rem:mrt2} In the homoscedastic case
  $\nlImS=\Nsuite[j]{\nlIm}$ and $\nlOpS=\Nsuite[j]{\nlOp}$ for
  $\nlIm,\nlOp\in\pRz$, we are especially interested in the behaviour
  of the radii of testing ${\yRaCi:=\ymRaCi}$ and ${\xRaCi:=\xmRaCi}$
  as $\nlIm,\nlOp\to 0$, which are then called rates of testing.  We call $\yRaCi$
  (respectively $\xRaCi$) parametric, if $ {\yRaCi/\nlIm}$ is bounded
  away from $0$ and infinity as $\nlIm\to 0$.  Since
  $\liminf_{\nlIm\to0}(\yRaCi/\nlIm)\geq\VnormInf{\wEvS}^{-2}$ and
  $\wSoS>\nSoS$, the rate becomes parametric if and only if
  $ \wEvS^{-2}\in\lp^2$.  Since $\wEvS\in\lp^\infty$, the rate
  $\yRaCi$ is always nonparametric, i.e.,
  $\liminf_{\nlIm\to0}\yRaCi/\nlIm=\infty$. On the other hand, for a
  goodness-of-fit task, the rate $\xRaCi$ is parametric if and only if
  $\oSoS^2/\wEvS^2\in\lp^2$. Note it is never faster than parametric,
  since
  $\liminf_{\nlOp\to0}(\xRaCi/\nlOp)\geq\Vnormlp{\oSoS^2}
  /\VnormInf{\wEvS}^2>0$. Finally, we shall stress that for fixed
  $\nlIm,\nlOp\in(0,1)$ there exists $\eta:=\eta(\nlIm,\nlOp)\in(0,1]$
  such that the additional assumption \eqref{assump1} is satisfied and, therefore,
  \cref{lowerbound2} establishes $\yRaCi\vee\xRaCi $ as a lower bound
  for the minimax radius of testing. If there exists an $\eta\in(0,1]$ such
  that the condition \eqref{assump1} holds uniformly as
  $\nlIm,\nlOp\to 0$, then $\yRaCi\vee\xRaCi $ is a minimax rate of
  testing.
\end{remark}
\begin{illustration}[homoscedastic case]\label{illustration:indirect}
  Throughout the paper we illustrate the order of the rates in the
  homoscedastic case $\nlImS=\Nsuite[j]{\nlIm}$ and
  $\nlOpS=\Nsuite[j]{\nlOp}$ under the following typical
  smoothness and ill-posedness assumptions. Concerning the class
  $\rwcSo$ we distinguish two behaviours of the sequence $\wSoS$,
  namely the \textbf{ordinary smooth} case:
  $\wSoS =\Nsuite[j]{ j^{-\wSoP}}$ for $\wSoP > 1/2$ where $\rwcSo$
  corresponds to a Sobolev ellipsoid, and the \textbf{super smooth}
  case: $\wSoS =\Nsuite{\exp(-j^{2\wSoP})}$ for $\wSoP > 0$, which can
  be interpreted as an analytic class of parameters. Concerning the
  class $\rwcEv$ we also distinguish two cases for the sequence
  $\wEvS$. Precisely, for $\wEvP>0$ we consider a \textbf{mildly
    ill-posed} model: $\wEvS = \Nsuite{j^{-\wEvP}}$ and a
  \textbf{severely ill-posed} model:
  $\wEvS =\Nsuite{\exp(-j^{2\wEvP})}$. Concerning the null hypothesis
  we restrict ourselves to two cases as well; the \textbf{signal
    detection task} $\oSoS=\nSoS$ and the \textbf{goodness-of-fit
    testing task} $\oSoS =\Nsuite{ j^{-\oSoP}}$ for some
  $\oSoP > 1/2$. The table displays the order of the optimal choice $\ymDiCi\wedge\xmDiCi$ for the dimension parameter as well as the order of the minimax rate
  $\ymRaCi\vee\xmRaCi$ for the signal detection task (with $\xmRaCi=0$
  as discussed in \cref{m:rem:mrt1}) and the goodness-of-fit task.
  Keep in mind that the rate $\ymRaCi$ does not depend on the null
  hypothesis, therefore it is the same for all $\oSoS \in \lp^2$.  In
  accordance with \cref{m:rem:mrt1}, $\xmRaCi$ is parametric for the
  goodness-of-fit task whenever $\oSoS^2/\wEvS^2\in\lp^2$.  Note that
  in all three cases the additional assumption \eqref{assump1} is
  satisfied uniformly in both noise levels, and hence
  $\ymRaCi\vee\xmRaCi $ is a minimax rate
  of testing due to \cref{lowerbound2} (see \cref{m:rem:mrt2}).\\[3ex]
  \centerline{\begin{tabular}{@{}ll|ll|l@{}l@{}r@{}}\toprule
      \multicolumn{7}{@{}l@{}}{Order of the minimax-optimal dimension
        $\ymDiCi\wedge\xmDiCi$ and rate
        $\ymRaCi\vee\xmRaCi$}\\\midrule
                $\wSoS$&$\wEvS$&\hfill$\ymDiCi$\hfill\null&\hfill$\ymRaCi$\hfill\null&\hfill$\xmDiCi$\hfill\null&\hfill$\xmRaCi$\hfill\null&\\
                {\smaller(smooth.)}&{\smaller(ill-pos.)}&\multicolumn{2}{|c|}{\smaller
                                                          $\oSoS \in
                                                          \lp^2$}&\multicolumn{2}{c@{}}{\smaller$\oSoS
                                                                   =\Nsuite{
                                                                   j^{-\oSoP}}$}\\\midrule
                $\Nsuite{ j^{-\wSoP}}$ &
                                         $\Nsuite{j^{-\wEvP}}$&$\nlIm^{-\frac{4}{4\wSoP+4\wEvP+1}}$&$\nlIm^{\frac{4\wSoP}{4\wSoP+4\wEvP+1}}$&
                                                                                                                                              $\begin{array}{@{}l}\nlOp^{-\frac{4}{4\wSoP+4(\wEvP-\oSoP)+1}}\\\nlOp^{-\frac{1}{\wSoP}
                                                                                                                                                 }\\\nlOp^{-\frac{1}{\wSoP}}\end{array}$&
                                                                                                                                                                                          $\begin{array}{l@{}}\nlOp^{\frac{4\wSoP}{4\wSoP+4(\wEvP-\oSoP)+1}}\\|\log\nlOp|^{\frac{1}{4}}\nlOp\\\nlOp\end{array}$&
                                                                                                                                                                                                                                                                                                                 $\begin{array}{@{}l@{}}{\scriptstyle\oSoP-\wEvP<1/4}\\{\scriptstyle\oSoP-\wEvP=1/4}\\{\scriptstyle\oSoP-\wEvP>1/4}\end{array}$\\\midrule
                $\Nsuite{e^{-j^{2\wSoP}}}$&$\Nsuite{j^{-\wEvP}}$&$|\log\nlIm|^{\frac{1}{2\wSoP}}$&$|\log
                                                                                                   \nlIm
                                                                                                   |^{\frac{4\wEvP+1}{8\wSoP}}\nlIm$&
                                                                                                                                      $\begin{array}{@{}l}|\log\nlOp|^{\frac{1}{2\wSoP}}\\|\log\nlOp|^{\frac{1}{2\wSoP}}\\|\log\nlOp|^{\frac{1}{2\wSoP}}\end{array}$&
                                                                                                                                                                                                                                                                      $\begin{array}{l}|\log\nlOp|^{\frac{4\wEvP-4\oSoP+1}{8\wSoP}}\nlOp\\(\log|\log\nlOp|)^{\frac{1}{4}}\,\nlOp\\\nlOp\end{array}$&
                                                                                                                                                                                                                                                                                                                                                                                                     $\begin{array}{l@{}}{\scriptstyle\oSoP-\wEvP<1/4}\\{\scriptstyle\oSoP-\wEvP=1/4}\\{\scriptstyle\oSoP-\wEvP>1/4}\end{array}$\\\midrule
                $\Nsuite{ j^{-\wSoP}}$&$\Nsuite{e^{-j^{2\wEvP}}}$ &
                                                                    $|\log\nlIm|^{\frac{1}{2\wEvP}}$&$|\log\nlIm|^{-\frac{\wSoP}{2\wEvP}}$
                                                                                     &
                                                                                       $
                                                                                       |\log
                                                                                       \nlOp|^{\frac{1}{2\wEvP}}$
                                                                                                                &
                                                                                                                  $|\log
                                                                                                                  \nlOp|^{-\frac{\wSoP}{2\wEvP}}$
                \\\bottomrule
\end{tabular}}\end{illustration}
\begin{remark} 
  Let us note that by applying Markov's inequality, it can be shown that the test
  $\Ind{\set{\tT_{\mDii} > 0}}$ with the simplified test statistic
  $\tT_{\mDii} := \hqF[\mDii]
  -\rqFPr[\mDii]{\sonlImS/\wEvS^2}\sqrt{2/\alpha}$ and $\mDii$ as in \eqref{de:it:op}, also attains the
  minimax radius of testing $\ymRaCi\vee\xmRaCi$. The approach of
  deriving radii of testing by applying Markov's inequality has for
  example been used in \cite{Kroll2019}.  Since we are in
  \cref{sec:adaptation} also concerned with adaptive Bonferroni
  aggregation, we need the sharper bound given in \cref{m:re:ub} for
  the threshold constant in terms of $\alpha$. This directly
  translates to the cost of adaptivity.
\end{remark}
The test $\teFi[\Di,]{\alpha}$ in \eqref{de:it} explicitly uses the
knowledge of $\wEvS$, which determines the asymptotic behaviour of the
sequence $\EvS \in \rwcEv$.  Inspired by
\cite{LaurentLoubesMarteau2011}, as an alternative we consider a
direct testing approach next.

%
%
%
%
\subsection{Direct testing procedure}\label{subsec:upper:direct}
In this section we derive an upper bound for the radius of testing
based on the estimation of the energy of the parameter
$\EvS(\SoS- \oSoS)$ instead of $\tfrac{\EvS}{\wEvS}(\SoS- \oSoS)$ as
in the section before. Precisely, consider $\tqFS=\Nsuite[\Di]{\tqF}$,
a sequence of unbiased estimators
$\tqF:=\sum_{j\in\nset{\Di}}((\yOb[j]-\oSo[j]\xOb[j])^2-\onlIm[j]^2)$ of
$\qFPr{\EvS(\SoS-\oSoS)}=\sum_{j\in\nset{\Di}}\Ev[j]^2(\So[j]-\oSo[j])^2$.
To formulate a result similar to \cref{m:re:ub}, we introduce for a
sequence $x_{\mbullet} \in \IR^\IN$ the minimum and minimiser,
respectively,
\begin{align}\label{direct:radius}
  (\mRaCd[x_{\mbullet}])^2:=\min(\wEvS^{-2}\rqFPrS{x_{\mbullet}^2}\vee\swSoS)
  \quad \text{and} \quad  \mDiCd[x_{\mbullet}]
  := \argmin(\wEvS^{-2}\rqFPrS{x_{\mbullet}^2}\vee\swSoS).
\end{align}
Replacing in \eqref{direct:radius} $x_{\mbullet}$ by the original and
the effective noise level $\nlImS$ and $\oSoS\nlOpS$ we establish
below $\ymRaCd \vee \xmRaCd$ as optimal achievable testing radius for
the direct test. Similar to \cref{m:re:ub} for the indirect test the
next result allows to evaluate the performance of the direct test
under both, the null hypothesis and the alternative.
\begin{proposition}\label{m:in:ub}
  Let $\alpha,\beta\in(0,1)$.  For each $\Di\in\Nz$ it holds
  \begin{equation}\label{m:in:ub:e1}
    \sup\set{\FuVg[\nlImS,\nlOpS]{\oSoS,\EvS}\big(\tqF
      >2\lcst{\alpha}\rqFPr{\sonlImS}+2\lcst[2]{\alpha}
      \mFPr{\sonlImS}\big),\EvS\in\rwcEv}
    \leq\alpha.
  \end{equation}
  Let $\mDid:=\ymDiCd\wedge\xmDiCd$ as in \eqref{direct:radius} and
  $\cst{\alpha,\beta}:= 5(\lcst{\alpha}+\lcst[2]{\alpha} +\lcst{\beta}
  +5\lcst[2]{\beta})$, then for each
  $\SoS-\oSoS\in\rwcSo\cap\lp[\rho]^2$ with
  $ \rho\geq
  (\rSo+\rEv\,\cst{\alpha,\beta})^{1/2}[\ymRaCd\vee\xmRaCd]$ it holds
  \begin{equation}\label{m:in:ub:e2}
    \sup\set{\FuVg[\nlImS,\nlOpS]{\SoS,\EvS}\big(\tqF[{\mDid}]
      \leq2\lcst{\alpha}\rqFPr[{\mDid}]{\sonlImS}+2\lcst[2]{\alpha}
      \mFPr[{\mDid}]{\sonlImS}\big),\EvS\in\rwcEv}
    \leq\beta.
  \end{equation}
\end{proposition}
\begin{proof}[Proof of \cref{m:in:ub}] We note that
  $(\yObS,\xObS)\sim\FuVg[\nlImS,\nlOpS]{\SoS,\EvS}$ implies
  $Q_{\Di}:=\tqF+\qFPr{\onlImS}\sim\FuVgQ[\tnlImS]{\mu_{\mbullet},\Di}$
  with $\tnlImS:=\onlImS$ and $\mu_{\mbullet}
  :=\EvS(\SoS-\oSoS)$, where we again use the notation of \cref{re:qchi} in the appendix. Therefore, the proof of
  \eqref{m:in:ub:e1} follows analogously to the proof of
  \eqref{m:re:ub:e1} in \cref{m:re:ub} by applying \cref{re:qchi}. Similar calculations as in the
  proof of \eqref{m:re:ub:p2} show that for each
  $\SoS-\oSoS\in \rwcSo\cap\lp[\rho]^2$ with
  $ \rho\geq
  \sqrt{\rSo+\rEv\cst{\alpha,\beta}}[\ymRaCd\vee\xmRaCd]$, we obtain
  \begin{equation}\label{m:in:ub:p1}
    \Vnormlp{\SoS-\oSoS}^2
    \geq \rSo\wSoDi[\mDid]^2 +
    \rEv\wEv[\mDid]^{-2} 
    \frac{5}{2}\big(\lcst{\alpha}\rqFPr[\mDid]{\stnlImS}+\lcst[2]{\alpha}
    \mFPr[\mDid]{\stnlImS}+\rqFPr[\mDid]{\stnlImS}(\lcst{\beta}+5\lcst[2]{\beta})\big)
  \end{equation}   
  using
  $(\ymRaCd\vee\xmRaCd)^2=\wEv[\mDid]^{-2}\rqFPr[\mDid]{\snlImS}
  \vee\wEv[\mDid]^{-2}\rqFPr[\mDid]{\snlOpS\oSoS^2} \vee\swSo[\mDid]$
  due to \cref{re:argmin} and
  $2[\rqFPr[\mDid]{\snlImS}\vee\rqFPr[\mDid]{\snlOpS\oSoS^2}] \geq
  \rqFPr[\mDid]{\stnlImS}\geq \mFPr[\mDid]{\stnlImS}$. Moreover, for
  each $\Di\in\Nz$ and $\EvS\in\rwcEv$ we have
  $\rEv\wEv^{-2}\qFPr{\mu_{\mbullet}}\geq\qFPr{\SoS-\oSoS}=
  \Vnormlp{\SoS-\oSoS}^2-\sbF{\SoS-\oSoS}$, which in turn for each
  $\SoS-\oSoS\in\rwcSo$ implies
  $\rEv\wEv^{-2}\qFPr{\mu_{\mbullet}}\geq\Vnormlp{\SoS-\oSoS}^2-\rSo\swSo[\Di]$. This
  bound together with \eqref{m:in:ub:p1} implies
  $\tfrac{4}{5}\qFPr[\mDid]{\mu_{\mbullet}}\geq
  2\lcst{\alpha}\rqFPr[\mDid]{\stnlImS}+2\lcst[2]{\alpha}
  \mFPr[\mDid]{\stnlImS}+\rqFPr[\mDid]{\stnlImS}
  2(\lcst{\beta}+5\lcst[2]{\beta})$. Rearranging the last inequality
  and proceeding as in the proof of \eqref{m:re:ub:e2} we obtain
  \eqref{m:in:ub:e2}, which completes the proof.
\end{proof}
\paragraph{Definition.}For $\alpha \in (0,1)$ and $\Di\in\Nz$ consider
the test statistic and the test
\begin{equation}\label{de:dt}
  \teSd[\Di,]{\alpha}:=\tqF-2\lcst{\alpha}\rqFPr{\sonlImS}-
  2\lcst[2]{\alpha}\mFPr{\sonlImS}\quad\text{ and }\quad
  \teFd[\Di,]{\alpha}:=\Ind{\set{\teSd[\Di,]{\alpha}>0}}.
\end{equation}
Exploiting \eqref{m:in:ub:e1}, $\teFd[\Di,]{\alpha/2}$ defined in
\eqref{de:dt} is a level-$\alpha/2$-test for any
$\Di\in\Nz$. Moreover,
\begin{equation}\label{de:dt:op}
  \teFd[]{\alpha/2}:=\teFd{\alpha/2}\quad\text{with
  }\mDid:=\ymDiCd\wedge\xmDiCd\text{ as in \eqref{direct:radius}}
\end{equation}  
is a $(1-\alpha/2)$-powerful test over
$\oacst{\alpha}[\ymRaCd\vee\xmRaCd]$-separated alternatives due to
\eqref{m:in:ub:e2} with $\beta=\alpha/2$ and
$\oacst[2]{\alpha}:=\rSo+\rEv(10\lcst{\alpha/2}+30\lcst[2]{\alpha/2})$. Hence, $\ymRaCd \vee\xmRaCd$ is an upper bound for the radius of testing of the test $  \teFd[]{\alpha/2}$.   Moreover, it is also a lower bound for its radius of
testing, which we prove in the next proposition.
\begin{proposition}\label{m:di:mrt}
  Let $\alpha\in(0,1)$. With
  $\oacst[2]{\alpha}:=\rSo+\rEv(10\lcst{\alpha/2}+30\lcst[2]{\alpha/2})$ we obtain
  \begin{equation}\label{m:di:mrt:e1}
    \forall \acst{}\in[\oacst{\alpha},\infty):
    \RiT{\teFd[]{\alpha/2}}{\wrcSo,\wrcEv}{\oSoS,\acst{}[\ymRaCd\vee\xmRaCd]}\leq \alpha,
  \end{equation}
  and with $\uacst[2]{\alpha}=\rSo\eta$ and
  $\eta\in(0,\swSo[\mDid]/(\ymRaCd \vee\xmRaCd)^2]$ it holds
  \begin{equation}\label{m:di:mrt:e2}
    \forall A \in(0,\uacst{\alpha}]:
    \RiT{\teFd[]{\alpha/2}}{\rwcSo,\rwcEv}{\oSoS,A [\ymRaCd \vee\xmRaCd]} \geq 1-\alpha,
  \end{equation} 
  i.e.  $\ymRaCd \vee\xmRaCd$ is a radius of testing of the test
  $\teFd[]{\alpha/2}$.
\end{proposition}
\begin{proof}[Proof of \cref{m:di:mrt}]Firstly, \eqref{m:di:mrt:e1} is
  an immediate consequence of \cref{m:in:ub} and we omit the
  details. Secondly, consider \eqref{m:di:mrt:e2}. We note that for
  each $\EvS\in\rwcEv$ and $\SoS-\oSoS \in \wrcSo$ with
  $\qFPr[\mDid]{\SoS-\oSoS}=0$ it holds
  $ \tqF+\qFPr{\onlImS}\sim\FuVgQ[\onlImS]{\nSoS,\Di}$ and thus
  $\alpha/2\geq \FuVg[\nlImS,\nlOpS]{\SoS, \EvS}(\teFd[]{\alpha} = 1)$
  due to \eqref{qchi:e1} in \cref{re:qchi}. For any
  $\SoS-\oSoS\in\rwcSo$ with $\qFPr[\mDid]{\SoS-\oSoS}=0$ and
  $\sbF[\mDid]{\SoS-\oSoS} = \rSo\swSo[\mDid]$, for example
  $\SoS-\oSoS=\Nsuite{\sqrt{\rSo}\wSo[\mDid]\Ind{\set{j=\mDid+1}}}$,
  it immediately follows
  $ \uacst[2]{\alpha}[\ymRaCd\vee\xmRaCd]^2 \leq \rSo\swSo[\mDid]=
  \sbF[\mDid]{\SoS-\oSo} = \Vnormlp{\SoS-\oSo}^2.$ Hence, for all
  $ A \in(0,\uacst{\alpha}]$ we obtain
  \begin{align*}
    \RiT{\teFd[]{\alpha/2}}{\rwcSo,\rwcEv}{\oSoS,A
    [\ymRaCd\vee\xmRaCd]}
    \geq \FuVg[\nlImS,\nlOpS]{\SoS, \EvS}(\teFd[]{\alpha} = 0) \geq 1 - \alpha/2.
  \end{align*}
  which shows \eqref{m:di:mrt:e2} and completes the proof.
\end{proof}
\begin{remark}\label{m:di:mrt:rem1}
  Considering the signal detection task, i.e., $\oSoS=\nSoS$, we have
  $\xmRaCd=0$ for all $\nlOpS\in\pRz^\Nz$, and thus the testing radius
  does not depend on the noise level $\nlOpS$.  Moreover, we shall
  emphasise that for all $\nlImS\geq\nlOpS$ we have
  $\wEvS^{-2}\rqFPrS{\snlOpS\oSoS^2}\leq \VnormInf{\oSoS}^2
  \wEvS^{-2}\rqFPrS{\snlImS}$ and thus
  $\xmRaCd\leq\VnormInf{\oSoS} \ymRaCd$. In other words, $\xmRaCd$ is
  negligible compared to $\VnormInf{\oSoS}\ymRaCd$ for all
  $\nlImS\geq\nlOpS$.  Let us briefly discuss under which conditions
  the direct test attains the minimax radius $\ymRaCi \vee
  \xmRaCi$. For any $\nlImS\in\pRz^\Nz$ the elementary inequality
  $\wEvS^{-2} \rqFPrS{\snlImS}\geq\rqFPrS{\snlImS/\wEvS^2}$ shows
  $\ymRaCd\geq\ymRaCi$. On the other hand, in the signal detection
  case, if there exists $c\in\pRz$ such that also
  $\wEvS^{-2} \rqFPrS{\snlImS}\leq c\rqFPrS{\snlImS/\wEvS^2}$, then
  the direct test $\teFd[]{\alpha/2}$ as in \eqref{de:dt:op} attains
  the minimax radius $\ymRaCi$. Note that, however, the additional condition is
  sufficient but not necessary as we will see in the illustration
  below.  Considering the goodness-of-fit task, i.e.,
  $\oSoS \not= \nSoS$, for all $\nlOpS\in\pRz^\Nz$ we obtain
  $\xmRaCd\geq\xmRaCi$ by exploiting again the elementary inequality
  $\wEvS^{-2}
  \rqFPrS{\snlOpS\oSoS^2}\geq\rqFPrS{\snlOpS\oSoS^2/\wEvS^2}$. Therefore,
  if there exists a $c\in\pRz$ such that
  $\wEvS^{-2} \rqFPrS{\snlImS}\leq c\rqFPrS{\snlImS/\wEvS^2}$ and also
  $\wEvS^{-2} \rqFPrS{\snlOpS\oSoS^2}\leq
  c\rqFPrS{\snlOpS\oSoS^2/\wEvS^2}$, then the direct test
  $\teFd[]{\alpha/2}$ attains the minimax radius $\ymRaCi\vee\xmRaCi$,
  where these additional conditions are again sufficient but not
  necessary.
\end{remark}
\begin{illustration}\label{illustration:direct}
  In the homoscedastic case, we illustrate the order of the rate and
  corresponding dimension parameter of the direct test
  $\teFd[]{\alpha/2}$ defined in \eqref{de:dt:op} by considering the
  typical smoothness and ill-posedness assumptions as in
  \cref{illustration:indirect}. The table displays the order of the
  rate $\ymRaCd\vee\xmRaCd$ for the signal detection task (with
  $\xmRaCd=0$ as discussed in \cref{m:di:mrt:rem1}) and the
  goodness-of-fit task. In comparison to \cref{illustration:indirect}
  we shall emphasise that in all three cases the order of $\ymRaCd$
  and $\ymRaCi$ coincide. Note that there exists a $c\in\pRz$ such that
  $\wEvS^{-2} \rqFPrS{\snlImS}\leq c\rqFPrS{\snlImS/\wEvS^2}$ in case
  of a mildly ill-posed model. In a severely ill-posed model, however,
  there exists no such $c$. Nonetheless, in both cases the direct test
  performs optimally with respect to the noise level
  $\nlIm$. Comparing the orders of $\xmRaCd$ and $\xmRaCi$ we note that in
  both a mildly and severely ill-posed model there does not exist
  $c\in\pRz$ such that
  $\wEvS^{-2} \rqFPrS{\snlOpS\oSoS^2}\leq c
  \rqFPrS{\snlOpS\oSoS^2/\wEvS^2}$.  Even so, for severely ill-posed
  models the rate $\ymRaCd \vee \xmRaCd$ and the minimax rate
  $\ymRaCi\vee\xmRaCi$ are of the same order, and thus the direct test
  is minimax optimal. On the other hand, for mildly ill-posed models
  the rate $\xmRaCd$ is always nonparametric and might be much slower
  than the
  rate $\xmRaCi$, which can be parametric.\\[3ex]
  \centerline{\begin{tabular}{@{}ll|ll|ll@{}} \toprule
                \multicolumn{6}{@{}l@{}}{Order of the dimension parameter $\ymDiCd\wedge\xmDiCd$ and  rate $\ymRaCd\vee\xmRaCd$} \\
                \midrule
                $\wSoS$  & $\wEvS$  &\hfill$\ymDiCd$\hfill\null&\hfill$\ymRaCd$\hfill\null&\hfill$\xmDiCd$\hfill\null&\hfill$\xmRaCd$\hfill\null\\
                {\smaller (smoothness)}  & { \smaller (ill-posedness)}& \multicolumn{2}{|c|}{\smaller $\oSoS \in \lp^2$}& \multicolumn{2}{c@{}}{\smaller$\oSoS =\Nsuite{ j^{-\oSoP}}$} \\
                \midrule
                $\Nsuite{ j^{-\wSoP}}$ & $\Nsuite{j^{-\wEvP}}$ & $\nlIm^{-\frac{4}{4\wSoP+4\wEvP+1}}$ & $\nlIm^{\frac{4\wSoP}{4\wSoP+4\wEvP+1}}$  &$\nlOp^{-\frac{1}{\wSoP+\wEvP}}$&$\nlOp^{\frac{\wSoP}{\wSoP+\wEvP}}$ \\
                $\Nsuite{e^{-j^{2\wSoP}}}$ & $\Nsuite{j^{-\wEvP}}$ &
                                                                     $|\log
                                                                     \nlIm
                                                                     |^{\frac{1}{2\wSoP}}$
                                                               &
                                                                 $|\log
                                                                 \nlIm
                                                                 |^{\frac{4
                                                                 \wEvP+1}{8\wSoP}}\nlIm$
                                                                                          &
                                                                                            $|\log \nlOp |^{\frac{1}{2\wSoP}}$ & $|\log \nlOp |^{\frac{\wEvP}{2\wSoP}}\nlOp $ \\
                $\Nsuite{ j^{-\wSoP}}$ & $\Nsuite{e^{-j^{2\wEvP}}}$ &
                                                                      $|\log\nlIm|^{\frac{1}{2\wEvP}}$
                                                               &$|\log
                                                                 \nlIm|^{-\frac{\wSoP}{2\wEvP}}$
                                                                                          &
                                                                                            $|\log
                                                                                            \nlOp
                                                                                            |^{\frac{1}{2\wEvP}}$
                                                                                                                     &
                                                                                                                       $|\log
                                                                                                                       \nlOp|^{-\frac{\wSoP}{2\wEvP}}$
                \\\bottomrule
\end{tabular}}\end{illustration}
\paragraph{Direct testing task.}\cite{LaurentLoubesMarteau2011} show that for known operators, under specific smoothness and ill-posedness assumptions (covered
also in the \cref{illustration:indirect,illustration:direct} above), every minimax optimal test for the direct task is
also minimax optimal for the indirect task. Even under these specific
assumptions, this is no longer the case for unknown operators if
$\ymRaCd$ is negligible compared to $\xmRaCd$. Keeping
\cref{m:di:mrt:rem1} and \cref{illustration:indirect} in mind, we observe that the test $\teFd[]{\alpha/2}$ defined in
\eqref{de:dt:op} is not always optimal for the indirect task. Nonetheless,
we show below that it attains the minimax radius for the direct task,
which we formalise next. Introducing
$\EvS\wrcSo:=\{\EvS\SoS\in\lp^2:\SoS\in\wrcSo\}$, the direct testing
task can be written as
\begin{align}\label{directtesting:e1}
  H_0:  \EvS(\SoS - \oSoS)= \nSoS,\EvS \in \rwcEv \;
  \text{ against } \; H_1^\rho:\EvS(\SoS - \oSoS) \in
  \lp[\rho]^2 \cap \EvS\wrcSo,\EvS \in \rwcEv.
\end{align}
Given a test $\varphi$ we define its maximal risk w.r.t the direct testing problem (DP) in \eqref{directtesting:e1} by
\begin{multline*}
  \RdT{\teF}{\rwcSo,\rwcEv}{\oSoS,\rho} :=
  \sup\set{ \FuVg[\nlImS,\nlOpS]{\oSoS, \EvS}(\teF = 1):\EvS \in
    \rwcEv}\\
  + \sup\set{\FuVg[\nlImS,\nlOpS]{\SoS, \EvS}(\teF = 1):\EvS(\SoS -
    \oSoS)
    \in \lp[\rho]^2 \cap \EvS\wrcSo, \EvS\in\rwcEv}
\end{multline*}
and we characterise the difficulty of the direct testing task by the minimax risk
\begin{align*}
  \mRdT{\rwcSo,\rwcEv}{\oSoS,\rho}  
  := \inf_{\teF} \RdT{\teF}{\rwcSo,\rwcEv}{\oSoS,\rho}
\end{align*}
where the infimum is taken over all possible tests.  Keeping the definition
\eqref{direct:radius} and $\mDid:=\ymDiCd\wedge\xmDiCd$ in mind, we
define
\begin{multline}
  \ymRaCdp := \wEv[\mDid] \ymRaCd \quad \text{and} \quad
  \xmRaCdp := \wEv[\mDid] \xmRaCd. \hfill
\end{multline}
We show next that the minimax radius for the direct problem is
given by $\ymRaCdp\vee\xmRaCdp$.
\begin{proposition}\label{m.di.mrt}Let $\eta\in(0,1]$ satisfy
  \begin{equation}\label{m.di.mrt.c}
    \eta \leq
    \frac{\rqFPr[\mDid]{\snlImS}\vee\rqFPr[\mDid]{\snlOpS\oSoS^2}\wedge\wEv[\mDid]^2\swSo[\mDid]}{(\ymRaCdp \vee \xmRaCdp)^2}=\frac{\rqFPr[\mDid]{\snlImS}\vee\rqFPr[\mDid]{\snlOpS\oSoS^2}\wedge\wEv[\mDid]^2\swSo[\mDid]}{\rqFPr[\mDid]{\snlImS}\vee\rqFPr[\mDid]{\snlOpS\oSoS^2}\vee\wEv[\mDid]^2\swSo[\mDid]}.
  \end{equation}
  Let $\alpha \in (0,1)$. With
  $\uacst[2]{\alpha}:=\eta\big(\rSo\wedge
  \sqrt{2\log(1+2\alpha^2)}\big)$ we obtain
  \begin{equation}\label{m.di.mrt.e1}
    \forall A \in [0, \uacst{\alpha}]: 
    \mRdT[\nlImS,\nlOpS]{\rwcSo, \rwcEv}{\oSoS, A[ \ymRaCdp \vee
      \xmRaCdp]}
    \geq 1- \alpha,
 \end{equation}
 and with
 $\oacst[2]{\alpha}:=\rSo\rEv+10\lcst{\alpha/2}+30\lcst[2]{\alpha/2}$
 and $\teFd[]{\alpha/2}$ defined in \eqref{de:dt:op} it holds
 \begin{multline}\label{m.di.mrt.e2}
   \forall A \in [\oacst{\alpha},\infty): 
   \mRdT[\nlImS,\nlOpS]{\rwcSo, \rwcEv}{\oSoS, A [ \ymRaCdp \vee
     \xmRaCdp]}\\
   \leq \RdT[\nlImS,\nlOpS]{\teFd[]{\alpha/2}}{\rwcSo, \rwcEv}{\oSoS,
     A[ \ymRaCdp \vee \xmRaCdp]} \leq \alpha,
 \end{multline}
 i.e.  $\ymRaCdp \vee \xmRaCdp$ is a {minimax radius of testing} for the direct testing task
 \eqref{directtesting:e1}.
\end{proposition}
\begin{proof}[Proof of \cref{m.di.mrt}] The proof of the lower bound
  \eqref{m.di.mrt.e1} follows along the lines of the proof of
  \cref{lowerbound}, using the same reduction argument with a mixture
  over vertices of hypercubes. Given
  $\mRadp := \ymRaCdp \vee \xmRaCdp$ and
  $\sonlImS=\snlImS+\oSoS^2\snlOpS$ let us define the parameter
  $\tSoS=\Nsuite{\tSo[j]}$ with $\tSo[j]:=0$ for
  $j>\mDid=\ymDiCd\wedge\xmDiCd$, and 
  \begin{equation*}
    \tSo[j]:=\tfrac{\mRadp\sqrt{\zeta\eta}}{\rqFPr[{\mDid}]{\sonlImS}}
    \frac{\sonlIm[j]}{\wEv[j]} \quad \text{ for } j\in\nset{\mDid},
    \quad\text{ and } \zeta:=\rSo\wedge \sqrt{2\log(1+\alpha^2)}.
  \end{equation*}
  We need to check that it satisfies the conditions \ref{lowerbound.a}
  and \ref{lowerbound.b} given in the proof of
  \cref{lowerbound}. Indeed with
  $\qFPrS{\sonlImS/\wEvS}\leq \wEvS^{-2}\qFPrS{\sonlImS}$ and
  $\sqrt{\eta}\mRadp \leq\wSo[\mDid] \wEv[\mDid]$ due to
  \eqref{m.di.mrt.c} we obtain \ref{lowerbound.a}, that is,
  $ \sbF[l]{\tSoS} \leq \qFPr[\mDid]{\tSoS} =\zeta\eta (\mRadp)^2
  \tfrac{\qFPr[\mDid]{\sonlImS/\wEvS}}{\qFPr[\mDid]{\sonlImS}}\leq
  \zeta\swSo[\mDid]\leq\rSo\swSo[\mDid] \leq \rSo\swSo[l] $ for all
  $l \in \nset{\mDid}$, $\sbF[l]{\tSoS} = 0 \leq \rSo\swSo[l]$ for all
  $l > \mDid$ and
  $ \Vnormlp{\wEvS\tSoS}^2=\qFPr[\mDid]{\wEvS\tSoS}
  =\zeta\eta(\mRadp)^2\tfrac{\qFPr[\mDid]{\sonlImS}}{\qFPr[\mDid]{\sonlImS}}=
  \zeta\eta(\mRadp)^2=\uacst[2]{\alpha}(\mRadp)^2$.  On the other
  hand, \ref{lowerbound.b} holds too, since
  $\eta^2(\mRadp)^4\leq
  [\rqFPr[\mDid]{\snlImS}\vee\rqFPr[\mDid]{\snlOpS\oSoS^2}]^2\leq
  \qFPr[{\mDid}]{\sonlImS}$ and $\zeta^2\leq 2 \log(1+\alpha^2)$ imply
  together
  $\qFPr[\mDid]{\wEvS^2\tSoS^2/\sonlImS}=\zeta^2
  \tfrac{\eta^2(\mRadp)^{4}}{(\qFPr[{\mDid}]{\sonlImS})^2}\qFPr[{\mDid}]{\sonlImS}
  \leq\zeta^2\leq 2 \log(1+\alpha^2)$. Combining \ref{lowerbound.a}
  and \ref{lowerbound.b} the claim \eqref{m.di.mrt.e1} follows.
  The upper bound \eqref{m.di.mrt.e2} is an immediate consequence of
  \begin{inparaenum}[i]\renewcommand{\theenumi}{\dgrau\rm(\roman{enumi})}
  \item\label{m.di.mrt.b1}
    $\sup\set{\FuVg[\nlImS,\nlOpS]{\oSoS,\EvS}
      \big(\teFd[]{\alpha/2}=1\big),\EvS\in\rwcEv}\leq\alpha/2$ due to
    \eqref{m:in:ub:e1} in \cref{m:in:ub}, and \item\label{m.di.mrt.b2}
    $\FuVg[\nlImS,\nlOpS]{\SoS,\EvS}\big(\teFd[]{\alpha/2}=0\big)\leq\alpha/2$
    for each $\EvS\in\rwcEv$ and
    $\EvS(\SoS-\oSoS)\in\EvS\rwcSo\cap\lp[\rho]^2$ with
    $ \rho\geq
  \oacst{\alpha}
    \mRadp$. \end{inparaenum}
  The proof of \ref{m.di.mrt.b2} is
  similar to the proof of \eqref{m:in:ub:e2} in
  \cref{m:in:ub}. Instead of \eqref{m:in:ub:p1} we obtain
  \begin{multline}\label{m.di.mrt.p2}
    \Vnormlp{\EvS(\SoS  - \oSoS)}^2
    \geq \rSo\rEv\wEv[\mDid]^{2}\swSo[\mDid] +
    [\rqFPr[\mDid]{\snlImS}\vee\rqFPr[\mDid]{\snlOpS\oSoS^2}]
    (10\lcst{\alpha/2}+30\lcst[2]{\alpha/2})
    \geq \rEv\rSo\wEv[\mDid]^2\swSo[\mDid] \\ +
    \rqFPr[\mDid]{\stnlImS}\tfrac{5}{2}
    \big(\lcst{\alpha/2}\rqFPr[\mDid]{\stnlImS}+\lcst[2]{\alpha/2}
    \mFPr[\mDid]{\stnlImS}+\rqFPr[\mDid]{\stnlImS}
    (\lcst{\alpha/2}+5\lcst[2]{\alpha/2})\big)
  \end{multline}   
  using
  $(\mRadp)^2=\rqFPr[\mDid]{\snlImS}\vee\rqFPr[\mDid]{\snlOpS\oSoS^2}
  \vee\wEv[\mDid]^{2}\swSo[\mDid]$ and
  $2[\rqFPr[\mDid]{\snlImS}\vee\rqFPr[\mDid]{\snlOpS\oSoS^2}]\geq
  \rqFPr[\mDid]{\stnlImS}\geq\mFPr[\mDid]{\stnlImS}$. Moreover, for
  each $\Di\in\Nz$ and $\EvS\in\rwcEv$ we have
  $\qFPr{\EvS(\SoS-\oSoS)}\geq
  \Vnormlp{\EvS(\SoS-\oSoS)}^2-\rEv\wEv[\Di]^2\sbF{\SoS-\oSoS}$ and hence
  $\qFPr{\EvS(\SoS-\oSoS)}\geq
  \Vnormlp{\EvS(\SoS-\oSoS)}^2-\rEv\rSo\wEv[\Di]^2\swSo[\Di]$
  for each $\SoS-\oSoS\in\rwcSo$. Together with
  \eqref{m.di.mrt.p2} this implies
  $\tfrac{4}{5}\qFPr[\mDid]{\EvS(\SoS-\oSoS)}\geq
  2\lcst{\alpha/2}\rqFPr[\mDid]{\stnlImS}+2\lcst[2]{\alpha/2}
  \mFPr[\mDid]{\stnlImS}+\rqFPr[\mDid]{\stnlImS}
  (2\lcst{\alpha/2}+10\lcst[2]{\alpha/2})$. Rearranging the last
  inequality we proceed exactly as in the proof of \eqref{m:re:ub:e2}
  in \cref{m:re:ub} and obtain the claim.
\end{proof}
The last result establishes the minimax optimality of the testing radius
$\ymRaCdp\vee\xmRaCdp$ for the direct testing task
\eqref{directtesting:e1} and hence also the minimax optimality of  the test $\teFd[]{\alpha/2}$ given in
\eqref{de:dt:op}. Similar findings as in \cref{m:rem:mrt1,m:rem:mrt2}
hold.
\begin{illustration}[homoscedastic case] We illustrate the order of
  the minimax rate for the direct testing task
  \eqref{directtesting:e1} by considering typical smoothness and
  ill-posedness assumptions as in the
  \cref{illustration:indirect,illustration:direct}. The table displays
  the order of the minimax rate $\ymRaCdp\vee\xmRaCdp$ for the direct
  signal detection (with
  $\xmRaCdp=0$) and the goodness-of-fit task.	\\[3ex]
  \centerline{\begin{tabular}{@{}ll|ll@{}} \toprule
                \multicolumn{4}{@{}l@{}}{Order of the minimax rate $\xmRaCdp \vee \ymRaCdp$ for the direct testing task} \\
                \midrule
                $\wSoS$  & $\wEvS$  & $\ymRaCdp$    & $\xmRaCdp$ \\
                {\smaller (smoothness)}  & { \smaller (ill-posedness)}& {\smaller $\oSoS \in \lp^2$}&  {\smaller$\oSoS =\Nsuite{ j^{-\oSoP}}$} \\
                \midrule
                $\Nsuite{ j^{-\wSoP}}$ & $\Nsuite{j^{-\wEvP}}$ & $ \nlIm ^{\frac{4\wSoP+4\wEvP}{4\wSoP+4\wEvP+1}}$ & $\nlOp$  \\
                $\Nsuite{e^{-j^{2\wSoP}}}$ &  $\Nsuite{j^{-\wEvP}}$ & $|\log \nlIm |^{\frac{1}{8\wSoP}}\nlIm  $ & $\nlOp $ \\
                $\Nsuite{ j^{-\wSoP}}$ & $\Nsuite{e^{-j^{2\wEvP}}}$ &$|\log  \nlIm|^{\frac{1}{8\wEvP}} \nlIm$  & $\nlOp$  \\
                \bottomrule
\end{tabular}}\end{illustration}

%
%
%
%
\section{Adaptation}\label{sec:adaptation}
\subsection{Description of the procedure}\label{sec:description}
For both the indirect and the direct test the optimal choice of the
dimension parameter relies on prior knowledge of the sequences
$\wSoS$ and $\wEvS$, which are typically unknown in practice. In this
section we study an aggregation of the tests for several dimension
parameters, which leads to a testing procedure that performs nearly
optimal over a wide range of regularity classes. We
first present the testing radii of these aggregation-tests, where
compared to the minimax radii of testing the noise level in the radii
of testing face a deterioration by a $\log$-factor only. Moreover, we
show that this deterioration is an unavoidable cost for adaptation.
 
Let us briefly describe a widely used aggregation strategy. For a
sequence of levels $(\alpha_{\Di})_{\Di \in \IN}$ let
$(\teststat)_{\Di\in \Nz}$ be a sequence of test statistics such that
$\test = \Ind{\{\teststat > 0\}}$ is a level-$\alpha_{\Di}$-test for
each $\Di \in \IN$. Note that both the indirect and the direct testing
procedures satisfy this condition by construction as shown in
\eqref{m:re:ub:e1} and \eqref{m:in:ub:e1} of \cref{m:re:ub,m:in:ub},
respectively. Given a finite collection $\mc K \subset \IN$ of
dimension parameters and $\alpha:= \sum_{\Di \in \mc K} \alpha_{\Di}$,
we consider the $\max$-test statistic and the corresponding
$\max$-test
\begin{align*}
  \maxteststat = \max_{\Di \in \cK} \teststat, \qquad
  \maxtest := \Ind{\{ \maxteststat > 0 \}},
\end{align*}
that is, the $\max$-test rejects the null hypothesis as soon as one of
the tests does. Due to the elementary inequality
\begin{align*} 
  \FuVg[\nlImS,\nlOpS]{\oSoS,\EvS}\lb \maxtest = 1 \rb =
  \FuVg[\nlImS,\nlOpS]{\oSoS,\EvS} \lb  \maxteststat > 0 \rb
  \leq \sum_{\Di \in \cK} 	\FuVg[\nlImS,\nlOpS]{\oSoS,\EvS}
  \lb\teststat
  > 0\rb\leq  \sum_{\Di \in \mc K} \alpha_{\Di} = \alpha,
\end{align*}
the $\max$-test $\maxtest$ is a level-$\alpha$-test. The type II error
probability of the $\max$-test can be controlled by any test contained
in the collection, since for all $\SoS \in \lp^2$ and
$\EvS \in \lp^\infty$
\begin{align*}
  \FuVg[\nlImS,\nlOpS]{\SoS,\EvS}\lb\maxtest= 0 \rb =
  \FuVg[\nlImS,\nlOpS]{\SoS,\EvS}\lb \maxteststat \leq  0\rb
  \leq \min_{\Di \in \cK}  \FuVg[\nlImS,\nlOpS]{\SoS,\EvS}\lb
  \teststat \leq 0 \rb
  = \min_{\Di \in \cK}  \FuVg[\nlImS,\nlOpS]{\SoS,\EvS}\lb \test = 0 \rb . 
\end{align*}
These two error bounds have oppositional effects on the choice of the
collection $\cK$. Roughly speaking, $\cK$ should not be too large due
to the aggregation of type I error probabilities. On the other hand, it
should still be large enough to minimise the type II error
probabilities. Typically, the choice of $\cK$ will depend on the
original and effective noise level, $\nlImS$ and $\oSoS\nlOpS$,
respectively.

The goal of the aggregation is to find testing strategies for which
the risk can be controlled over large families of alternatives. To
measure the cost for adaptation, we introduce factors $\cRaLo[\nlImS]$
and $\cRaLo[\nlOpS]$, which are typically called \textit{adaptive factors}
(cf. \cite{Spokoiny1996}), for a test $\phi_\alpha$ and a family of
alternatives $\CrwcSo \times \CrwcEv$, if for every
$\alpha \in (0,1)$ there exists a constant $\oacst{\alpha} \in \IR_+$
such that
\begin{equation}\label{adaptive:upper:bound}
  \forall \acst{}\in[\oacst{\alpha},\infty):
  \sup_{(\wSoS, \wEvS) \in \CwSo \times \CwEv}
  \RiT{\phi_\alpha}{\wrcSo,\wrcEv}{\oSoS,\acst{}
    [\yomRaCiaa \vee \xomRaCiaa]}\leq\alpha.
\end{equation}
Here, $\ymRaCi\vee \xmRaCi$ is the minimax radius of testing over $\rwcSo \times \rwcEv$ defined in  \eqref{indirect:radius}.  $\cRaLo[\nlImS]$ and $\cRaLo[\nlOpS]$ are called \textit{minimal adaptive factors} if in addition for every $\alpha \in (0,1)$ there exists a constant $\underline{A}_\alpha$ such that 
\begin{equation}\label{adaptive:lower:bound}
  \forall A \in [0,\underline{A}_\alpha]: \sup_{(\wSoS, \wEvS) \in \CwSo \times \CwEv}	\mRiT{\wrcSo,\wrcEv}{\oSoS,\acst{}[\yomRaCiaa  \vee \xomRaCiaa]}\geq  1 - \alpha.
\end{equation}
If the minimal adaptive factors tend to infinity as the noise levels decrease to zero, then this phenomenon is typically called \textit{lack of adaptability}. \\
In this section we first carry out an aggregation of the
indirect tests. Recall that the indirect test statistic \eqref{de:it}
explicitly uses the knowledge of the sequence $\wEvS$. Therefore, we
consider adaptation to $\CrwcSo \times \rwcEv$ for given $\wEvS$
only. We present the adaptive factors for the indirect-$\max$-test and
show that they coincide asymptotically with the minimal adaptive
factors. Afterwards as an alternative to the indirect-$\max$-test we
study an aggregation of the direct tests, which depend on the sequence
$\wEvS$ only through the choice of the optimal dimension
parameter. Hence, for a direct-$\max$-test we consider adaptation to
both $\CrwcSo$ and $\CrwcEv$.
\begin{remark} Let us briefly comment on possible choices of the error levels. Throughout the paper, given a level $\alpha \in (0,1)$ and a finite collection $\mc \cK \subset \IN$ we consider
  Bonferroni levels $\alpha_{\Di}:=\alpha/|\cK|$, $\Di\in\Nz$, i.e.,
  the same level $\alpha/|\cK|$ for each test statistic
  $\teSa{\alpha_{\Di}}$ in the collection $\Di\in\cK$. A usually proposed alternative is to select the value
  $\alpha^\circ:=\sup\{u\in(0,1):\FuVg[\nlImS,\nlOpS]{\oSoS,\EvS}(\max_{\Di
    \in \cK} \teSa{u}>0)\leq\alpha\}$ as a common level for all tests
  in the collection. By construction the max-test corresponding to the
  max-test statistic $\max_{\Di \in \cK} \teSa{\alpha^\circ}$ is a
  level-$\alpha$-test and it is at least as powerful as the max-test
  with Bonferroni levels, if in addition also
  $\max_{\Di\in \cK} \teSa{\alpha^\circ}\geq \max_{\Di\in \cK}
  \teSa{\alpha/|\cK|}$ holds. To be more precise, let us revisit
  the indirect test statistic $\teSi[\Di,]{\alpha}$ given in
  \eqref{de:it}. For $u\in(0,1)$ and $\Di\in\Nz$ we denote by
  $t_{\Di}(u)$ the $(1-u)$-quantile of $\hqF$ under the null hypothesis,
  i.e., $\FuVg[\nlImS,\nlOpS]{\oSoS,\EvS}(\hqF>t_{\Di}(u))=u$. Then
  for each $\Di\in\Nz$ and $\alpha\in(0,1)$ the test corresponding to
  the test-statistic $\teSa{\alpha}:=\hqF-t_{\Di}(\alpha)$ is a
  level-$\alpha$-test and $\teSa{\alpha}\geq\teSi[\Di,]{\alpha}$ due
  to \eqref{m:re:ub:e1} in \cref{m:re:ub}. Consequently, the
  level-$\alpha$ max-test corresponding to
  $\max_{\Di \in \cK} \teSa{\alpha^\circ}$ is at least as powerful as
  the level-$\alpha$ max-test
  $\teFi[\cK,]{\alpha}:=\Ind{\set{\teS[\cK,]{\alpha}>0}}$
  corresponding to the max-test statistic
  $\max_{\Di \in \cK} \teS[\cK,]{\alpha/|\cK|}$ with Bonferroni
  levels. However, in opposition to the Bonferroni levels there is no
  explicit expression for the value $\alpha^\circ$ and, hence, it has
  to be determined by a simulation study.
\end{remark}
\subsection{Adaptation to smoothness }\label{adaptation:indirect}
\paragraph{Indirect testing procedure.}In this section we consider the
adaptation of the indirect test in \eqref{de:it} to a family of
alternatives $\CrwcSo$. Given $\alpha \in (0,1)$ and a finite collection
$\mc \cK \subset \IN$ define the $\max$-test statistic with Bonferroni
levels and the corresponding $\max$-test
\begin{equation}\label{de:it:ad}
  \teS[\cK,]{\alpha}:=\max_{\Di\in\cK}(\teS[\Di,]{\alpha/|\cK|})
  \quad\text{ and }\quad
  \teFi[\cK,]{\alpha}:=\Ind{\set{\teSi[\cK,]{\alpha}>0}},
\end{equation}
which is a level-$\alpha$-test due to \eqref{m:re:ub:e1} in
\cref{m:re:ub}.  Its testing radius faces a deterioration compared to
the minimax radius due to the Bonferroni aggregation, which we
formalise next.  Analogously to \eqref{indirect:radius}, for each
$x_{\mbullet} \in \IR^\IN$ let us define the
minimum  over the collection $\cK$ 
\begin{equation}\label{it:radius:re}
  (\mRaCir[x_{\mbullet}])^2:=
  \min\nolimits_{\cK}(\mFPrS{x_{\mbullet}^2/\wEvS^2}\vee\swSoS)
\end{equation}
and the minimum and minimiser over $\cK$, respectively,
\begin{align}\label{it:radius:ad}
  (\mRaCia[x_{\mbullet}])^2:=
  \min\nolimits_{\cK}(\rqFPrS{x_{\mbullet}^2/\wEvS^2}\vee\swSoS)
  \quad \text{and} \quad  \mDiCia[x_{\mbullet}]:=
  \argmin\nolimits_{\cK}(\rqFPrS{x_{\mbullet}^2/\wEvS^2}\vee\swSoS).
\end{align}
We first prove an upper bound in terms of the reparametrised noise
level $\sonlImS=\snlImS+\oSoS^2\snlOpS$ and the adaptive factor
$\dr\cRaLo:=(1\vee\log|\cK|)^{1/4}$.  The upper bound consists of the
two terms $\mRaCir$ and $\mRaCia$, defined by replacing $x_{\mbullet}$ with
${\dr\cRaLo^2}\onlImS$ and ${\dr\cRaLo}\onlImS$ in
\eqref{it:radius:re} and \eqref{it:radius:ad}, respectively. We think
of $\mRaCir$ as a reminder term, which is typically negligible
compared to $\mRaCia$ (cf. \cref{a:re:ub:rem} below).

\begin{proposition}\label{a:re:ub} For $\alpha\in(0,1)$ define
  $\oacst[2]{\alpha}:=\rSo+\rEv(5\lcst{\alpha/2}+15\lcst[2]{\alpha/2}+5)$. Then%
  \begin{equation}\label{a:re:ub:e}
    \forall \acst{}\in[\oacst{\alpha},\infty): \sup_{(\wSoS, \wEvS) \in \CwSo \times \CwEv}
    \RiT{\teFi[\cK,]{\alpha/2}}{\wrcSo,\wrcEv}{\oSoS,\acst{}
      [\mRaCir\vee\mRaCia]}\leq\alpha.
  \end{equation}
\end{proposition}
\begin{proof}[Proof of \cref{a:re:ub}] The proof follows along the
  lines of the proof of \cref{m:re:ub} and exploits \eqref{qchi:e1}
  and \eqref{qchi:e2} in \cref{re:qchi}.  For
  $(\yObS,\xObS)\sim\FuVg[\nlImS,\nlOpS]{\SoS,\EvS}$, we have
  $Q_{\Di}:=\hqF+\qFPr{\tnlImS}\sim\FuVgQ[\tnlImS]{\mu_{\mbullet},\Di}$
  for each $\Di\in\Nz$ with $\stnlImS:=\sonlImS/\wEvS^2$ and
  $\mu_{\mbullet}:=\EvS(\SoS-\oSoS)/\wEvS$. \eqref{qchi:e1} implies
  that under the null hypothesis with
  $\dr L:=\sqrt{\log(2|\cK|/\alpha)}$ the quantile satisfies
  $\quVgQ[\tnlImS]{\nSoS,\Di}{\alpha/(2|\cK|)}\leq \qFPr{{\tnlImS}}+
  2L\rqFPr{\stnlImS}+2L^2\mFPr{\stnlImS}$ and, therefore, 
  \begin{multline}\label{a:p1} 
    \FuVg[\nlImS,\nlOpS]{\oSoS,\EvS}\big(\teFi[\cK,]{\alpha/2}=1\big)
    =\FuVg[\nlImS,\nlOpS]{\oSoS,\EvS}\big(\teS[\cK,]{\alpha/2}>0\big)
    \leq \sum_{\Di\in\cK}\FuVg[\nlImS,\nlOpS]{\oSoS,\EvS}\big(\hqF
    >2L\rqFPr{\stnlImS}+2L^2\mFPr{\stnlImS}\big)\\\leq
    \sum_{\Di\in\cK}\FuVg[\nlImS,\nlOpS]{\oSoS,\EvS}\big(Q_{\Di}
    > \quVgQ[\tnlImS]{\nSoS,\Di}{\alpha/(2|\cK|)}\big)=
    \sum_{\Di\in\cK}\frac{\alpha}{2|\cK|}=\alpha/2.
  \end{multline} 
  On the other hand, under the alternative
  $\SoS-\oSoS\in \rwcSo\cap\lp[\rho]^2$ with
  $\rho\geq \oacst{\alpha}[\mRaCir\vee\mRaCia]$%
  \begin{multline}\label{a:p2}
    \Vnormlp{\SoS-\oSoS}^2
    \geq \rSo\swSo[\mDi] +
    \rEv\frac{5}{2}\big(L\rqFPr[\mDi]{\stnlImS} +
    L^2\mFPr[\mDi]{\stnlImS}
    +\rqFPr[\mDi]{\stnlImS}(\lcst{\alpha/2}+5\lcst[2]{\alpha/2})\big),\hfill
  \end{multline}
  where we successively use 
  $(\mRaCir\vee\mRaCia)^2=\rqFPr[\aDi]{{\dr\cRaLo^2}\stnlImS}
  \vee\mFPr[\aDi]{{\dr\cRaLo^4}\stnlImS} \vee\swSo[\aDi]$ with
  $\aDi:=\argmin\nolimits_{\cK}(\mFPrS{{\dr\cRaLo^4}\,\stnlImS}\vee\swSoS)
  \wedge\argmin\nolimits_{\cK}(\rqFPrS{{\dr\cRaLo^2}\,\stnlImS}\vee\swSoS)$
  due to \cref{re:argmin},
  $\rqFPr[\aDi]{{\dr\cRaLo^2}\stnlImS}\geq\rqFPr[\aDi]{\stnlImS}$,
  $\rqFPr[\aDi]{{\dr\cRaLo^2}\stnlImS}(\lcst{\alpha/2}+1)\geq
  L\rqFPr[\aDi]{\stnlImS}$ and
  $\mFPr[\aDi]{{\dr\cRaLo^4}\stnlImS}(\lcst[2]{\alpha/2}+1)\geq
  L^2\mFPr[\aDi]{\stnlImS}$. For all $\Di\in\Nz$, $\EvS\in\rwcEv$ and
  $\SoS-\oSoS\in\rwcSo \cap \lp[\rho]^2$ it holds
  $\rEv\qFPr{\mu_{\mbullet}}\geq\Vnormlp{\SoS-\oSoS}^2-\rSo\swSo[\Di]$,
  which together with \eqref{a:p2} implies
  $\tfrac{4}{5}\qFPr[\aDi]{\mu_{\mbullet}}\geq2L\rqFPr[\aDi]{\stnlImS}
  + 2L^2\mFPr[\aDi]{\stnlImS}
  +\rqFPr[\aDi]{\stnlImS}2(\lcst{\alpha/2}+5\lcst[2]{\alpha/2})$.
  Rearranging the last inequality and using \eqref{qchi:e2} in
  \cref{re:qchi} shows that for all $\EvS\in\rwcEv$
  \begin{multline}\label{a:p3} 
    \FuVg[\nlImS,\nlOpS]{\SoS,\EvS}\big(\teFi[\cK,]{\alpha/2}=0\big)
    \leq
    \min\{\FuVg[\nlImS,\nlOpS]{\SoS,\EvS}\big(Q_{\Di}
    \leq2L\rqFPr{\stnlImS} +2L^2\mFPr{\stnlImS}
    +\qFPr{\tnlImS}\big):\Di\in\cK\}\\
    \leq
    \FuVg[\nlImS,\nlOpS]{\SoS,\EvS}\big(Q_{\aDi}
    \leq \quVgQ[\tnlImS]{\mu_{\mbullet},\aDi}{1-\alpha/2}\big)=\alpha/2.
  \end{multline}  
  Combining \eqref{a:p1} and \eqref{a:p3} we obtain the assertion
  \eqref{a:re:ub:e}, which completes the proof.
\end{proof}
\begin{remark}\label{a:re:ub:rem}
  The first term
  $ \mRaCir =\min_{\cK}(\mFPrS{{\dr\cRaLo^4}\,\sonlImS/\wEvS^2}
  \vee\swSoS)$ in the upper bound \eqref{a:re:ub:e} for the adaptive
  testing radius can always be bounded by
  $\mRaCia[{\dr\cRaLo^2}\onlImS] =
  \min_{\cK}(\rqFPrS{{\dr\cRaLo^4}\,\sonlImS/\wEvS^2}\vee\swSoS)$ due
  to the elementary inequality
  $\mFPrS{{\dr\cRaLo^4}\,\sonlImS/\wEvS^2} \leq
  \rqFPrS{{\dr\cRaLo^4}\,\sonlImS/\wEvS^2}$. Note that
  $\mRaCia[{\dr\cRaLo^2}\onlImS]$ only differs from the second term
  $\mRaCia$ in \eqref{a:re:ub:e} by an additional factor $\dr
  \cRaLo$. Hence, we can always show that $\dr \cRaLo^2$ is an
  adaptive factor. However, often this bound is too rough and the term
  $\mRaCir$ is negligible compared to $\mRaCia$, which results in an
  adaptive factor $\dr \cRaLo$. Let us give sufficient conditions for the negligibility. Consider
  $\mDii:=\mDiCia$ as in \eqref{it:radius:ad}. Then we have
  $\mRaCir\leq \cst[2]{\cK}\mRaCia$ for any $\cst[2]{\cK}>0$ with
  ${\dr\cRaLo^2}(\mDii)^{-1/2}\rqFPr[\mDii]{{\dr\cRaLo^2}\,\sonlImS/\wEvS^2}\leq
  \mFPr[\mDii]{{\dr\cRaLo^4}\,\sonlImS/\wEvS^2}\leq
  \cst[2]{\cK}\rqFPr[\mDii]{{\dr\cRaLo^2}\,\sonlImS/\wEvS^2}$, which is
  satisfied whenever
  \begin{align}\label{negligible}
    {\dr\cRaLo^2}\leq \cst{\cK} \sqrt{\mDii} \qquad \text{
    and } \qquad  \sqrt{\mDii} \mFPr[\mDii]{ \onlImS^2/\wEvS^2}
    \leq  \cst{\cK}\;\rqFPr[\mDii]{ \onlImS^2/\wEvS^2}.
  \end{align}
  Moreover, comparing $\mRaCia$ and $\mRaCiaf$ (defined as in
  \eqref{indirect:radius} by replacing $x_{\mbullet}$ with
  ${\dr\cRaLo}\,\onlImS$) it holds $\mRaCiaf\leq\mRaCia$ for any
  collection $\cK$. In the \cref{illustration:adaptive} below we
  select a suitable collection $\cK$ such that uniformly for all 
  $\wSoS\in\CwSo$ we get $\mRaCia\leq\cst{\cK}\mRaCiaf$ for some
  $\cst{\cK}\geq1$.
\end{remark}
Assuming $\mRaCir\vee\mRaCia$ is negligible compared to $\mRaCiaf$ let
us reformulate the upper bound \eqref{a:re:ub:e} in terms of noise
levels ${\dr\cRaLo[\nlImS]}\nlImS$ and
${\dr\cRaLo[\nlOpS]}\oSoS\nlOpS$, respectively.  Keeping the minimax
optimal choice $\ymDiCi\wedge\xmDiCi$ for the dimension parameter
(c.f. \cref{m:re:ub}) in mind we note that $\ymDiCi$ and $ \xmDiCi$
depend only on $\nlImS$ and $\nlOpS$, respectively. Therefore, we
eventually choose collections $\cK_{\nlImS}$ and $\cK_{\nlOpS}$
depending on $\nlImS$ respectively $\nlOpS$ only, and set
$\cK := \cK_{\nlImS} \cap \cK_{\nlOpS}$,
$\cRaLo[\nlImS] := \cRaLo[{\cK_{\nlImS}}]$ and
$\cRaLo[\nlOpS] := \cRaLo[{\cK_{\nlOpS}}]$ where
$\lv \cK \rv \leq \lv \cK_{\nlImS} \rv \wedge \lv \cK_{\nlOpS} \rv$
and hence $\cRaLo \leq \cRaLo[\nlImS] \wedge
\cRaLo[\nlOpS]$. Exploiting $2[\ymRaCia \vee \xmRaCia] \geq \mRaCiaf$,
the next result is a direct consequence of \cref{a:re:ub} and its
proof is omitted.
\begin{theorem}\label{separate:ub}  Let
	$\cK := \cK_{\nlImS} \cap \cK_{\nlOpS}$,
	$\cRaLo[\nlImS] := \cRaLo[{\cK_{\nlImS}}]$ and
	$\cRaLo[\nlOpS] := \cRaLo[{\cK_{\nlOpS}}]$. Assume there exists a 
  $\cst{\cK}\geq1$ with $\mRaCir\vee\mRaCia\leq \cst{\cK}\mRaCiaf$ for
  all $\wSoS \in \CwSo$, $\wEvS\in\CwEv$.  Then for each
  $\alpha \in (0,1)$ with
  $\oacst[2]{\alpha}:=\rSo+\rEv(10
  \lcst{\alpha/2}+30\lcst[2]{\alpha/2}+10)$ it follows
  \begin{equation*}
    \forall \acst{}\in[\oacst{\alpha},\infty):
    \sup_{(\wSoS,\wEvS) \in \CwSo\times\CwEv}
    \RiT{\teFi[\cK,]{\alpha/2}}{\wrcSo,\wrcEv}
    {\oSoS,\acst{}\cst{\cK}[\ymRaCia \vee\xmRaCia]}\leq\alpha.
  \end{equation*}
\end{theorem}
\begin{remark}\label{a:re:ub:rem:2}
  Let us briefly discuss the choice of $\cK$ in the homoscedastic case
  as in \cref{m:rem:mrt2}. Considering the signal detection task it is
  easily seen that for all $\wEvS\in\CwEv$ and $\wSoS\in\CwSo$ the
  minimax optimal dimension parameter $\ymDiCi$ as in
  \eqref{indirect:radius} is never larger than ${\nlIm^{-4}}$.
  Therefore, exploiting the natural choice $\cK=\nset{{\nlIm^{-4}}}$
  the factor $\cRaLo$ is of order $|\log\nlIm|^{1/4}$. However, in
  many cases it is sufficient to aggregate over a polynomial grid
  $\cK=\set{2^l, l\in\nset{{4|\log_2\nlIm|}}}$. Obviously, $\cRaLo$ is
  then of order $(\log|\log\nlIm|)^{1/4}$. For a goodness-of-fit task
  the upper bound for the minimax optimal dimension parameter
  $\ymDiCi\wedge\xmDiCi$ can further be improved by exploiting the
  knowledge of $\oSoS$. More precisely, since
  $\qFPr{\sonlImS/\wEvS^2}\geq \qFPr{\sonlImS} \geq \nlIm^4 \Di +
  \nlOp^4 \qFPr{\oSoS^2}$ and $\VnormInf{\wSoS} \leq 1$, any $k\in\Nz$
  such that $\nlOp^4\qFPr{\oSoS^2} \geq 1$ is an upper bound for the
  dimension parameter. For the goodness-of-fit task
  $\oSoS =\Nsuite{ j^{-\oSoP}}$ as considered in
  \cref{illustration:adaptive} below, the upper bound is of order
  $\nlOp^{-4}$, which results in the natural choice
  $\cK=\nset{{\nlIm^{-4}} }\cap \nset{ \nlOp^{-4}} = \cK_{\nlIm} \cap
  \cK_{\nlOp}$ and an adaptive factor
  $\lv \log \nlIm\rv^{1/4} \wedge \lv \log \nlOp \rv^{1/4}$. However,
  since a polynomial grid
  $\cK=\set{2^l, l\in\nset{{4\lv \log_2 \nlIm \rv}} \cap \nset{ 4 \lv
      \log_2 \nlOp \rv}}$ is again sufficient, $\cRaLo$ is of order
  $(\log|\log\nlIm|)^{1/4} \wedge (\log|\log\nlOp|)^{1/4}=
  \cRaLo[\nlIm] \wedge \cRaLo[\nlOp]$.
\end{remark}
\begin{illustration}[homoscedastic case]
  \label{illustration:adaptive}
  Considering the typical smoothness and ill-posedness assumptions of
  \cref{illustration:indirect} and a polynomial grid
  $\cK=\set{2^l, l\in\nset{{4\lv \log_2 \nlIm \rv}} \cap \nset{ 4 \lv
      \log_2 \nlOp \rv}}$ as discussed in \cref{a:re:ub:rem} it holds
  $\mRaCia\leq\cst{\cK}\mRaCiaf$ for some $\cst{\cK}\geq1$ uniformly
  for all $\wSoP\in[\wSoP_\star,\wSoP^\star]$.  Moreover, for mildly
  ill-posed models $\mRaCir$ as in
  \eqref{a:re:ub:e} is negligible compared to $\mRaCia$,
  i.e. uniformly for all $\wSoP\in[\wSoP_\star,\wSoP^\star]$ it
  holds $\mRaCir\leq\cst{\cK}\mRaCia$ for some $\cst{\cK}\geq1$, since
  the conditions \eqref{negligible} are fulfilled for
  $\cRaLo\leq\cRaLo[\nlIm] \wedge \cRaLo[\nlOp]$ with
  $\cRaLo[\nlIm] = (\log \lv \log\nlIm\rv)^{1/4}$ and
  $ \cRaLo[\nlOp] = (\log \lv \log\nlOp\rv)^{1/4}$. Furthermore, the
  constant $\cst{\cK}$ can be chosen uniformly for all sufficiently
  small noise levels.  In the severely ill-posed case $\mRaCir$,
  $\mRaCia$ and $\mRaCiaf$ are all of the same order and the adaptive
  factors have no effect on the rate. We present the resulting rate of
  testing $\ymRaCia \vee \xmRaCia$ in terms of the originals noise
  levels $\nlIm$ and $\nlOp$ for both the signal detection and the
  goodness-of-fit task $\oSoS
  =\Nsuite{ j^{-\oSoP}}$, in the table below.\\[3ex]
  \centerline{\begin{tabular}{@{}ll|ll@{}} \toprule
                \multicolumn{4}{@{}l@{}}{Order of $\ymRaCia
                \vee \xmRaCia$ \smaller for  $\teFi[\cK,]{\alpha/2}$ with $\cK=\set{2^l, l\in\nset{{4\lv \log_2 \nlIm \rv}} \cap \nset{ 4 \lv \log_2 \nlOp \rv}}$} \\
		\midrule
	 	$\wSoS$  & $\wEvS$  & $\ymRaCia $   & $\xmRaCia$ \\
		{\smaller (smoothness)}  & { \smaller (ill-posedness)}& {\smaller $\oSoS \in \lp^2$}& {\smaller$\oSoS =\Nsuite{ j^{-\oSoP}}, 4\oSoP- 4 \wEvP  < 1$} \\
		\midrule
		$\Nsuite{ j^{-\wSoP}}$ & $\Nsuite{j^{-\wEvP}}$ & $\lb  (\log \lv \log \nlIm \rv)^{\frac{1}{4}} \nlIm\rb^{\frac{4\wSoP}{4\wSoP+4\wEvP+1}}$  &$\lb (\log \lv \log \nlOp \rv)^{\frac{1}{4}}\nlOp\rb^{\frac{4\wSoP}{4\wSoP+4(\wEvP-\oSoP)+1}}$  \\
		$\Nsuite{e^{-j^{2\wSoP}}}$ &  $\Nsuite{j^{-\wEvP}}$ & $|\log \nlIm |^{\frac{4 \wEvP+1}{8\wSoP}}(\log \lv \log \nlIm \rv)^{\frac{1}{4}}\nlIm $  & $|\log \nlOp |^{\frac{4\wEvP-4 \oSoP+1}{8\wSoP}}(\log \lv \log \nlOp \rv)^{\frac{1}{4}}\nlOp  $ \\
		$\Nsuite{ j^{-\wSoP}}$ & $\Nsuite{e^{-j^{2\wEvP}}}$ &$|\log  \nlIm|^{- \frac{\wSoP}{2\wEvP}}$ & $|\log  \nlOp|^{-\frac{\wSoP}{2\wEvP}}$  \\
		\bottomrule
              \end{tabular}}\\[2ex]
            In case of super smoothness $\wSoS=\Nsuite{e^{-j^{2\wSoP}}}$ and
            mild ill-posedness  (see \cref{illustration:indirect}) the minimax optimal
            dimension parameter   is of order $|\log
            \nlIm|^{1/(2\wSoP)}$ and $|\log\nlIm|^{{1}/2\wSoP)}
            \wedge|\log\nlOp|^{1/(2\wSoP)}$ for the signal detection task  and the
            goodness-of-fit task, respectively, which suggests
            for $\wSoP\in[\wSoP_\star,
            \wSoP^\star]$ a polynomial  grid $\cK := \set{2^l,
              l\in\nsetB{{\tfrac{1}{2 \wSoP_\star}  \log_2 \lv \log \nlIm \rv}}
              \cap \nsetB{ \tfrac{1}{2 \wSoP_\star} \log_2 \lv \log \nlOp \rv}}$
            and an adaptive factor  $\cRaLo\leq\cRaLo[\nlIm] \wedge
            \cRaLo[\nlOp]$
            with   $\cRaLo[\nlIm] = (\log \log\lv \log\nlIm\rv)^{1/4}$ and
            $ \cRaLo[\nlOp] = (\log \log\lv \log\nlOp\rv)^{1/4}$. Indeed, in
            this situation there exists a $\cst{\cK}\geq1$  such that 
            $\mRaCir\vee\mRaCia\leq \cst{\cK}\mRaCiaf$ uniformly
            for all $\wSoP\in[\wSoP_\star,
            \wSoP^\star]$ and for all sufficiently small noise levels. We present the
            resulting rate of testing in terms of the
            originals noise levels $\nlIm$ and $\nlOp$ for both the signal
            detection and the goodness-of-fit task  in the table below. \\[3ex]
            \centerline{\begin{tabular}{@{}ll|ll@{}} \toprule
                          \multicolumn{4}{@{}l@{}}{Order of $\ymRaCia
                          \vee \xmRaCia$  \smaller for $\teFi[\cK,]{\alpha/2}$ with $\cK=\set{2^l,
                          l\in\nsetB{{\tfrac{1}{2 \wSoP_\star}  \log_2 \lv \log \nlIm \rv}}
                          \cap \nsetB{ \tfrac{1}{2 \wSoP_\star} \log_2 \lv \log \nlOp \rv}}$} \\
                          \midrule
                          $\wSoS$  & $\wEvS$  & $\ymRaCia $   & $\xmRaCia$ \\
                          {\smaller (smoothness)}  & { \smaller (ill-posedness)}& {\smaller $\oSoS \in \lp^2$}& {\smaller$\oSoS =\Nsuite{ j^{-\oSoP}}, 4\oSoP- 4 \wEvP  < 1$} \\
                          \midrule
                          $\Nsuite{e^{-j^{2\wSoP}}}$ &  $\Nsuite{j^{-\wEvP}}$ & $|\log \nlIm |^{\frac{4 \wEvP+1}{8\wSoP}}(\log\log \lv \log \nlIm \rv)^{\frac{1}{4}}\nlIm $  & $|\log \nlOp |^{\frac{4\wEvP-4 \oSoP+1}{8\wSoP}}(\log\log \lv \log \nlOp \rv)^{\frac{1}{4}}\nlOp  $ \\
                          \bottomrule
                        \end{tabular}}%
\end{illustration}
\begin{remark}In a direct sequence space model (i.e.
  $\wEvS \equiv 1$) adaptation to the radius $\rSo \in \IR_+$ of the
  ellipsoid $\rwcSo$ without a loss is not possible (c.f.
  \cite{Baraud2002} Section 6.3). However, due to \cref{separate:ub}
  adaptation to a bounded interval $[\rSo_\star, \rSo^\star]$ without
  further loss is possible even in an indirect sequence space
  model. Precisely, for
  $\oacst[2]{\alpha} := \rSo^\star
  +\rEv(10\lcst{\alpha/2}+30\lcst[2]{\alpha/2}+10)$ a similar result
  to \cref{separate:ub} holds where the supremum is additionally taken
  over all $\rSo \in [\rSo_\star, \rSo^\star]$.
\end{remark}
The next proposition states conditions under which a deterioration of
the minimax testing radius
$\ymRaCi[{\dr\cRaLo[\nlImS]}\,\nlImS] \vee
\xmRaCi[{\dr\cRaLo[\nlOpS]}\,\oSoS\nlOpS]$ as in
\eqref{indirect:radius} by factors $\cRaLo[\nlImS]$ and
$\cRaLo[\nlOpS]$ are unavoidable for adaptation over $\CrwcSo$. Recall
that the parameter $\wEvS$ in $\rwcEv$ is still assumed to be known.%
\begin{proposition}\label{prop:condition_adaptive} Let
  $\alpha\in(0,1)$, $\cRaLo[\nlImS], \cRaLo[\nlOpS] \geq1$,
  $\wEvS\in\CwEv$.  Assume a collection of $N \in \Nz$ sequences
  $\lcb \wSoS^j\in \CwSo, j \in \nset{N} \rcb \subset \CwSo$ where
  $\mDiwSo[j]:= \yomDiCiaj \wedge \xomDiCiaj$ and
  $\mRawSo[j]:= \yomRaCiaj \vee \xomRaCiaj $ satisfies
  \begin{inparaenum}[i]\renewcommand{\theenumi}{\dr{\normalfont\rmfamily{\dgrau\bfseries(C\arabic{enumi})}}} 
  \item\label{pro:ca:c2}$\mDiwSo[j]\leq\mDiwSo[l]$ and
  \item\label{pro:ca:c3}$\lb \cRaLo[\nlImS] \vee \cRaLo[\nlOpS]
    \rb^4\smRawSo[j]{2}\leq \smRawSo[l]{2}$ for all
    $j,l \in \nset{N}$, $j < l$.  Moreover, suppose there exists a
    $c_{\alpha}>0$ such that
  \item\label{pro:ca:c1}
    $\exp(c_\alpha \lb \cRaLo[\nlImS] \vee \cRaLo[\nlOpS] \rb^4)\leq
    N\alpha^2$.
  \end{inparaenum} 
  If $\eta\in(0,1]$ satisfies
  $\eta \leq
  \inf_{j\in\nset{N}}\smRawSo[j]{-2}(\rqFPr[{\mDiwSo[j]}]{\cRaLo[\nlImS]^2
    \snlImS/\wEvS^2}\vee \rqFPr[{\mDiwSo[j]}]{\cRaLo[\nlOpS]^2
    \snlOpS \oSoS^2/\wEvS^2} \wedge(\wSo[{\mDiwSo[j]}]^j)^2)$, then
  with
  $\uacst[2]{\alpha} := \eta(r \wedge \sqrt{\log(1 + \alpha^2)}\wedge
  \sqrt{c_\alpha})$
  \begin{equation}\label{pro:ca:e1}
    \forall A \in [0,\underline{A}_{\alpha}]:
    \inf_{\teF} \sup_{(\wSoS, \wEvS) \in \CwSo \times \lcb \wEvS \rcb}
    \RiT{\teF}{\rwcSo,\rwcEv}{\oSoS,A
      [\ymRaCi[{\dr\cRaLo[\nlImS]}\,\nlImS]
      \vee  \xmRaCi[{\dr\cRaLo[\nlOpS]}\,\oSoS\nlOpS] }\geq 1-\alpha, 
  \end{equation}
  i.e.  $\cRaLo[\nlImS]$ respectively $\cRaLo[\nlOpS]$ are lower
  bounds for the minimal adaptive factors over $\CrwcSo$.
\end{proposition}
\begin{proof}[Proof of \cref{prop:condition_adaptive}] The proof
  relies on the reduction scheme and notation used in the proof of
  \cref{lowerbound}. Let us define
  ${\onlImS}^{2,\cRaLo[]} := \cRaLo[\nlImS]^2 \snlImS +
  \cRaLo[\nlOpS]^2 \oSoS^2 \snlOpS$ For each $j\in\nset{N}$ we
  introduce $\tSoS^{j} \in\lp^2$ with $\tSo[\Di]^{j} =0$ for
  $\Di > \mDiwSo[j]$,
  \begin{equation*}
    \tSo[\Di]^{j} := \frac{\mRawSo[j]\sqrt{ \zeta \eta} }{
      \rqFPr[{\mDiwSo[j]}]{\cRaLo[\nlImS]^2\snlImS/\wEvS^2}
      \vee\rqFPr[{\mDiwSo[j]}]{\cRaLo[\nlOpS]^2\snlOpS\oSoS^2/\wEvS^2}}
    \frac{{\onlIm[\Di]}^{2,\cRaLo[]}}{\wEv[\Di]^2},
    \quad \Di \in \nset{\mDiwSo[j]}\quad\text{and}\quad\zeta: = r
    \wedge \sqrt{\log(1+\alpha^2)}\wedge \sqrt{c_{\alpha}},
  \end{equation*}
  where
  $\tSoS^{j} \in \wrcSo[\wSoS^j]\cap
  \lp[\underline{A}_\alpha{\mRawSo[j]}]^2$ follows from arguing line by line as in
  the proof of \cref{lowerbound}.  Considering the uniform mixture
  measure over the vertices of a hypercube
  $\FuVg{1,j} := {2^{-\mDiwSo[j]}} \sum_{\tau \in
    \set{\pm1}^{\mDiwSo[j]}} \FuVg[\nlImS,\nlOpS]{\oSoS+\tSoS^{j\tau},
    \wEvS}$, the uniform mixture measure
  $\FuVg{1}:= \frac{1}{N} \sum_{j=1}^N\FuVg{1,j} $ and
  $\FuVg{0} := \FuVg[\nlImS,\nlOpS]{\oSoS, \wEvS}$ are supported on
  the alternative and the null hypothesis, respectively. The joint
  distribution $\tFuVg{0}$ and $\tFuVg{1}$ of the reparametrised
  observation $(\tyObS:=\yObS-\oSoS\xObS,\xObS)$ given $\FuVg{0}$ and
  $\FuVg{1}$, respectively, still satisfies
  $\text{TV}(\FuVg{1}, \FuVg{0})=\text{TV}(\tFuVg{1}, \tFuVg{0})$ and,
  hence, from \eqref{reduction} it follows
  \begin{equation}\label{reduction:2}
    \inf_{\teF} \sup_{(\wSoS, \wEvS) \in \CwSo \times \lcb \wEvS \rcb}
    \RiT{\teF}{\rwcSo,\rwcEv}{\oSoS,A[\ymRaCi[{\dr\cRaLo[\nlImS]}\,\nlImS]
      \vee  \xmRaCi[{\dr\cRaLo[\nlOpS]}\,\oSoS\nlOpS]] }
    \geq 1 - \sqrt{\frac{\chi^2(\tFuVg{1},\tFuVg{0})}{2}}.
  \end{equation}
  Arguing as in the proof of \cref{lowerbound} and applying
  \cref{lemma:adaptive_lower_bound} in the appendix, we obtain
  \begin{align*}
    \chi^2(\tFuVg{1},\tFuVg{0})  \leq N^{-2}\sum_{j\in\nset{N}}\sum_{l\in\nset{N}}
    \exp \lb \tfrac{1}{2} \qFPr[{\mDiwSo[j]\wedge
    \mDiwSo[l]}]{\wEvS^2\tSoS^{j}\tSoS^{l}/\sonlImS}\rb- 1.
  \end{align*}
  Exploiting successively \ref{pro:ca:c2} and the definition of $\eta$,
  for $ j\leq l$ it holds
  $\qFPr[{\mDiwSo[j]\wedge\mDiwSo[l]}]{\wEvS^2\tSoS^{j}\tSoS^{l}/\sonlImS}
  =\qFPr[{\mDiwSo[j]}]{\wEvS^2\tSoS^{j}\tSoS^{l}/\sonlImS} \leq
  \frac{2 \smRawSo[j]{2} \smRawSo[l]{2} \zeta^2\eta^2 \lb
    \cRaLo[\nlImS] \vee \cRaLo[\nlOpS] \rb^4}{
    \qFPr[{\mDiwSo[l]}]{\cRaLo[\nlImS]^2\snlImS/\wEvS^2} \vee
    \qFPr[{\mDiwSo[l]}]{\cRaLo[\nlOpS]^2\snlOpS\oSoS^2/\wEvS^2}} \leq
  2 \zeta^2\lb \cRaLo[\nlImS] \vee \cRaLo[\nlOpS] \rb^4
  \frac{\smRawSo[j]{2}}{\smRawSo[l]{2}}$, and, hence,
  $\chi^2(\tFuVg{1},\tFuVg{0}) \leq \tfrac{1}{N} \exp \lb
  \zeta^2\lb \cRaLo[\nlImS] \vee \cRaLo[\nlOpS] \rb^4 \rb+
  \tfrac{N(N-1)}{N^2}\exp \lb \zeta^2\rb- 1$ due to
  \ref{pro:ca:c3}. Combining the last bound, the definition of $\zeta$
  and \ref{pro:ca:c1} implies
  $\chi^2(\tFuVg{1},\tFuVg{0}) \leq 2\alpha^2$, which together
  with \eqref{reduction:2} shows \eqref{pro:ca:e1} and completes the
  proof.
\end{proof}
\begin{remark}\label{rem:condition_adaptive}
  Let us briefly discuss the conditions \ref{pro:ca:c2} -
  \ref{pro:ca:c1} in \cref{prop:condition_adaptive}.  Under
  \ref{pro:ca:c2} and \ref{pro:ca:c3} the class $\CwSo$ is rich enough
  to make adaptation unavoidable, i.e. it contains enough
  distinguishable elements $\wSoS$ resulting in significantly
  different radii.  \ref{pro:ca:c1} is a bound for the maximal size of
  an unavoidable adaptation factor. Lastly, the condition on $\eta$ is
  similar to the balancing condition \eqref{assump1} (see the remark
  below \cref{lowerbound}) in the nonadaptive case, but now needs to
  hold uniformly for all elements $\wSoS\in\CwSo$.
\end{remark}
Next, we demonstrate how to use \cref{prop:condition_adaptive} in the
homoscedastic case, when $\CwSo$ is nontrivial with respect to
polynomial decay.
\begin{theorem} \label{unavoidable:lower} In the homoscedastic case
  $\nlImS = \Nsuite[j]{\nlIm}$ and $\nlOpS = \Nsuite[j]{\nlOp}$ with
  $\cRaLo[\nlImS]^4 = \log \lv \log \nlIm \rv$ and
  $\cRaLo[\nlOpS]^4 = \log\lv \log \nlOp \rv$, let
  $ \wEvS:= \Nsuite{j^{-\wEvP}} $, $\wEvP>0$, and
  $\CwSo = \lcb\Nsuite[j]{ j^{-2\wSoP}}, \wSoP \in [\wSoP_\star,
  \wSoP^\star] \rcb$ for $\wSoP_\star < \wSoP^\star$. Assume for all
  $ \wSoS \in \CwSo$ either
  \begin{inparaenum}[i]\renewcommand{\theenumi}{\dr{\normalfont\rmfamily{\dgrau\bfseries(A\arabic{enumi})}}} 	
  \item \label{eps_larger}$\ymRaCia \geq \xmRaCia $ or
  \item \label{sigma_larger}$\ymRaCia \leq \xmRaCia $.
  \end{inparaenum}
  For $\alpha \in (0,1)$ set
  $\uacst[2]{\alpha} := \eta(r \wedge \sqrt{\log(1 + \alpha^2)}\wedge
  1/2)$ with $\eta$ as in \cref{prop:condition_adaptive}.
  \begin{Liste}[]
  \item[\mylabel{oo:e}{\dgrau\upshape\bfseries{A1}}]
    There exists ${\nlIm}_{\alpha}\in(0,1)$
    such that for all
    $(\nlIm, \nlOp) \in [0, {\nlIm}_\alpha] \times [0,1)$ satisfying
    \ref{eps_larger} 
    \begin{equation*}
      \forall A \in [0,\underline{A}_{\alpha}]:
      \inf_{\teF} \sup_{(\wSoS, \wEvS) \in \CwSo \times \lcb \wEvS \rcb} 
      \RiT{\teF}{\rwcSo,\rwcEv}{\oSoS,A\ymRaCia  }\geq 1-\alpha. 
    \end{equation*} 
  \item[\mylabel{oo:s}{\dgrau\upshape\bfseries{A2}}] 
    Let $ \oSoP< \tfrac{1}{4}+  \wEvP $. There is ${\nlOp}_{\alpha} \in (0,1)$
    such that for all $(\nlIm, \nlOp) \in [0, 1) \times
    [0,{\nlOp}_\alpha]$ satisfying \ref{sigma_larger} 
    \begin{equation*}
      \forall A \in [0,\underline{A}_{\alpha}]:
      \inf_{\teF} \sup_{(\wSoS, \wEvS) \in \CwSo \times \lcb \wEvS \rcb}
      \RiT{\teF}{\rwcSo,\rwcEv}{\oSoS,A\xmRaCia }\geq 1-\alpha. 
    \end{equation*} 
  \end{Liste}
\end{theorem}
\begin{remark} Let us briefly comment on the assumptions and results
  of \cref{unavoidable:lower}. Considering mildly ill-posed models as
  in \cref{illustration:adaptive} the adaptive factors
  $\cRaLo[\nlImS] = (\log \lv \log \nlIm \rv)^{1/4}$ and
  $\cRaLo[\nlOpS] = (\log \lv \log \nlOp \rv)^{1/4}$ visible in the
  resulting rates are minimal adaptive factors due to
  \cref{unavoidable:lower}. We distinguish two cases \ref{eps_larger}
  and \ref{sigma_larger} insuring, respectively, that either all rates
  in $\cRaLo[\nlImS]\nlIm$ or all rates in $\cRaLo[\nlOpS]\oSoS\nlOp$
  are dominant, and hence we exclude mixed situations which we are not
  interested in here. Note that, \cref{unavoidable:lower} provides
  intrinsically asymptotic results, since for each $\alpha \in (0,1)$
  the noise level have to be sufficiently small. Moreover, for any
  $\nlOp \asymp \nlIm^c $ with $c \in \IR_+$, the factors
  $\cRaLo[\nlIm]$ and $\cRaLo[\nlOp]$ are of the same order anyway,
  and hence asymptotically only the cases \ref{eps_larger} and
  \ref{sigma_larger} appear. The additional restriction
  $ \oSoP - \wEvP < 1/4$ allows us to apply
  \cref{prop:condition_adaptive}. In case $\oSoP - \wEvP \geq 1/4$ the
  minimax rate $\xmRaCi$ does not depend on the smoothness parameter
  $\wSoP$ (\cref{illustration:indirect}), and hence, \ref{pro:ca:c1}
  in \cref{prop:condition_adaptive} is violated.  In this situation,
  however, is $\xmRaCi$ (almost) parametric, i.e., \ref{eps_larger}
  will typically govern the behaviour of the minimax rate. Finally,
  \cref{unavoidable:lower} covers only combinations of ordinary
  smoothness and mildly ill-posedness. For ordinary smoothness and
  severely ill-posedness, the optimal dimension parameter does not
  depend on the smoothness parameter, compare
  \cref{illustration:indirect}, hence, as usual our testing procedure
  is automatically adaptive to $\CrwcSo$, which is also reflected in
  the table in \cref{illustration:adaptive}.  The remaining case of
  super smoothness and mildly ill-posedness is considered separately in
  \cref{unavoidable:lower:exp} below.
\end{remark}
\begin{proof}[Proof of \cref{unavoidable:lower}] 
  We only prove \ref{oo:e}, the arguments for \ref{oo:s} are similar
  (simply replace $\wEvP$ by $\wEvP - \oSoP$) and thus omitted.  We
  apply \cref{prop:condition_adaptive}. Let $ \wSoS \in \CwSo$, due to
  \ref{eps_larger} the rate is given by
  $\ymRaCia \vee \xmRaCia = \ymRaCia$, which implies in turn
  $\ymDiCia \wedge \xmDiCia = \ymDiCia$ due to \cref{re:argmin}. For
  $\wSoS = \Nsuite[j]{ j^{-2\wSoP}}$ and
  $\wEvS = \Nsuite[j]{ j^{-\wEvP}}$ setting
  $e(\wSoP):= \tfrac{4 \wSoP}{4 \wSoP + 4 \wEvP + 1}$ we have
  $\cst[-1]{} \leq \ymRaCia / \lb \nlIm \cRaLo[\nlImS] \rb^{e(\wSoP)}
  \leq \cst{}$ and
  $\cst[-1]{} \leq \ymDiCia / \lb \nlIm \cRaLo[\nlImS]
  \rb^{e(\wSoP)/\wSoP} \leq\cst{}$ for some constant $\cst{} > 0$.
  Let $e_\star := e(\wSoP_\star)$, $e^\star := e(\wSoP^\star)$ and
  $\lcb e(\wSoP_l):=e^\star - (l-1) \Delta : l \in \nset{N}\rcb$,
  where $\Delta := \frac{e^\star - e_\star}{N}$ and
  $N:= \frac{e^\star - e_\star}{4} \frac{\lv \log(\cRaLo[\nlImS]
    \nlIm))\rv}{\log\cRaLo[\nlImS]}$.  The collection of $N$ sequences
  is now given by
  $\lcb (j^{-2 \wSoP_l})_{j \in \IN} : l \in \nset{N} \rcb$. Under
  \ref{eps_larger} it remains to check \ref{pro:ca:c2} -
  \ref{pro:ca:c1} for
  $\cRaLo[\nlImS] = (\log \lv \log \nlIm \rv)^{1/4}$ (setting
  $\cRaLo[\nlOpS] = 1$). Since by construction
  $\sup_{j<l} \cRaLo[\nlImS]^2 \tfrac{\mRawSo[j]}{\mRawSo[l]} \leq
  \cRaLo[\nlImS]^2 \cst[2]{} \lb \cRaLo[\nlImS] \nlIm\rb^{\Delta} \to
  0$ and $\sup_{j<l} \tfrac{\mDiwSo[j]}{\mDiwSo[l]} \to 0$ as
  $\nlIm \to 0$, \ref{pro:ca:c2} and \ref{pro:ca:c3} hold for all
  $\nlIm$ small enough. Finally \ref{pro:ca:c1} follows from
  $c_{\alpha} \cRaLo[\nlImS]^2 - \log(N) \leq 2 \log(\alpha)$ for
  $\nlIm$ small enough, since
  $\frac{1}{2} \cRaLo[\nlImS]^2 - \log(N) \to -\infty$ as
  $\nlIm \to 0$, which completes the proof.
\end{proof}
In the homoscedastic super smooth case (see
\cref{illustration:adaptive}) when $\CwSo$ is nontrivial with respect
to an exponential decay, which is obviously more restrictive than a
polynomial decay considered in \cref{unavoidable:lower}, the testing
radius of the indirect-$\max$-test $\teFi[\cK,]{\alpha}$ in
\eqref{de:it:ad} with polynomial grid
$\cK=\set{2^l, l\in\nsetB{{\tfrac{1}{2 \wSoP_\star} \log_2 \lv \log
      \nlIm \rv}} \cap \nsetB{ \tfrac{1}{2 \wSoP_\star} \log_2 \lv
    \log \nlOp \rv}}$ features a deterioration compared to the minimax
rate by factors $\cRaLo[\nlIm] = (\log \log\lv \log\nlIm\rv)^{1/4}$
and $ \cRaLo[\nlOp] = (\log \log\lv \log\nlOp\rv)^{1/4}$ only.
Applying \cref{prop:condition_adaptive} we show these are minimal
adaptive factors in this more restrictive situation.%
\begin{theorem}\label{unavoidable:lower:exp}
  In the homoscedastic case with
  $\cRaLo[\nlImS]^4 := \log\log \lv \log \nlIm \rv$ and
  $\cRaLo[\nlOpS]^4 := \log\log\lv \log \nlOp \rv$, let
  $ \wEvS:= \Nsuite{j^{-\wEvP}} $, $\wEvP>0$, and
  $\CwSo = \lcb\Nsuite[j]{ e^{-j^{2\wSoP}}}, \wSoP \in [\wSoP_\star,
  \wSoP^\star] \rcb$ for $\wSoP_\star < \wSoP^\star$. Consider
  \ref{eps_larger} or \ref{sigma_larger}, and $\uacst[2]{\alpha}$ for
  $\alpha \in (0,1)$ as in \cref{unavoidable:lower}.
  \begin{Liste}[]
  \item[\mylabel{so:e}{\dgrau\upshape\bfseries{(A1)}}] There exists
    ${\nlIm}_{\alpha}\in(0,1)$ such that for all
    $(\nlIm, \nlOp) \in [0, {\nlIm}_\alpha] \times [0,1)$ satisfying
    \ref{eps_larger} 
    \begin{equation*}
      \forall A \in [0,\underline{A}_{\alpha}]:
      \inf_{\teF}  \sup_{(\wSoS, \wEvS) \in \CwSo \times \lcb \wEvS \rcb} 
      \RiT{\teF}{\rwcSo,\rwcEv}{\oSoS,A\omRaCia[{\dr \cRaLo[\nlImS]}
        \nlImS]  }
      \geq 1-\alpha. 
    \end{equation*} 
  \item[\mylabel{so:s}{\dgrau\upshape\bfseries{(A2)}}] Let $ \oSoP< \tfrac{1}{4}+  \wEvP $. There is ${\nlOp}_{\alpha} \in (0,1)$
    such that for all $(\nlIm, \nlOp) \in [0, 1) \times
    [0,{\nlOp}_\alpha]$ satisfying \ref{sigma_larger}  \linebreak
    \begin{equation*}\label{unavoidable:sigma:exp}
      \forall A \in [0,\underline{A}_{\alpha}]:
      \inf_{\teF}  \sup_{(\wSoS, \wEvS) \in \CwSo\times \lcb \wEvS \rcb} 
      \RiT{\teF}{\rwcSo,\rwcEv}{\oSoS,A\omRaCia[{\dr \cRaLo[\nlOpS]}
        \nlOpS]  }
      \geq 1-\alpha. 
    \end{equation*} 
  \end{Liste}
\end{theorem}
\begin{proof}[Proof of \cref{unavoidable:lower:exp}] 
  We only prove \ref{so:e}, the arguments for \ref{so:s} are similar
  (simply replace $\wEvP$ by $\wEvP - \oSoP$) and thus
  omitted. We apply \cref{prop:condition_adaptive} similarly to the proof of \cref{unavoidable:lower}.  Due to
  \ref{eps_larger} for $\wSoS = \Nsuite[j]{ e^{-j^{2\wSoP}}}$ and
  $\wEvS = \Nsuite[j]{ j^{-\wEvP}}$ setting
  $e(\wSoP):= \tfrac{4 \wEvP+1}{4 \wSoP }$ we have
  $\cst[-1]{} \leq \nlIm \cRaLo[\nlImS]\lb\log\nlIm\rb^{e(\wSoP)}/
  \ymRaCia \leq \cst{}$ and
  $\cst[-1]{} \leq \lb\log\nlIm\rb^{e(\wSoP)/(4\wEvP+1)} / \ymDiCia
  \leq\cst{}$ for some constant $\cst{} > 0$.  Let
  $e_\star := e(\wSoP^\star)$, $e^\star := e(\wSoP_\star)$ and
  $\lcb e(\wSoP_l):=e_\star + (l-1) \Delta : l \in \nset{N}\rcb$,
  where $\Delta := \frac{e^\star - e_\star}{N}$ and
  $N:= \frac{e^\star - e_\star}{4} \frac{\lv \log(
    \nlIm)\rv}{\log(\cRaLo[\nlIm])}$.  The collection of $N$ sequences
  is now given by
  $\lcb (e^{-j^{2 \wSoP_l}})_{j \in \IN} : l \in \nset{N} \rcb$.
  Under \ref{eps_larger} it remains to check \ref{pro:ca:c2} -
  \ref{pro:ca:c1} for $\cRaLo[\nlImS]^4 = \log\log \lv \log \nlIm \rv$
  (setting $\cRaLo[\nlOpS] = 1$). Since by construction
  $\sup_{j<l} \cRaLo[\nlImS]^2 \tfrac{\mRawSo[j]}{\mRawSo[l]} \leq
  \cRaLo[\nlImS]^2\cst[2]{} \lv \log \nlIm\rv^{\Delta} \to 0$ and
  $\sup_{j<l} \tfrac{\mDiwSo[j]}{\mDiwSo[l]} \to 0$ as $\nlIm \to 0$,
  \ref{pro:ca:c2} and \ref{pro:ca:c3} hold for $\nlIm$ small
  enough. Lastly \ref{pro:ca:c1} follows from
  $c_{\alpha} \cRaLo[\nlImS]^2 - \log(N) \leq 2 \log(\alpha)$ for
  $\nlIm$ small enough, since
  $\frac{1}{2} \cRaLo[\nlImS]^2 - \log(N) \to -\infty$ as
  $\nlIm \to 0$, \ref{pro:ca:c1}, which completes the proof.
\end{proof}
The indirect-$\max$-test $\teFi[\cK,]{\alpha}$ in \eqref{de:it:ad} is
eventually adaptive to $\CrwcSo$ when $\CwSo$ is nontrivial with
respect to polynomial or exponential decay (see
\cref{illustration:adaptive}). However, the indirect test in
\eqref{de:it} makes explicite use of $\wEvS$ and thus a prior
knowledge of the class $\rwcEv$ is required, which the direct test in \eqref{de:dt}
avoids. Therefore, next, we consider its adaptation both to a family of
smoothness classes $\CrwcSo$ and a family of ill-posedness classes
$\CrwcEv$.
\subsection{Adaptation to smoothness and ill-posedness}\label{adaptation:direct}
\paragraph{Direct testing procedure.} In this section we consider the
adaptation of the direct test in \eqref{de:dt} to families of
alternatives $\CrwcSo$ and $\CrwcEv$. Given $\alpha\in(0,1)$ and a finite
collection $\cK\in\Nz$ define the $\max$-test statistic with
Bonferroni levels and the corresponding $\max$-test%
\begin{equation}\label{de:dt:ad}
  \teSd[\cK,]{\alpha}:=\max_{\Di\in\cK}(\teSd[\Di,]{\alpha/|\cK|})
  \quad\text{ and }\quad
  \teFd[\cK,]{\alpha}:=\Ind{\set{\teSd[\cK,]{\alpha}>0}},
\end{equation}
which is a level-$\alpha$-test due to \eqref{m:di:mrt:e1} in
\cref{m:di:mrt}.  Its testing radius faces a deterioration compared to
the optimal direct testing radius derived in \cref{m:di:mrt} due to
the Bonferroni aggregation which we formalise next. Analogously to 
\eqref{direct:radius}, for each $x_{\mbullet} \in \IR^\IN$ let us define the
minimum over the collection $\cK$ 
\begin{equation}\label{dt:radius:re}
  (\mRaCdr[x_{\mbullet}])^2:=\min\nolimits_{\cK}(\wEvS^{-2}\mFPrS{x_{\mbullet}^2}\vee\swSoS)
\end{equation}
and the minimum and minimiser over $\cK$, respectively,
\begin{align}\label{dt:radius:ad}
  (\mRaCda[x_{\mbullet}])^2:=
  \min\nolimits_{\cK}(\wEvS^{-2}\rqFPrS{x_{\mbullet}^2}\vee\swSoS)
  \;\text{ and} \;  \mDiCda[x_{\mbullet}]:= \argmin\nolimits_{\cK}(\wEvS^{-2}\rqFPrS{x_{\mbullet}^2}\vee\swSoS).
\end{align}
We first present an upper bound in terms of the reparametrised noise
level $\sonlImS=\snlImS+\oSoS^2\snlOpS$ and the adaptive factor
$\dr\cRaLo:=(1\vee\log|\cK|)^{1/4}$.  The upper bound consists of the
two terms $\mRaCdr$ and $\mRaCda$, defined by replacing $x_{\mbullet}$ with
${\dr\cRaLo^2}\onlImS$ and ${\dr\cRaLo}\onlImS$ in
\eqref{it:radius:re} and \eqref{it:radius:ad},
respectively.
\begin{proposition}\label{a:re:ub:d} For $\alpha\in(0,1)$ define
  $\oacst{\alpha}:=\rSo+\rEv(5\lcst{\alpha/2}+15\lcst[2]{\alpha/2}+5)$. Then
  \begin{equation}\label{a:re:ub:d:e}
    \forall \acst{}\in[\oacst{\alpha},\infty):\ \sup_{(\wSoS, \wEvS) \in \CwSo \times \CwEv}
    \RiT{\teFd[\cK,]{\alpha/2}}{\wrcSo,\wrcEv}{\oSo,\acst{}
      [\mRaCdr\vee\mRaCda]}\leq\alpha.
  \end{equation}
\end{proposition}
\begin{proof}[Proof of \cref{a:re:ub:d}]The proof follows along the
  lines of the proof of \cref{a:re:ub} using \cref{m:in:ub} rather
  than \cref{m:re:ub}, and we omit the details.
\end{proof}
\begin{remark}
  The upper bound \eqref{a:re:ub:d:e} in \cref{a:re:ub:d} consist of
  two terms similar to the upper bound \eqref{a:re:ub:e} in
  \cref{a:re:ub}. In contrast to $\mRaCir$ in \eqref{a:re:ub:e} the
  term $\mRaCdr$ in \eqref{a:re:ub:d:e} is not negligible compared to
  $\mRaCda$ if the effective noise level $\oSoS\snlOpS$ determines the
  rate. Similar to \cref{a:re:ub:rem} consider
  $\mDid:=\mDiCda[{\dr\cRaLo^2}\,\nlImS]$ as in \eqref{dt:radius:ad}
  replacing $x_{\mbullet}$ by ${\dr\cRaLo^2}\nlImS$. Then we have
  $\mRaCdr[{\dr\cRaLo^2}\,\nlImS]\leq
  \cst[2]{\cK}\mRaCda[{\dr\cRaLo}\,\nlImS]$ for any $\cst[2]{\cK}>0$
  with
  ${\dr\cRaLo^2}(\mDid)^{-1/2}\rqFPr[\mDid]{{\dr\cRaLo^2}\,\snlImS}\leq
  \mFPr[\mDid]{{\dr\cRaLo^4}\,\snlImS}\leq
  \cst[2]{\cK}\rqFPr[\mDid]{{\dr\cRaLo^2}\,\snlImS}$, which is
  satisfied whenever
  \begin{align}
    \label{dt:negligible}
    {\dr\cRaLo^2}\leq \cst{\cK} \sqrt{\mDid} \qquad \text{
    and } \qquad \sqrt{\mDid} \mFPr[\mDid]{ \snlImS}
    \leq  \cst{\cK}\;\rqFPr[\mDid]{ \snlImS}.
  \end{align}
  However, we have
  $\mRaCdr[{\dr\cRaLo^2}\,\oSoS\nlOpS]\leq\mRaCda[{\dr\cRaLo^2}\,\oSoS\nlOpS]$
  where in the homoscedastic case
  $\mRaCdr[{\dr\cRaLo^2}\,\oSoS\nlOpS]$ and
  $\mRaCda[{\dr\cRaLo^2}\,\oSoS\nlOpS]$ are of the same order.
  Finally, below we select the collection $\cK$ such that uniformly
  for all $\wSoS \in \CwSo $ and $\wEvS \in \CwEv$ we get
  $\mRaCda[{\dr\cRaLo^2}\,\oSoS\nlOpS]\vee\mRaCda[{\dr\cRaLo}\,\nlImS]\leq
  \cst{\cK}[\mRaCd[{\dr\cRaLo^2}\,\oSoS\nlOpS]\vee\mRaCd[{\dr\cRaLo}\,\nlImS]]$
  for some $\cst{\cK}\geq1$.
\end{remark}
Assuming $\mRaCdr\vee\mRaCda$ is negligible compared to
$\ymRaCda\vee\xmRaCda$ we reformulate the upper bound
\eqref{a:re:ub:e} in terms of the noise levels
${\dr\cRaLo[\nlImS]}\nlImS$ and ${\dr\cRaLo[\nlOpS]^2}\oSoS\nlOpS$,
respectively. Similar to \cref{separate:ub} we choose collections
$\cK_{\nlImS}$ and $\cK_{\nlOpS}$ depending on $\nlImS$ respectively
$\nlOpS$ only, and set $\cK := \cK_{\nlImS} \cap \cK_{\nlOpS}$,
$\cRaLo[\nlImS] := \cRaLo[{\cK_{\nlImS}}]$ and
$\cRaLo[\nlOpS] := \cRaLo[{\cK_{\nlOpS}}]$ where
$\lv \cK \rv \leq \lv \cK_{\nlImS} \rv \wedge \lv \cK_{\nlOpS} \rv$
and hence $\cRaLo \leq \cRaLo[\nlImS] \wedge \cRaLo[\nlOpS]$. The next
result is an immediate consequence of \cref{a:re:ub:d} and its proof
is omitted.
\begin{theorem}\label{separate:ub:d}  Let
	$\cK := \cK_{\nlImS} \cap \cK_{\nlOpS}$,
	$\cRaLo[\nlImS] := \cRaLo[{\cK_{\nlImS}}]$ and
	$\cRaLo[\nlOpS] := \cRaLo[{\cK_{\nlOpS}}]$. Assume there exists a $\cst{\cK}\geq1$ with
  $[\mRaCdr\vee\mRaCda]\leq \cst{\cK}[\ymRaCda\vee\xmRaCda]$ for all
  $\wSoS \in \CwSo$, $\wEvS\in\CwEv$.  Then for each
  $\alpha \in (0,1)$ with
  $\oacst[2]{\alpha}:=\rSo+\rEv(5
  \lcst{\alpha/2}+15\lcst[2]{\alpha/2}+5)$
  \begin{equation*}
    \forall \acst{}\in[\oacst{\alpha},\infty): \sup_{(\wSoS,\wEvS) \in \CwSo\times\CwEv}
    \RiT{\teFd[\cK,]{\alpha/2}}{\wrcSo,\wrcEv}{\oSoS,\acst{}\cst{\cK}
      [\ymRaCda \vee\xmRaCda]}\leq\alpha.
\end{equation*}  
\end{theorem}
\begin{remark}Comparing the upper bounds in \cref{separate:ub:d,separate:ub} for the testing radius of the
  indirect-  and direct-$\max$-test, respectively, there is an additional adaptive factor
  $\cRaLo[\nlOpS]$ in the
  term  $\xmRaCda$. However, $\xmRaCd$ is generally much
  larger than the minimax optimal $\xmRaCi$  and the additional
  deterioration by an factor $\cRaLo[\nlOpS]$ is negligible compared
  with it. On the other hand,  the
  term  $\ymRaCda$ is typically of the optimal order $\ymRaCia$.
\end{remark}
\begin{illustration}[homoscedastic case]\label{illustration:adaptive:d} 
  Considering the typical smoothness and ill-posedness assumptions of
  \cref{illustration:direct} and a polynomial grid
  $\cK=\set{2^l, l\in\nset{{4\lv \log_2 \nlIm \rv}} \cap \nset{ 4 \lv
      \log_2 \nlOp \rv}}$ similar to \cref{illustration:adaptive} it
  holds
  $\mRaCda[{\dr\cRaLo[\nlOpS]^2}\,\oSoS\nlOpS]\vee
  \mRaCda[{\dr\cRaLo[\nlImS]}\,\nlImS]\leq
  \cst{\cK}[\mRaCd[{\dr\cRaLo[\nlOpS]^2}\,\oSoS\nlOpS]\vee
  \mRaCd[{\dr\cRaLo[\nlImS]}\,\nlImS]]$
  for some $\cst{\cK}\geq1$ uniformly for all
  $\wSoP\in[\wSoP_\star,\wSoP^\star]$ and
  $\wEvP\in[\wEvP_\star,\wEvP^\star]$.  Moreover, for mildly ill-posed
  models as considered below we have
  $\mRaCdr[{\dr\cRaLo^2}\,\nlImS]\leq
  \cst{\cK}\mRaCda[{\dr\cRaLo[\nlImS]}\,\nlImS]$ for some
  $\cst{\cK}\geq1$ uniformly for all
  $\wSoP\in[\wSoP_\star,\wSoP^\star]$ and
  $\wEvP\in[\wEvP_\star,\wEvP^\star]$, since the conditions
  \eqref{dt:negligible} are fulfilled for
  $\cRaLo[\nlIm] = (\log \lv \log\nlIm\rv)^{1/4}\geq\cRaLo$. On the
  other hand we use
  $\mRaCdr[{\dr\cRaLo^2}\,\nlOpS]\leq\mRaCda[{\dr\cRaLo[\nlOp]^2}\,\nlOpS]$
  with $\cRaLo[\nlOp] = (\log \lv \log\nlOp\rv)^{1/4}\geq\cRaLo$.
  Furthermore, the constant $\cst{\cK}$ can be chosen uniformly for
  all sufficiently small noise levels.  In the severely ill-posed case
  $\mRaCdr$, $\mRaCda$ and $\mRaCd[{\dr\cRaLo}\,\onlImS]$ are all of
  the same order and the adaptive factors have no effect on the
  rate. We present the resulting rate of testing
  $\ymRaCda \vee \xmRaCda$ in terms of the originals noise levels
  $\nlIm$ and $\nlOp$ for both the signal detection and the
  goodness-of-fit task $\oSoS =\Nsuite{ j^{-\oSoP}}$ in the table
  below. We shall stress that the order $\ymRaCda$ and $\ymRaCia$ (see
  \cref{illustration:adaptive}) of the direct and indirect max-test
  coincide and, hence the direct test features a deterioration by a
  minimal adaptive factor in $\cRaLo[\nlImS]$ only. However, the order of
  $\xmRaCda$ is much slower than the optimal order $\xmRaCia$.\\[3ex] 
 \centerline{
   \begin{tabular}{@{}ll|ll@{}}      
     \toprule
     \multicolumn{4}{@{}l@{}}{Order of 
     $\ymRaCda \vee \xmRaCda$ \smaller for
     $\teFd[\cK,]{\alpha/2}$ with $\cK=\set{2^l, l\in\nset{{4\lv
     \log_2 \nlIm \rv}} \cap \nset{ 4 \lv \log_2 \nlOp \rv}}$}  \\  
     \midrule
     $\wSoS$  & $\wEvS$  & $\ymRaCda$   & $\xmRaCda$ \\
     {\smaller (smoothness)}  & { \smaller (ill-posedness)}& {\smaller $\oSoS \in \lp^2$}&  {\smaller$\oSoS =\Nsuite{ j^{-\oSoP}}$} \\
     \midrule
     $\Nsuite{ j^{-\wSoP}}$ & $\Nsuite{j^{-\wEvP}}$ & $\lb
                                                      (\log
                                                      \lv
                                                      \log
                                                      \nlIm
                                                      \rv)^{\frac{1}{4}}
                                                      \nlIm
                                                      \rb^{\frac{4\wSoP}{4\wSoP+4\wEvP+1}}$
                                        &  $\lb (\log \lv
                                          \log \nlOp
                                          \rv)^{\frac{1}{2}}
                                          \nlOp  \rb^{\frac{\wSoP}{\wSoP+\wEvP}}$  \\
     $\Nsuite{e^{-j^{2\wSoP}}}$ &  $\Nsuite{j^{-\wEvP}}$ &
                                                       $|\log \nlIm
                                                           |^{\frac{4
                                                           \wEvP+1}{8\wSoP}}
                                                           (\log \lv
                                                           \log \nlIm
                                                           \rv)^{\frac{1}{4}}
                                                           \nlIm$
                                        &  $|\log \nlOp |^{\frac{\wEvP}{2\wSoP}} \lb\log \lv \log \nlOp \rv\rb^{\frac{1}{2}} \nlOp $ \\
     $\Nsuite{ j^{-\wSoP}}$
              & $\Nsuite{e^{-j^{2\wEvP}}}$
                         &$|\log  \nlIm|^{-\frac{\wSoP}{2\wEvP}}$
                                        &  $|\log  \nlOp|^{-\frac{\wSoP}{2\wEvP}}$  \\
     \bottomrule
   \end{tabular}}\\[2ex]
 In case of super smoothness $\wSoS=\Nsuite{e^{-j^{2\wSoP}}}$, $\wSoP\in[\wSoP_\star,\wSoP^\star]$, and
 mild ill-posedness as in \cref{illustration:adaptive} we
 consider   a polynomial
 grid $\cK := \set{2^l,
   l\in\nsetB{{\tfrac{1}{2 \wSoP_\star}  \log_2 \lv \log \nlIm \rv}}
   \cap \nsetB{ \tfrac{1}{2 \wSoP_\star} \log_2 \lv \log \nlOp \rv}}$
 and an adaptive factor  $\cRaLo\leq\cRaLo[\nlIm] \wedge \cRaLo[\nlOp]$ with   $\cRaLo[\nlIm] = (\log \log\lv \log\nlIm\rv)^{1/4}$ and
 $ \cRaLo[\nlOp] = (\log \log\lv \log\nlOp\rv)^{1/4}$. 
 We present the
 resulting rate of testing in terms of the
 originals noise levels $\nlIm$ and $\nlOp$ for both the signal
 detection and the goodness-of-fit task  in the table below. \\[3ex]
 \centerline{\begin{tabular}{@{}ll|ll@{}} \toprule
               \multicolumn{4}{@{}l@{}}{Order of $\ymRaCda
               \vee \xmRaCda$  \smaller for $\teFd[\cK,]{\alpha/2}$ with $\cK=\set{2^l,
               l\in\nsetB{{\tfrac{1}{2 \wSoP_\star}  \log_2 \lv \log \nlIm \rv}}
               \cap \nsetB{ \tfrac{1}{2 \wSoP_\star} \log_2 \lv \log \nlOp \rv}}$} \\
               \midrule
               $\wSoS$  & $\wEvS$  & $\ymRaCda $   & $\xmRaCda$ \\
               {\smaller (smoothness)}  & { \smaller (ill-posedness)}& {\smaller $\oSoS \in \lp^2$}& {\smaller$\oSoS =\Nsuite{ j^{-\oSoP}}, 4\oSoP- 4 \wEvP  < 1$} \\
               \midrule
               $\Nsuite{e^{-j^{2\wSoP}}}$ &  $\Nsuite{j^{-\wEvP}}$ & $|\log \nlIm |^{\frac{4 \wEvP+1}{8\wSoP}}(\log\log \lv \log \nlIm \rv)^{\frac{1}{4}}\nlIm $  & $|\log \nlOp |^{\frac{\wEvP}{2\wSoP}}(\log\log \lv \log \nlOp \rv)^{\frac{1}{2}}\nlOp  $ \\
               \bottomrule
              \end{tabular}}%
\end{illustration}
The adaptive factors $\cRaLo[\nlIm] = ( \log\lv \log\nlIm\rv)^{1/4}$
and $\cRaLo[\nlIm] = (\log \log\lv \log\nlIm\rv)^{1/4}$ given in
\cref{illustration:adaptive:d} are minimal due to
\cref{unavoidable:lower,unavoidable:lower:exp},
respectively. Therefore, it is an unavoidable cost to pay for an
adaptation to $\CrwcSo[\CwSo]$ whenever it is nontrivial with
respect to a polynomial or exponential decay.  Lastly, we give
conditions under which a deterioration of the minimax testing radius
$\ymRaCi[{\dr\cRaLo[\nlImS]}\,\nlImS] \vee
\xmRaCi[{\dr\cRaLo[\nlOpS]}\,\oSoS\nlOpS]$ as in
\eqref{indirect:radius} by factors $\cRaLo[\nlImS]$ and
$\cRaLo[\nlOpS]$ is unavoidable for adaptation to the ill-posedness of the model, i.e. a class $\CrwcEv$. Note that the sequence $\wSoS$
is fixed in the next proposition.
\begin{proposition}\label{prop:condition_adaptive:v}
  Let $\alpha\in(0,1)$, $\cRaLo[\nlImS], \cRaLo[\nlOpS] \geq1$,
  $\wSoS\in\CwSo$. Assume a collection of $N \in \Nz$ sequences
  $\lcb \wEvS^j\in \CwEv, j \in \nset{N} \rcb \subset \CwEv$ where
  $\mDiwSo[j]:= \yomDiCivj \wedge \xomDiCivj$ and
  $\mRawSo[j]:= \yomRaCivj \vee \xomRaCivj $ satisfies
  \begin{inparaenum}[i]\renewcommand{\theenumi}{\dr{\normalfont\rmfamily{\dgrau\bfseries(D\arabic{enumi})}}}
  \item\label{pro:da:c2}$\mDiwSo[j]\leq\mDiwSo[l]$ and
  \item\label{pro:da:c3}$\lb \cRaLo[\nlImS] \vee \cRaLo[\nlOpS] \rb^4
    \tfrac{ \smRawSo[j]{2}}{\smRawSo[l]{2}} \leq
    \tfrac{ \qFPr[{\mDiwSo[j]}]{\cRaLo[\nlImS ]^2\snlImS/(\wEvS^l)^2}
      \vee \qFPr[{\mDiwSo[j]}]{\cRaLo[\nlOpS]^2 \snlOpS
        \oSoS^2/(\wEvS^l)^2} }
    { \qFPr[{\mDiwSo[j]}]{\cRaLo[\nlImS ]^2\snlImS/(\wEvS^l \wEvS^j)}
      \vee \qFPr[{\mDiwSo[j]}] {\cRaLo[\nlOpS]^2 \snlOpS \oSoS^2 /(\wEvS^l \wEvS^j)} }$ for  $j,l \in \nset{N}$, $j < l$.
    Suppose  $c_{\alpha}>0$ with 
  \item\label{pro:da:c1}
    $\exp(c_\alpha \lb \cRaLo[\nlImS] \vee \cRaLo[\nlOpS] \rb^4)\leq
    N\alpha^2$.  \end{inparaenum} If $\eta\in(0,1]$ satisfies
  $\eta \leq
  \inf_{j\in\nset{N}}\smRawSo[j]{-2}(\rqFPr[{\mDiwSo[j]}]{\cRaLo[\nlImS]^2
    \snlImS/(\wEvS^j)^2}\vee \rqFPr[{\mDiwSo[j]}]{\cRaLo[\nlOpS]^2
    \snlOpS \oSoS^2/(\wEvS^j)^2}\wedge\swSo[{\mDiwSo[j]}])$ and
  $\uacst[2]{\alpha} := \eta(r \wedge \sqrt{\log(1 + \alpha^2)}\wedge
  \sqrt{c_\alpha})$
  \begin{equation}\label{pro:da:e1}
    \forall A \in [0,\uacst{\alpha}]:
    \inf_{\teF} \sup_{(\wSoS, \wEvS) \in \lcb \wSoS \rcb \times \CwEv}
    \RiT{\teF}{\rwcSo,\rwcEv}{\oSoS,A
      [\ymRaCi[{\dr\cRaLo[\nlImS]}\,\nlImS] 
      \vee  \xmRaCi[{\dr\cRaLo[\nlOpS]}\,\oSoS\nlOpS]}\geq 1-\alpha,
  \end{equation}
  i.e.  $\cRaLo[\nlImS]$ respectively $\cRaLo[\nlOpS]$ are lower
  bounds for the minimal adaptive factors over $\CrwcEv$.
\end{proposition}
\begin{proof}[Proof of \cref{prop:condition_adaptive:v}]
  The proof follows along the lines of the proof of
  \cref{prop:condition_adaptive}, which relies on the reduction scheme
  and notation used in the proof of \cref{lowerbound}. Let us define
  ${\onlImS}^{2,\cRaLo[]} := \cRaLo[\nlImS]^2 \snlImS +
  \cRaLo[\nlOpS]^2\oSoS^2\snlOpS$ For each $j\in\nset{N}$ we introduce
  $\tSoS^{j} \in\lp^2$ with $\tSo[\Di]^{j} =0$ for $\Di >
  \mDiwSo[j]$,%
  \begin{equation}\label{pro:da:p1}
    \tSo[\Di]^{j} := \frac{\mRawSo[j] \sqrt{ \zeta \eta }}{
      \rqFPr[{\mDiwSo[j]}]{\cRaLo[\nlImS]^2\snlImS/(\wEvS^j)^2}
      \vee
      \rqFPr[{\mDiwSo[j]}]{\cRaLo[\nlOpS]^2\snlOpS\oSoS^2/(\wEvS^j)^2}}
    \frac{{\onlIm[\Di]}^{2,\cRaLo[]}}{(\wEv[\Di]^j)^2},
    \; \Di \in \nset{\mDiwSo[j]}\,\text{ and }\,\zeta: = r
    \wedge \sqrt{\log(1+\alpha^2)}\wedge \sqrt{c_{\alpha}},
  \end{equation}
  where
  $\tSoS^{j} \in \wrcSo[\wSoS]\cap
  \lp[\underline{A}_\alpha{\mRawSo[j]}]^2$ arguing line by line as in
  the proof of \cref{lowerbound}.  Considering the uniform mixture
  measure over the vertices of a hypercube
  $\FuVg{1,j} := {2^{-\mDiwSo[j]}} \sum_{\tau \in
    \set{\pm1}^{\mDiwSo[j]}} \FuVg[\nlImS,\nlOpS]{\oSoS+\tSoS^{j\tau},
    \wEvS^j}$, the uniform mixture measure
  $\FuVg{1}:= \frac{1}{N} \sum_{j=1}^N\FuVg{1,j} $ and
  $\FuVg{0} :=  \frac{1}{N} \sum_{j=1}^N \FuVg[\nlImS,\nlOpS]{\oSoS, \wEvS^j}$ are supported on
  the alternative and the null hypothesis, respectively. The joint
  distribution $\tFuVg{0}$ and $\tFuVg{1}$ of the reparametrised
  observation $(\tyObS:=\yObS-\oSoS\xObS,\xObS)$ given $\FuVg{0}$ and
  $\FuVg{1}$, respectively, still satisfies
  $\text{TV}(\FuVg{1}, \FuVg{0})=\text{TV}(\tFuVg{1}, \tFuVg{0})$ and,
  hence, from \eqref{reduction} it follows
  \begin{equation} 
    \label{reduction:3}
    \inf_{\teF} \sup_{(\wSoS, \wEvS) \in \lcb \wSoS \rcb \times \CwEv}
    \RiT{\teF}{\rwcSo,\rwcEv}{\oSoS,A[\yomRaCiaa \vee  \xomRaCiaa] }
    \geq  1 - \text{TV}(\tFuVg{1}, \tFuVg{0})
  \end{equation}
  Since $\tyObS$ is a sufficient statistic for $\tSoS$, the
 conditional distribution of $\xObS$ given $\tyObS$ does not depend
 on $\tSoS$. Hence, $\text{TV}(\tFuVg{1}, \tFuVg{0})$ can be bounded by the total variation distance and thus the $\chi^2$-divergence of the mixture over the marignal distributions $\FuVg[\onlImS]{\wEvS^j \tSoS^\tau}$ of $\tyObS$. 
Applying
  \cref{lemma:adaptive_lower_bound} in the appendix, we obtain
  \begin{align*}
    \text{TV}(\tFuVg{1},\tFuVg{0})  \leq N^{-2}\sum_{j\in\nset{N}}\sum_{l\in\nset{N}}
    \exp \lb \tfrac{1}{2} \qFPr[{\mDiwSo[j]\wedge\mDiwSo[l]}]{\wEvS^j
    \wEvS^l \tSoS^{j}\tSoS^{l}/\sonlImS}\rb- 1.
  \end{align*}
  Exploiting \ref{pro:da:c2} and the definition of $\eta$ for
  $1\leq j\leq l\leq N$ it holds
  $ \qFPr[{\mDiwSo[j]\wedge\mDiwSo[l]}]{\wEvS^j \wEvS^l
    \tSoS^{j}\tSoS^{l}/\sonlImS} =\qFPr[{\mDiwSo[j]}]{\wEvS^j \wEvS^l
    \tSoS^{j}\tSoS^{l}/\sonlImS} \leq \frac{2\zeta^2 (\cRaLo[\nlImS]
    \vee \cRaLo[\nlOpS])^4 \smRawSo[j]{2}}{\smRawSo[l]{2}} \frac{
    \qFPr[{\mDiwSo[j]}]{\cRaLo[\nlImS ]^2\snlImS/(\wEvS^l \wEvS^j)}
    \vee \qFPr[{\mDiwSo[j]}]{\cRaLo[\nlOpS]^2 \snlOpS \oSoS^2
      /(\wEvS^l
      \wEvS^j)}}{\qFPr[{\mDiwSo[j]}]{\cRaLo[\nlImS]^2\snlImS/(\wEvS^j)^2}
    \vee\qFPr[{\mDiwSo[j]}]{\cRaLo[\nlOpS]^2\snlOpS\oSoS^2/(\wEvS^{j})^2}}$,
  and, hence,
  $\text{TV}(\tFuVg{1},\tFuVg{0}) \leq \tfrac{1}{N} \exp \lb
  \zeta^2\lb \cRaLo[\nlImS] \vee \cRaLo[\nlOpS] \rb^4 \rb+
  \tfrac{N(N-1)}{N^2}\exp \lb \zeta^2\rb- 1$ due to
  \ref{pro:da:c3}. Combining the last bound, the definition of $\zeta$
  and \ref{pro:da:c1} implies
  $\text{TV}(\tFuVg{1},\tFuVg{0}) \leq 2\alpha^2$, which together
  with \eqref{reduction:3} shows \eqref{pro:da:e1} and completes the
  proof.
\end{proof}
The conditions \ref{pro:da:c2} - \ref{pro:da:c1} in the last assertion
are similar to \ref{pro:ca:c2} - \ref{pro:ca:c1} in
\cref{prop:condition_adaptive}, which are briefly discussed in
\cref{rem:condition_adaptive}. Let us demonstrate how to use
\cref{prop:condition_adaptive:v} in the homoscedastic case, when
$\CwEv$ is nontrivial with respect to polynomial decay. The next
result is similar to \cref{unavoidable:lower} where we distinguished
cases \ref{eps_larger} and \ref{sigma_larger} insuring roughly that
either $\ymRaCia$ or $\xmRaCia$ is dominant. In the next result we
only consider a case similar to \ref{eps_larger}, since in the
opposite case the obtainable lower bound does not match the upper
bound of the direct-$\max$-test.
\begin{theorem}\label{unavoidable:lower:nu}
  In the homoscedastic case $\nlImS = \Nsuite[j]{\nlIm}$ and
  $\nlOpS = \Nsuite[j]{\nlOp}$ with
  $\cRaLo[\nlImS]^4 = \log \lv \log \nlIm \rv$, let
  $\wSoS = \Nsuite{j^{-\wSoP}}$, $\wSoP > 1/2$, and
  $\CwEv = \lcb \Nsuite[j]{ j^{-\wEvP}}, \wEvP \in [\wEvP_\star,
  \wEvP^\star] \rcb$ for $\wEvP_\star < \wEvP^\star$.  For
  $\alpha \in (0,1)$ set
  $\uacst[2]{\alpha} := \eta(r \wedge \sqrt{\log(1 + \alpha^2)}\wedge
  1/2)$ with $\eta$ as in \cref{prop:condition_adaptive:v}. There
  exists ${\nlIm}_{\alpha}\in(0,1)$ such that for all
  $(\nlIm, \nlOp) \in [0, {\nlIm}_\alpha] \times [0,1)$ with
  $\yomRaCiaa \geq \xomRaCiaa $ for all $ \wEvS \in \CwEv$ 
  \begin{equation}\label{unavoidable:eps:nu}
    \forall A \in [0,\uacst{\alpha}]:
    \inf_{\teF}  \sup_{(\wSoS, \wEvS) \in \lcb \wSoS \rcb\times\CwEv} 
    \RiT{\teF}{\rwcSo,\rwcEv}{\oSoS,A\yomRaCiaa}\geq 1-\alpha. 
  \end{equation} 
\end{theorem}
\begin{proof}[Proof of \cref{unavoidable:lower:nu}]Applying
  \cref{prop:condition_adaptive:v}, the proof follows along the lines of the
  proof of \cref{unavoidable:lower}. By assumption for
  $\wSoS = \Nsuite[j]{ j^{-\wSoP}}$ and
  $\wEvS = \Nsuite[j]{ j^{-\wEvP}}$ setting
  $e(\wEvP):= \tfrac{4 \wSoP}{4\wSoP+4 \wEvP+1 }$ we have
  $\cst[-1]{} \leq \yomRaCiaa / \lb \nlIm \cRaLo[\nlImS]
  \rb^{e(\wEvP)} \leq \cst{}$ and
  $\cst[-1]{} \leq \yomDiCiaa / \lb \nlIm \cRaLo[\nlImS]
  \rb^{e(\wEvP)/(2\wSoP)} \leq \cst{}$ for some constant
  $\cst{} > 0$.  Let $e_\star := e(\wEvP^\star)$,
  $e^\star := e(\wEvP_\star)$ and
  $\lcb e(\wSoP_l):=e_\star + (l-1) \Delta : l \in \nset{N}\rcb$,
  where $\Delta := \frac{e^\star - e_\star}{N}$ and
  $N:= \frac{e^\star - e_\star}{4} \frac{\lv \log(\cRaLo[\nlImS]
    \nlIm))\rv}{\log(\cRaLo[\nlImS])}$.  The collection of $N$
  sequences is now given by
  $\lcb (j^{-2 \wEvP_l})_{j \in \IN} : l \in \nset{N} \rcb$.  It
  remains to check \ref{pro:da:c2} - \ref{pro:da:c1} for
  $\cRaLo[\nlImS]^4 = \log \lv \log \nlIm \rv$ (setting
  $\cRaLo[\nlOpS] = 1$). Since by construction
  $\sup_{j<l} \cRaLo[\nlImS]^2 \tfrac{\smRawSo[j]{}}{\smRawSo[l]{}}
  \leq \cRaLo[\nlImS]^2 \cst[2]{} \lb \cRaLo[\nlImS] \nlIm\rb^{\Delta}
  \to 0$ and $\sup_{j<l} \tfrac{\mDiwSo[j]}{\mDiwSo[l]} \to 0$ as
  $\nlIm \to 0$, \ref{pro:da:c2} and \ref{pro:da:c3} hold for $\nlIm$
  small enough.  Finally \ref{pro:da:c1} follows from
  $c_{\alpha} \cRaLo[\nlImS]^2 - \log(N) \leq 2 \log(\alpha)$ for
  $\nlIm$ small enough, since
  $\frac{1}{2} \cRaLo[\nlImS]^2 - \log(N) \to -\infty$ as
  $\nlIm \to 0$, which completes the proof.
\end{proof}
\begin{remark}In the homoscedastic case with prior known ordinary
  smoothness of the alternative and mild ill-posedness with
  unknown degree $\wEvP\in[\wEvP_\star, \wEvP^\star]$
  an adaptive factor of order
  $\cRaLo[\nlImS] = (\log \lv \log \nlIm \rv)^{1/4}$ is unavoidable
  due to \cref{unavoidable:lower:nu} and the direct-$\max$-test attains
  the minimax-optimal rate with a minimal adaptive factor (see
  \cref{illustration:adaptive:d}). Analogously to
  \cref{unavoidable:lower:exp}, if the alternative is known to be super
  smooth an adaptive factor of order
  $\cRaLo[\nlImS] = (\log\log \lv \log \nlIm \rv)^{1/4}$ is
  unavoidable for adaptation to unknown ill-posedness, and hence the adaptive factor in the testing rate of
  the direct-$\max$-test (see \cref{illustration:adaptive:d}) is also
  minimal. However, the optimality of the direct-$\max$-test is only
  guaranteed if the rate in terms of $\oSoS\nlOp$ is negligible compared to the
  rate in $\nlIm$. The order of an optimal rate in the opposite case
  is still an open question, when the ill-posedness of the model is unknown. 
\end{remark}

%
%
%
%
\appendix 
\section{Appendix}\label{appendix}
In this section we gather technical results and their proofs.
\begin{lemma}\label{re:argmin} Let $\aS\in\pRz^\Nz$ and  $\bS$, $\ceS\in\pRz^\Nz$ be
  monotonically nonincreasing and nondecreasing, respectively. For
  $\dRai[\bS]:=\min(\aS\vee \bS)$ and
  $\dRai[\ceS]:=\min(\aS\vee \ceS)$ it follows
  $\dRai[\bS]\vee\dRai[\ceS]=\dRai[\bS\vee
    \ceS]:=\min(\aS\vee \bS\vee
  \ceS)$. Moreover,
  $\dDii[\bS]:=\argmin(\aS\vee \bS)$ and
  $\dDii[\ceS]:=\argmin(\aS\vee \ceS)$ satisfy
  $\dDii[\bS]\wedge\dDii[\ceS]=\dDii[\bS\vee
    \ceS]:=\argmin(\aS\vee \bS\vee \ceS)$.
\end{lemma}
\begin{figure}[h]
	\centering
	\begin{tikzpicture}
	\draw[->] (-0.2,0) -- (10.2,0) node[right] {$\Di$};
	\draw[->] (0,-1.2) -- (0,5.4);
	\draw[color=red] (0,4.0) to  (10,0.5) node[right] {$\aS$};
	\draw[color=blue] (0,0.2) to  (10,5.2) node[right] {$\bS$};
	\draw[color=teal] (0,2.5) -- (10,4) node[right] {$\ceS$};
	\draw[color=teal, dotted] (3,-0.2) -- (3,2.95);
	\draw[color=teal] (3,-0.2) -- (3,0.2);
	\node[color=teal] at (3,-0.5) {$\dDii[\ceS]$};
	\node[draw,color=red] at (3,-1.2) {$\dDii[\bS\vee \ceS]=\dDii[\bS]\wedge\dDii[\ceS]$};
	\draw[color=blue, dotted] (4.47,-0.2) -- (4.47,2.4) ;
	\draw[color=blue] (4.47,-0.2) -- (4.47,0.2) ;
	\node[color=blue] at (4.47,-0.5) {$\dDii[\bS]$};
	\end{tikzpicture}
	\caption{Illustration of \cref{re:argmin}}
\end{figure} 
\begin{proof}[Proof of \cref{re:argmin}]We start the proof with the
  observation that
  $\dRai[\bS]\vee\dRai[\ceS]\leq\dRai[\bS\vee\ceS]$ and hence also
  $\dDii[\bS]\wedge\dDii[\ceS]\geq\dDii[\bS\vee\ceS]$, since in the nontrivial case
  $\dDii[\bS\vee \ceS]>1$ for each
  $\Di<\dDii[\bS\vee \ceS]$ it holds
  $\aS[\Di]\vee\bS[\Di]\vee\ceS[\Di]=\aS[\Di]>\dRai[\bS\vee
  \ceS]\geq\dRai[\bS]\vee\dRai[\ceS]$.   Moreover,
  there are
  $\odDii[\bS\vee\ceS],\odDii[\bS],\odDii[\ceS]\in\Nz\cup\{\infty\}$
  such that
  $\nset{\dDii[\bS],\odDii[\bS]}:=\minset(\aS\vee
  \bS):=\{\Di\in\Nz:a_{\Di}\vee b_{\Di}\leq a_{j}\vee b_{j},\;\forall j\in\Nz\}$,
  $\nset{\dDii[\ceS],\odDii[\ceS]}=\minset(\aS\vee\ceS)$ and
  $\nset{\dDii[\bS\vee \ceS],\odDii[\bS\vee\ceS]}=\minset(\aS\vee\bS\vee\ceS)$, where
  $\nset{\dDii[\bS],\odDii[\bS]}\subset\nset{\dDii[\bS\vee\ceS],\odDii[\bS\vee \ceS]}$ or
  $\nset{\dDii[\ceS],\odDii[\ceS]}\subset\nset{\dDii[\bS\vee\ceS],\odDii[\bS\vee \ceS]}$, because in
  the nontrivial case $\odDii[\bS\vee \ceS]<\infty$ for $\Di:=\odDii[\bS\vee \ceS]+1$ it holds
  $[\dRai[\bS]\vee \dRai[\ceS]]\leq\dRai[\bS\vee\ceS]<\bS[\Di]\vee \ceS[\Di]=[\aS[\Di]\vee\bS[\Di]]\vee
  [\aS[\Di]\vee\ceS[\Di]]$. Without loss of generality let
  $\nset{\dDii[\bS],\odDii[\bS]}\subset\nset{\dDii[\bS\vee
    \ceS],\odDii[\bS\vee \ceS]}$. Note that
  there is
  $\Di\in\nsetro{\dDii[\bS],\dDii[\bS\vee
    \ceS]}$ if and only if
  $\dRai[\bS]<\aS[\Di]=\aS[\Di]\vee\bS[\Di]\leq [\aS[\Di]\vee\bS[\Di]\vee\ceS[\Di]]=\dRai[\bS\vee \ceS]$, which in
  turn implies
  $\dRai[\bS\vee \ceS]=\aS[\Di]\vee\ceS[\Di]$ for all
  $\Di\in\nsetro{\dDii[\bS],\dDii[\bS\vee
    \ceS]}$.  We distinguish the two
  cases \begin{inparaenum}[i]\renewcommand{\theenumi}{\dgrau\rm(\alph{enumi})}\item\label{re:argmin:c1}
    $ \dRai[\bS]=\dRai[\bS\vee \ceS]$
    and \item\label{re:argmin:c2}
    $ \dRai[\bS]<\dRai[\bS\vee\ceS]$. \end{inparaenum} Firstly, consider
  \ref{re:argmin:c1}
  $ \dRai[\bS]=\dRai[\bS\vee \ceS]$, and
  hence $\dDii[\bS]=\dDii[\bS\vee \ceS]$. Consequently,
  $\dRai[\bS]\vee\dRai[\ceS]\leq\dRai[\bS\vee\ceS]=\dRai[\bS]=\dRai[\bS]\vee\dRai[\ceS]$
  and
  $\dDii[\bS]\wedge\dDii[\ceS]\geq\dDii[\bS\vee\ceS]=\dDii[\bS]=\dDii[\bS]\wedge\dDii[\ceS]$. Consider
  secondly \ref{re:argmin:c2}
  $ \dRai[\bS]<\dRai[\bS\vee \ceS]$, and
  hence $\dDii[\bS]>\dDii[\bS\vee \ceS]$,
  where $\dRai[\bS\vee \ceS]=\aS[\Di]\vee\ceS[\Di]$
  for all
  $\Di\in\nsetro{\dDii[\bS\vee\ceS],\dDii[\bS]}$. Moreover for all
  $\Di\in\nset{\dDii[\bS],\odDii[\bS]}$ it holds
  $\aS[\Di]\vee\bS[\Di]=\dRai[\bS]<\dRai[\bS\vee \ceS]=\aS[\Di]\vee\bS[\Di]\vee\ceS[\Di]$, which in turn
  implies
  $\dRai[\bS\vee \ceS]=\ceS[\Di]=\aS[\Di]\vee\ceS[\Di]$  for all
  $\Di\in\nset{\dDii[\bS],\odDii[\bS]}$. Consequently,
  $\aS[\Di]\vee\ceS[\Di]=\dRai[\bS\vee \ceS]$ for all
  $\Di\in\nset{\dDii[\bS\vee \ceS],\odDii[\bS]}$ and
  $\aS[{\odDii[\bS]}]\vee\ceS[{\odDii[\bS]}]\leq\ceS[\Di]=\aS[\Di]\vee\ceS[\Di]$ for all
  $\Di\geq \odDii[\bS]$. Since
  $\dRai[\ceS]\leq \dRai[\bS\vee
  \ceS]<\aS[\Di]=\aS[\Di]\vee\ceS[\Di]$ for all
  $\Di<\dDii[\bS\vee \ceS]$ it follows
  $\dRai[\ceS]=\dRai[\bS\vee \ceS]$ and
  $\dDii[\ceS]=\dDii[\bS\vee \ceS]$, which in
  turn implies the claims
  $\dRai[\bS]\vee\dRai[\ceS]=\dRai[\bS\vee
  \ceS]$ and
  $\dDii[\bS]\wedge\dDii[\ceS]=\dDii[\bS\vee\ceS]$, and completes the proof.
\end{proof}

\begin{lemma}\label{re:qchi} For $\mu_{\mbullet}\in\lp^2$ and $e_{\mbullet} \in\pRz^{\Nz}$ let   $\zObS\sim\FuVg[e_{\mbullet}]{\mu_{\mbullet}}$. For each $\Di\in\Nz$ define $Q_{\Di}:=\sum_{j\in\nset{\Di}}\zOb[j]^2$ and denote by
     $\FuVgQ[e_{\mbullet}]{\mu_{\mbullet},\Di}$ its distribution,
     i.e., $Q_{\Di}\sim\FuVgQ[e_{\mbullet}]{\mu_{\mbullet},\Di}$, and by $\quVgQ[e_{\mbullet}]{\mu_{\mbullet},\Di}{u}$ the $1-u$
    quantile of $\FuVgQ[e_{\mbullet}]{\mu_{\mbullet},\Di}$,
    i.e., $\FuVg[e_{\mbullet}]{\mu_{\mbullet}}\big(Q_{\Di}\leq\quVgQ[e_{\mbullet}]{\mu_{\mbullet},\Di}{u}\big)=1-u$.
    For any $\Di\in\Nz$ and  $u\in(0,1)$  with $\lcst{u}:=\sqrt{|\log
      u|}$ we have
    \begin{align}\nonumber
      \quVgQ[e_{\mbullet}]{\nSoS,\Di}{u}&\leq
      \qFPr{e_{\mbullet}}+2\lcst{u}\,\rqFPr{e^2_{\mbullet}}+2\lcst[2]{u}\mFPr{e^2_{\mbullet}}\\&\label{qchi:e1}\hspace*{20ex}\leq
      \qFPr{e_{\mbullet}}+2\big(\lcst{u}+\lcst[2]{u}\big)\rqFPr{e_{\mbullet}^2}\\\label{qchi:e2}
      \quVgQ[e_{\mbullet}]{\mu_{\mbullet},\Di}{1-u}&\geq \qFPr{e_{\mbullet}}+ \tfrac{4}{5}\qFPr{\mu_{\mbullet}} -2\big(5\lcst[2]{u}+\lcst{u}\big)\rqFPr{e^2_{\mbullet}}.\hfill
    \end{align} 
\end{lemma}
\begin{proof}[Proof of \cref{re:qchi}]We start our proof with the
	observation that
	$\FuEx[e_{\mbullet}]{\mu_{\mbullet}}(Q_{\Di})=\sum_{j\in\nset{\Di}}(e_j^2+\mu_{j}^2)=\qFPr{e_{\mbullet}}+\qFPr{\mu_{\mbullet}}$,
	$\Sigma_{\Di}:=\tfrac{1}{2}\sum_{j\in\nset{\Di}}\FuVar[e_{\mbullet}]{\mu_{\mbullet}}(\zOb[j]^2)=
	\sum_{j\in\nset{\Di}}e_j^2(e_j^2+2\mu_j^2)=\qFPr{e_{\mbullet}^2}+2\qFPr{\mu_{\mbullet}e_{\mbullet}}$
	and
	$\sqrt{\qFPr{e^2_{\mbullet}}}=\rqFPr{e^2_{\mbullet}}\geq\mFPr{e^2_{\mbullet}}=\max_{\nset{\Di}}(e^2_{\mbullet})$, which we use below without further reference.  Due to
	\cite{Birge2001} (Lemma 8.1)  it holds for all $x>0$
	\begin{multline*}
		\FuVg[e_{\mbullet}]{\mu_{\mbullet}}\big(Q_{\Di}-\FuEx[e_{\mbullet}]{\mu_{\mbullet}}(Q_{\Di})\geq
		2\sqrt{\Sigma_{\Di}x}+2\mFPr{e^2_{\mbullet}}x\big)\leq \exp(-x),\\
		\FuVg[e_{\mbullet}]{\mu_{\mbullet}}\big(Q_{\Di}-\FuEx[e_{\mbullet}]{\mu_{\mbullet}}(Q_{\Di})\leq -2\sqrt{\Sigma_{\Di}x}\big)\leq\exp(-x),
	\end{multline*}
	which
	for all $u\in(0,1)$ with  $\lcst{u}=\sqrt{|\log u|}$ implies    \begin{multline}\label{qchi:p1}
		\quVgQ[e_{\mbullet}]{\mu_{\mbullet},\Di}{u}\leq
		\qFPr{e_{\mbullet}}+\qFPr{\mu_{\mbullet}} + 2 \sqrt{\Sigma_{\Di}\lcst[2]{u}}
		+2\mFPr{e^2_{\mbullet}}\lcst[2]{u},\\
		\quVgQ[e_{\mbullet}]{\mu_{\mbullet},\Di}{1-u}\geq\qFPr{e_{\mbullet}}+\qFPr{\mu_{\mbullet}} - 2
		\sqrt{\Sigma_{\Di}\lcst[2]{u}}.\hfill
	\end{multline}
	For $\mu_{\mbullet}=\nSoS\in\lp^2$ we have $\qFPr{\mu_{\mbullet}}=0$ and $\Sigma_{\Di}=\qFPr{e^2_{\mbullet}}$, hence from the
	first bound in \eqref{qchi:p1} we immediately obtain 
        \eqref{qchi:e1}. 
     For $\mu_{\mbullet}\in\lp^2$  we have
	$\Sigma_{\Di}\leq\qFPr{e^2_{\mbullet}}+2\qFPr{\mu_{\mbullet}}\mFPr{e^2_{\mbullet}}$,
	and hence using $\sqrt{x+y}\leq\sqrt{x}+\sqrt{y}$ and
	$2\sqrt{xy}\leq cx+c^{-1}y$ for $x,y,c\in\pRz$ with  $c=10$ it follows
	\begin{multline*}
		2 \sqrt{\Sigma_{\Di}\lcst[2]{u}}\leq 2 \sqrt{2\qFPr{\mu_{\mbullet}}\mFPr{e^2_{\mbullet}}\lcst[2]{u}}+2 \sqrt{\qFPr{e^2_{\mbullet}} \lcst[2]{u}}\\\leq \tfrac{1}{5}\qFPr{\mu_{\mbullet}}+10\mFPr{e^2_{\mbullet}}\lcst[2]{u}+2 \sqrt{\qFPr{e^2_{\mbullet}} \lcst[2]{u}}\leq \tfrac{1}{5}\qFPr{\mu_{\mbullet}}+(10\lcst[2]{u}+2\lcst{u})\rqFPr{e^2_{\mbullet}},
	\end{multline*}
	which together with the second bound in \eqref{qchi:p1} implies 
	\eqref{qchi:e2} and 
        completes the proof.
\end{proof}

\begin{lemma}
	\label{lemma:adaptive_lower_bound}
	 For each $s \in \cS$, where $\cS$ is an arbitrary index set with  $\lv \cS \rv =
        N\in\Nz$, let $\kappa^s \in \Nz$, $\SoS^s\in\lp^2$ and $\wEvS^s
        \in \lp^\infty$.  For  the  mixing measure $\FuVg{\mu}
	:= \frac{1}{N} \sum_{s \in \cS} \frac{1}{2^{\kappa^s}} \sum_{\tau
          \in \{\pm 1\}^{\kappa^s}}
        \FuVg[\nlImS]{\wEvS^s\SoS^{s,\tau}}$ with
        $\SoS^{s,\tau}=\Nsuite[j]{\tau_j\So[j]^{s}\Ind{j\in\nset{\kappa^s}}}$
        and $\FuVg{0}:= \FuVg[\nlImS]{\nSoS}$ the $\chi^2$-divergence satisfies
	\begin{equation}\label{eq: mixing measure chi-squared}
          \chi^2(\FuVg{\mu}, \FuVg{0}) \leq \frac{1}{N^2} \sum_{s,t \in \cS} \exp\lb \frac{1}{2} \qFPr[\kappa^s \wedge \kappa^t] {\wEvS^s \SoS^s \wEvS^t \SoS^t /\snlImS}\rb - 1.	\end{equation}
\end{lemma}
\begin{proof}[Proof of \cref{lemma:adaptive_lower_bound}]
  Inspecting the calculations in the direct Gaussian sequence space
  model with coordinate-wise constant noise levels by
  \cite{Baraud2002} (proof of Theorem 1) it is readily seen that for
  any $z_{\mbullet}=\Nsuite[j]{z_j}\in\Rz^\Nz$ the likelihood ratio is
  given by
  \begin{align*}
    \frac{\dif \FuVg{\mu}}{\dif \FuVg{0}}(z_{\mbullet}) =
    \frac{1}{N} \sum_{s \in \cS} \exp\lb
    -\tfrac{1}{2}\qFPr[\kappa^s]{\wEvS^s\SoS^s/\nlImS}
    \rb \prod_{j=1}^{\kappa^s} \frac{1}{2} \lb \exp\lb - \frac{\wEv[j]^s \So[j]^s z_j}{\snlIm[j]}\rb + \exp \lb\frac{\wEv[j]^s \So[j]^s z_j}{\snlIm[j]} \rb \rb.
  \end{align*}
  Keep in mind that for $\zObS\sim\FuVg[\nlImS]{\nSoS}$ we have
  $\FuEx[\nlImS]{\nSoS}\lb\exp(a Z_j)\rb = \exp(a^2 \snlIm[j]/2)$ for
  any $j\in\Nz$ and $a\in\Rz$.  By taking the expectation of the
  squared likelihood ratio with respect to $\FuVg{0}$ we obtain
  \begin{multline*}
    \FuEx{0}\lb	\frac{\dif \FuVg{\mu}}{\dif \FuVg{0}}(z_{\mbullet}) \rb^2  
    = \frac{1}{N^2} \sum_{s,t \in \cS}
    \prod_{j=1}^{\kappa^s \wedge \kappa^t} \frac{1}{2} \lb \exp\lb -
    \frac{\wEv[j]^s \So[j]^s \wEv[j]^t \So[j]^t}{\snlIm[j]}\rb
    + \exp \lb\frac{\wEv[j]^s \So[j]^s \wEv[j]^t \So[j]^t}{\snlIm[j]} \rb \rb \\
    = \frac{1}{N^2} \sum_{s,t \in \cS}  	\prod_{j=1}^{\kappa^s \wedge \kappa^t} \cosh\lb  \frac{\wEv[j]^s \So[j]^s \wEv[j]^t \So[j]^t}{\snlIm[j]}\rb .
  \end{multline*}
  Exploiting the elementary inequality $\cosh(x) \leq \exp(x^2/2)$,
  $x\in\Rz$, and the definition of the $\chi^2$-divergence,  we obtain
  \eqref{eq: mixing measure chi-squared}, which completes the proof.
\end{proof}


\bibliography{lit.bib}
\end{document}